\documentclass[12pt,oneside]{amsart}

\usepackage{amssymb} 
\usepackage{enumerate, amsfonts, latexsym,epsfig, graphicx,color}
\usepackage{epstopdf, amscd}
\usepackage{hyperref}
\usepackage{tikz-cd}
\usepackage{mathptmx}

\DeclareFontFamily{U}{mathx}{\hyphenchar\font45}
\DeclareFontShape{U}{mathx}{m}{n}{
      <5> <6> <7> <8> <9> <10>
      <10.95> <12> <14.4> <17.28> <20.74> <24.88>
      mathx10
      }{}
\DeclareSymbolFont{mathx}{U}{mathx}{m}{n}
\DeclareFontSubstitution{U}{mathx}{m}{n}
\DeclareMathAccent{\widecheck}{0}{mathx}{"71}

\DeclareFontFamily{U}{mathb}{\hyphenchar\font45}
\DeclareFontShape{U}{mathb}{m}{n}{
      <5> <6> <7> <8> <9> <10> gen * mathb
      <10.95> mathb10 <12> <14.4> <17.28> <20.74> <24.88> mathb12
      }{}
\DeclareSymbolFont{mathb}{U}{mathb}{m}{n}
\DeclareFontSubstitution{U}{mathb}{m}{n}
\DeclareMathSymbol{\bigast}{1}{mathb}{"06}

\input xy
\xyoption{all}

\newcommand{\llangle}{\langle\negthinspace\langle}
\newcommand{\rrangle}{\rangle\negthinspace\rangle}
 
\usepackage{amsmath}

\newcommand{\leftQ}[2]{\left.\raisebox{-.2em}{$#2$}\middle\backslash\raisebox{.2em}{$#1$}\right.}

\newtheorem {theorem}{Theorem}[section]
\newtheorem {lemma} [theorem] {Lemma}
\newtheorem {proposition} [theorem] {Proposition}
\newtheorem {example} [theorem] {Example}
\newtheorem {corollary} [theorem] {Corollary}

\newtheorem {conjecture}[theorem] {Conjecture}

\newtheorem {notation} [theorem] {Notation}

\newtheorem {construction} [theorem] {Construction}

\newtheorem {convention} [theorem] {Convention}
\newtheorem {assumption} [theorem] {Assumption}

\newtheorem{thmA}{Theorem}

\newtheorem{corA}[thmA]{Corollary}

\theoremstyle{definition}
\newtheorem {definition} [theorem] {Definition}
\newtheorem {remark} [theorem] {Remark}
\newtheorem {claim}{Claim}[theorem]
\newtheorem*{claim*}{Claim}

\newcommand{\CalA}{\mathcal{A}}

\newcommand{\co}{\colon\thinspace}

\def\ssm {\smallsetminus}
\def\mc {\mathcal}
\def\into {\hookrightarrow}

\def\R {\mathbb R}

\def\Z {\mathbb Z}

\def\Stab {\mathrm{Stab}}

\def\Isom {\mathrm{Isom}}

\def\H {\mathbb H}

\newcommand{\CS}{\mathcal{S}} 
 \newcommand{\CTC}{\mathcal{C}} 
\newcommand{\SCS}{\mc{S}} 
\newcommand{\SCTC}{\mc{C}} 

\newcommand{\Xz}{X_0} 
\newcommand{\gammax}{\gamma_0} 
\newcommand{\gammay}{\gamma} 

\newcommand{\Yx}{Y_0} 

\newcommand{\horba}{\mc{H}} 
\newcommand{\horbaD}{\mc{D}} 

\newcommand{\doub}{N_{\mathsf{doub}}} 

\newcommand{\sys}{\operatorname{sys}} 
\newcommand{\sep}{\chi} 

\newcommand{\ray}[1]{r(#1)} 

\newcommand{\Path}{\mathbb{P}} 

\newcommand{\UG}{\mathrm{UG}} 

\newcommand{\wchXz}{{\widecheck X}_0}
\newcommand{\wchY}{\widecheck Y}

\begin{document}

\title{Drilling hyperbolic groups}
\author[D.Groves]{Daniel Groves}
\address{Department of Mathematics, Statistics, and Computer Science,
University of Illinois at Chicago,
322 Science and Engineering Offices (M/C 249),
851 S. Morgan St.,
Chicago, IL 60607-7045}
\email{dgroves@uic.edu}
\author[P. Ha\"issinsky]{Peter Ha\"issinsky}
\address{Aix-Marseille Univ., CNRS, I2M, UMR 7373, Marseille, France}
\email{phaissin@math.cnrs.fr}
\author[J.F. Manning]{Jason F. Manning}
\address{Department of Mathematics, Cornell University, Ithaca, NY 14853, USA
}
\email{jfmanning@cornell.edu}
\author{Damian Osajda}
    \address{Department of Mathematical Sciences,
		University of Copenhagen, Nør\-re\-gade 10, 1165 København, Denmark; 
        Mathematical Institute, University of Wroc{\l}aw, pl.\ Grunwaldzki 2/4, 50–384 Wroc{\l}aw, Poland}
    \email{dosaj@math.uni.wroc.pl}
\author{Alessandro Sisto}
    \address{Department of Mathematics, Heriot-Watt University and Maxwell Institute for Mathematical Sciences, Edinburgh, UK}
    \email{a.sisto@hw.ac.uk}

\author{Genevieve S. Walsh}
\address{Department of Math, Tufts University, Medford, MA 02155, USA} 
\email{genevieve.walsh@tufts.edu}
\date{\today}

\begin{abstract}Given a hyperbolic group $G$ and a maximal infinite cyclic subgroup $\langle g \rangle$, we define a {\it drilling of $G$ along $g$}, which is a relatively hyperbolic group pair $(\widehat{G}, P)$.   This is inspired by the well-studied procedure of drilling a hyperbolic $3$--manifold along an embedded geodesic.  We prove that, under suitable conditions, a hyperbolic group with $2$-sphere  boundary admits a drilling where the resulting relatively hyperbolic group pair $(\widehat{G}, P)$ has relatively hyperbolic boundary $S^2$. This allows us to reduce the Cannon Conjecture (in the residually finite case) to a relative version, which is likely to be more tractable.
\end{abstract}
\thanks{DG was partially supported by a Simons Fellowship, \#1034629 to Daniel Groves and the National Science Foundation, DMS-1904913 and DMS-2203343. PH is partially supported by ANR-22-CE40-0004 GoFR. JM was partially supported by the Simons foundation, grants number 942496 and 524176 and by the National Science Foundation under grant DMS-1462263.  DO was supported by the Carlsberg Foundation, grant CF23-1226. GW was partially supported by NSF Grants DMS-2005353 and DMS-1709964.}

\maketitle
\setcounter{tocdepth}{1}
\tableofcontents

\section{Introduction}

An important feature in low dimensions is the imposition of geometry by topology.
This is illustrated by the classification of low-dimensional 
manifolds provided by the Koebe-Poincar\'e uniformisation of surfaces and the Thurston-Perel'man uniformisation of $3$-manifolds. Even more, the knowledge of the fundamental group of a manifold narrows down the possible  topology of the given manifold and the geometry it may be equipped with. Starting from dimension $4$, this aspect completely breaks down as Dehn showed that any finitely presented group appeared as the fundamental group of a manifold of any dimension at least $4$. 

Despite the important body of knowledge that has been acquired these last decades on $3$-manifolds and their fundamental  groups, including most of Thurston's program \cite{otal:thurston_prog}, there remains an important open problem known as the Cannon conjecture that would provide a dynamical characterisation of uniform lattices of
$PSL_2(\mathbb{C})$ and that would complete this principle that topology determines geometry in low dimension:

\begin{conjecture}[Cannon Conjecture] \label{conj:Cannon}
Let $G$ be a word hyperbolic group whose Gromov boundary is homeomorphic to $S^2$. Then $G$ is virtually Klei\-nian.
\end{conjecture} 

The Cannon Conjecture was introduced by Cannon in \cite{Cannon91} as a step towards the global hyperbolisation conjecture, see also \cite[Conjecture 6.2]{MartinSkora} for a broader conjecture.  In the torsion-free case, the Cannon Conjecture is a special case of a question of Wall \cite[Question G2, p.391]{Hom-Gp-Thy}. In dimension 3, this question asks whether every finitely generated $PD(3)$--group is the fundamental group of a closed aspherical $3$--manifold.   These conjectures are connected by work of Bestvina \cite[Theorem 2.8, Remark 2.9]{Bestvina:local_homology_boundary}, who proved (building on work of Bestvina--Mess \cite{BesMes}) that a  torsion-free hyperbolic group is a $\mathrm{PD}(3)$ group exactly when it has Gromov boundary homeomorphic to $S^2$. 

Many approaches to the Cannon Conjecture have been attempted, but so far without success.
One interpretation of the difficulty comes from
the assumption on the boundary: the sphere is so round that one does not know how to handle it. One consequence of that is the fact that all uniform
lattices are quasi-isometric to one another. On the contrary, non-uniform lattices ---that are canonically relatively hyperbolic groups with $2$-sphere Bowditch boundary--- are quasi-isometric if and only if they are commensurable \cite{Schwartz}. This indicates that it is more tractable to work with a relative version of the Cannon conjecture.  

\begin{conjecture}[Toral Relative Cannon Conjecture]  \label{conj:relativecannon}
Let $(G,P)$ be a relatively hyperbolic group pair whose Bowditch boundary is homeomorphic to $S^2$, $P \neq \emptyset$, and where every element of $P$ is $\mathbb{Z}^2$. Then $G$ is Kleinian.
\end{conjecture}

A group satisfying the hypotheses and conclusion of Conjecture~\ref{conj:Cannon} will be virtually the fundamental group of a closed hyperbolic $3$--manifold.  On the other hand, a group satisfying the hypotheses and conclusion of Conjecture~\ref{conj:relativecannon} will be virtually the fundamental group of a finite-volume, non-compact hyperbolic $3$--manifold.

 A group pair satisfying the hypotheses of the
Toral Relative Cannon Conjecture  is \emph{Haken} in the
sense of \cite[p.282]{KK_coarse_alexander_duality}; this property is an analogue of the
Haken property for $3$-manifolds. The latter enables one to cut a
manifold along incompressible surfaces into finitely many balls yielding
a so-called \emph{hierarchy} that allows proofs based on inductive arguments.
Many important breakthroughs, such as the geometrization and the
cubulation of atoroidal $3$--manifolds, were first solved in the Haken
setting by this procedure. In the same spirit, one may infer that the
toral parabolic subgroups could be used to split the group inductively
into simpler groups, see \cite[p.282]{KK_coarse_alexander_duality}. 

Further evidence for our belief that Conjecture~\ref{conj:relativecannon} is easier than
Conjecture~\ref{conj:Cannon} is the result of the odd-numbered authors of this paper that the Cannon Conjecture \emph{implies} the Toral Relative Cannon Conjecture \cite[Corollary 1.2]{GMS}.  
 Our work implies the following converse in
the residually finite case.

\begin{thmA}\label{thm:RelCannon}
  Suppose that the Toral Relative Cannon Conjecture is true, and suppose that $G$ is a residually finite hyperbolic group for which $\partial G \cong S^2$.  Then $G$ is virtually Kleinian.
\end{thmA}
Theorem \ref{thm:RelCannon} follows from our main technical result, Theorem \ref{t:Drill} below. We present a short proof of that  just after Remark \ref{rmk:sufficient}. 
Our statement and proof of Theorem \ref{t:Drill} is inspired by an important tool in the study of hyperbolic $3$--manifolds, namely the operation of drilling out an embedded geodesic.  It follows from Thurston's geometrization theorem that  if $M$ is a closed hyperbolic $3$--manifold and $k$ is an embedded geodesic loop then the complement $M \smallsetminus k$ admits a complete hyperbolic metric of finite volume (see~\cite{Sakai,Kojima}).  The operation of drilling geodesics has been used in the proofs of some of the most important theorems about hyperbolic $3$--manifolds.  Some examples are Canary's proof that topologically tame hyperbolic $3$--manifolds are geometrically tame \cite{Canary:Ends}, the proof of the Bers Density Theorem by Brock--Bromberg in \cite{BrockBromberg:Density}, and Calegari--Gabai's proof of Tameness using Shrinkwrapping in \cite{CalegariGabai:Shrinkwrapping}, among many others. We introduce in this paper a notion of `drilling' in the group-theoretic setting.  Here is what we mean by this.

\begin{definition} \label{def:drill}
Suppose that $G$ is a hyperbolic group, and that $g$ is an element so that $\langle g \rangle$ is a maximal cyclic subgroup.  A {\em drilling of $G$ along $g$} is a relatively hyperbolic pair $(\widehat{G},P)$ along with a normal subgroup $N \unlhd P$ with an identification $P/N \cong \langle g \rangle$ so that $\widehat{G} / \llangle N \rrangle \cong G$, with the quotient map inducing the identification of $P/N$ with $\langle g \rangle$.
\end{definition}

Let us illustrate this definition in the classic setting of $3$-manifolds.
Let $M$ be a closed hyperbolic $3$--manifold and $k$ an embedded geodesic loop.  The fundamental group $\pi_1(M)$ is a (Gromov) hyperbolic group since it acts geometrically on $\H^3$. The Gromov boundary of $\pi_1(M)$ is $S^2$.  The pair $(\pi_1(M\smallsetminus k), \mathbb{Z}^2)$ is a relatively hyperbolic pair, and the relatively hyperbolic boundary $\partial(\pi_1(M\smallsetminus k), \mathbb{Z}^2)$ is $S^2$, since $\pi_1(M \smallsetminus k)$ acts geometrically finitely on $\H^3$.  The parabolic subgroups in this action are the conjugates of $\mathbb{Z}^2 = \pi_1(T^2)$, where $T^2$ is the peripheral torus.  Let $m \in \pi_1(M \smallsetminus k)$ represent a meridian, i.e.\ a curve bounding a disk in $M$ which is punctured exactly once by $k$.  The quotient of $\pi_1(M \smallsetminus k)$ by the normal closure of $m$ is $\pi_1(M)$.

 The main technical result of this paper is that given  a hyperbolic group with boundary $S^2$, there are often drillings in the sense of Definition \ref{def:drill}  that are consistent with the above description of drilling a closed hyperbolic 3-manifold along an embedded geodesic. 
 
\begin{thmA} \label{t:Drill}
 Let $X$ be a hyperbolic graph with boundary $S^2$. For every hyperbolic isometry $g$ of $X$, with ($g$--invariant, quasi-geodesic) axis $\ell$, there exists $\Sigma>0$ so that the following holds:

 Suppose $G$ is a subgroup of $\Isom(X)$ so that
 \begin{enumerate}
     \item $G$ acts freely and cocompactly on $X$; and
     \item for any $h \in G \smallsetminus \langle g\rangle$ the axes $\ell$ and $h \cdot \ell$ are at least $\Sigma$ apart.
 \end{enumerate}
 Then there exists a drilling $(\widehat{G},P)$ of $G$ along $g$ which satisfies the following properties:
\begin{enumerate}
\item\label{item:is a drilling} $(\widehat{G},P)$ is relatively hyperbolic with $P$ either free abelian of rank $2$ or the fundamental group of a Klein bottle;
\item\label{item:boundary S^2} The Bowditch boundary of $(\widehat{G},P)$ is a $2$--sphere.
\end{enumerate}
The peripheral subgroup $P$ is free abelian if and only if $g$ preserves an orientation on $\partial X$.

Finally, $\widehat{G}$ is torsion-free if and only if $G$ is torsion-free.
\end{thmA}

\begin{remark} \label{rmk:sufficient} 
    In Theorem \ref{t:Drill} we give a sufficient \emph{geometric} condition for drilling.  In \cite[Section 5, p.79]{CalegariFujiwara:scl}, Calegari and Fujiwara suggest that the condition of $g$ having very small \emph{stable commutator length} might be sufficient to perform some kind of drilling.  We do not know if such a condition on the stable commutator length is sufficient, but believe it is an interesting question.
\end{remark}

With Theorem \ref{t:Drill} in hand, we obtain the following proof of Theorem \ref{thm:RelCannon}.

\begin{proof}
Let $X$ be the Cayley graph of $G$ with respect to some generating set.  Since $G$ is residually finite, there is a subgroup $G_0$ of finite index which is torsion-free and preserves an orientation on $\partial G$.

Let $g \in G_0$ be so that $\langle g \rangle$ is maximal cyclic.  Since maximal abelian subgroups of residually finite groups are separable (see, for example, \cite[Proposition, p.484]{Long87}), there is a finite index subgroup $G_1$ of $G_0$ so that $g \in G_1$ and so that the axis of $g$ in $X$ satisfies the hypotheses of Theorem \ref{t:Drill} with respect to the $G_1$--action on $X$.

Therefore, there exists a drilling $\left( \widehat{G_1}, P \right)$ of $G_1$ along $g$ so that $\left(\widehat{G_1}, P \right)$ is relatively hyperbolic with Bowditch boundary $S^2$, $P$ free abelian of rank $2$, and an element $k \in P$ so that
\[	G_1 \cong \widehat{G_1} / \llangle k \rrangle	,	\]
which identifies $\langle g \rangle$ with $P / \langle k \rangle$.

Since $G_1$ is torsion-free, $\widehat{G_1}$ is also torsion-free, by Theorem~\ref{t:Drill}.  According to the Toral Relative Cannon Conjecture, $\widehat{G_1}$ is Kleinian. Suppose that $M$ is a $1$--cusped finite-volume hyperbolic $3$--manifold such that $\pi_1(M) = \widehat{G_1}$.  The group $G_1$ can be obtained as the fundamental group of a Dehn filling of $M$ which implies that $G_1$ is the fundamental group of a closed $3$--manifold.  Since $G_1$ is word-hyperbolic, and $\partial G_1 \cong S^2$, it follows from the Geometrization Theorem that $G_1$ is Kleinian.  This completes the proof.
\end{proof}

The first three paragraphs of the proof of Theorem~\ref{thm:RelCannon} prove the following unconditional statement.
\begin{corA}
  If $G$ is a residually finite hyperbolic group with $\partial G\cong S^2$, then $G$ virtually admits a drilling as in the statement of Theorem~\ref{t:Drill}.
\end{corA}

Our approach to proving Theorem~\ref{t:Drill} is to do ``medium-scale geometric topology''.  Our spaces are graphs, which of course do not have much useful local structure.  However, using the Gromov hyperbolic geometry of $X$, we identify certain large (but fixed) scales where a neighborhood of the axis $\ell$ ``looks like" a $3$--dimensional solid tube $\R \times D^2$, and the ``frontier'' of this neighborhood ``looks like'' an open annulus $\R \times S^1$.  Much of Section~\ref{s:shells} is devoted to making  this idea precise in a useful way.

\medskip

Theorem~\ref{t:Drill} also has the following consequence for PD$(3)$ groups and pairs.

\begin{corA}\label{cor:Wall}
    Suppose that $G, X, g$, and $\ell$ satisfy the hypotheses of Theorem~\ref{t:Drill}, and furthermore suppose that $G$ is a PD$(3)$--group (i.e.\ suppose that $G$ is torsion-free).  Let $(\widehat{G},P)$ be the drilling of $G$ along $g$ from the conclusion of Theorem~\ref{t:Drill}.  Then $(\widehat{G},P)$ is a PD$(3)$--pair.
\end{corA}
\begin{proof}
    A PD$(3)$--group has cohomological dimension $3$, so is torsion-free.  By the final conclusion of Theorem~\ref{t:Drill}, $\widehat{G}$ is also torsion-free.  That $(\widehat{G},P)$ is a PD$(3)$--pair follows from \cite[Corollary 1.3]{TshishikuWalsh} (see also \cite[Corollary 4.3]{ManningWang}).
\end{proof}

\bigskip 

As drilling plays an important role in $3$-dimensional topology, we believe that drilling groups will also become a very useful tool for analysing groups. 
Moreover, its proof develops many different tools that are interesting on their own right that go from coarse topology to metric geometry and hyperbolic geometry.
Here are some examples of (group-theoretic) drilling that show the wide potential range of applications we may expect.

\begin{example} \label{ex:3-man drill}
    Let $M$ be a closed hyperbolic $3$--manifold, and let $k$ be an embedded geodesic loop.  Let $M^\circ = M \smallsetminus k$.  If $G = \pi_1(M)$, $g$ is the element represented by $k$, $\widehat{G} = \pi_1(M^\circ)$, and $P$ is the subgroup of $\pi_1(M^\circ)$ corresponding to a regular neighborhood of $k$, then $(\widehat{G},P)$ is a drilling of $G$ along $g$.
\end{example}

\begin{example} \label{ex:trivial drill}
    Let $H$ be a hyperbolic group and suppose that $h \in H$ generates an infinite maximal elementary subgroup.  Let $T = \langle a,b |[a,b]=1 \rangle$.  Define $\widehat{H} = H \ast_{h=a} T$, and note that $(\widehat{H},T)$ is a drilling of $H$ along $h$.  We refer to this drilling as the \emph{trivial drilling}.
\end{example}

\begin{example} \label{ex:free and surface drill}
    Let $F$ be the fundamental group of a $3$--dimensional handlebody, and $f \in F$ a diskbusting curve.  We can realize $f$ as an embedded geodesic in some negatively curved metric on the handlebody, and remove a neighborhood of this curve to obtain a drilling of a free group.  It is possible to choose such a drilling to be nontrivial.

Similarly, one can perform a nontrivial drilling of a surface group $G = \pi_1(S)$ by realizing $G$ as the fundamental group of some quasi-fuchsian $S \times I$, and drilling a simple geodesic $\gamma$ out of $S \times I$.  This can give a trivial drilling, but the drilling is nontrivial if $\gamma$ is non-simple when realized as a geodesic in a hyperbolic metric on $S$ itself. \end{example}

\begin{example} \label{ex:undo filling}
    Let $(\Gamma,P)$ be a relatively hyperbolic group, and suppose that $N \unlhd P$ is such that $P / N \cong \Z$.  It follows from the relatively hyperbolic Dehn filling theorem \cite{osin:peripheral} (see also \cite{rhds}) that $\overline{\Gamma} = \Gamma / \llangle N \rrangle$ is often hyperbolic, and that the image of $P$ in $\overline{\Gamma}$ is maximal elementary.  Thus $(\Gamma,P)$ is a drilling of $\overline{\Gamma}$.  
\end{example}
It is notable that the remarkable $5$--dimensional hyperbolic groups constructed by Italiano, Martelli, and Migliorini in \cite{IMM} fall into the class described in Example~\ref{ex:undo filling}, and so admit (nontrivial) drillings.  Moreover, there are finite-volume hyperbolic manifolds of all dimensions at least $3$ which have quotients as in Example~\ref{ex:undo filling}, so it follows from Examples~\ref{ex:free and surface drill} and~\ref{ex:undo filling} that there are hyperbolic groups of all positive dimensions admitting nontrivial drillings.

\begin{remark}
    Our condition that $P/N \cong \langle g \rangle$ is designed to mimic what happens on the level of the fundamental group when a geodesic is drilled.  There are situations where it is interesting to remove other geometrically controlled sets (for example removing totally geodesic codimension--$2$ sub-manifolds from negatively curved manifolds of high dimension as in \cite{Bel12a,Bel12b}).  Many of our basic tools work in such a setting, and we phrase our results to apply in these settings as much as possible. As an example, we highlight Theorem~\ref{t:Pi_isomorphism}, which shows that, given an appropriate hyperbolic space $X$ and quasi-convex subset $Y$, the coarse fundamental group of the ``completed shell" around $Y$ in $X$ is isomorphic to $\pi_1(\partial X \smallsetminus \Lambda Y)$.   However, we do not pursue these ideas in this paper.
\end{remark}

 Section \ref{sec:outline} gives a general outline of the proof of  Theorem~\ref{t:Drill}, while general background on relatively hyperbolic groups and their boundaries is contained in Section \ref{sec:background}. In Section \ref{s:coarse_topology}, we review and develop some notions of coarse topology, notably the coarse fundamental group, coarse Cartan-Hadamard, and coarse deformation retractions.  Section \ref{sec:linear_connectedness} defines a condition we call `spherical connectivity', which gives a local (in terms of bounded sized balls in a hyperbolic space) criterion for verifying that a boundary is linearly connected.  We also relate linear connectivity of the boundary for visual metrics from varying basepoints to cut points on the boundary.  In Section~\ref{s:shells}, we study the geometry of large tubes around geodesics, and maps from the boundary of a hyperbolic space to the ``boundary" of a neighborhood of a convex subset. We show how to unwrap the complement of a tube about a geodesic (take a coarse infinite cyclic cover) in Section \ref{sec:unwrap}, and develop tools for studying the resulting space, and also the ``unwrapped" space with a horoball attached to the pre-image of the boundary of the tube.  In Section \ref{s:unwrap family}, we show how to unwrap a whole family of these tubes. In Section \ref{sec:limitsofspheres}, we give a condition for a weak Gromov-Hausdorff limit to be $S^2$, which is used in Section~\ref{sec:proof_of_main_theorem} to prove that the boundary of the completely unwrapped space (the limit of the partially unwrapped spaces) is $S^2$.  In Section~\ref{sec:proof_of_main_theorem} we also prove the statement about torsion in Theorem~\ref{t:Drill} completing the proof.  In Section~\ref{sec:CL} we prove that the quotient map from the drilled group to the original group is a ``long filling'' in a reasonable sense.
 In Section~\ref{sec:notation}, we give a guide to some of our notation and constructions.

\subsection*{Acknowledgments}
This project was started at the American Institute for Mathematics (AIM) workshop “Boundaries of groups”, and the bulk of the work was also carried out at AIM in the context of the SQuaRE program. We thank AIM for the very generous support. We thank C. G. Hruska and L. Paoluzzi in particular for early conversations about this work.

\section{Outline of proof of Theorem \ref{t:Drill}} \label{sec:outline}
Our proof of Theorem~\ref{t:Drill} is constructive.  We start with a hyperbolic graph $X_0$ on which $G$ acts geometrically and such that $\partial X_0$ is $S^2$.  We are given a $g$--invariant axis $\gammax$ whose $G$--translates are far apart.   We want to build a space $\widehat Y$ which is hyperbolic, $\partial \widehat Y = S^2$, and such that the drilled group  $(\widehat G, P)$ acts on $\widehat Y$, and this action is cusp-uniform with $P$ as the stabilizer of a horoball.  This implies that $(\widehat{G}, P)$ is relatively hyperbolic, and that the Bowditch boundary is $S^2$. Roughly speaking, we will build $\widehat Y$ as a limit of (uniformly) hyperbolic spaces, and the boundary of $\widehat Y$ will be a limit, in a suitable sense, of the boundaries of the approximating spaces, which will all be homeomorphic to $S^2$.

Classically speaking ``drilling" involves removing a regular neighborhood of a loop in a three-manifold.  Since we are starting with a graph, we must take a more coarse approach.
For example, for some scale $D$ one can define the $D$--fundamental group of a graph to be the fundamental group of the space obtained by gluing on a disk to every loop of length at most $D$, see Section \ref{s:coarse_topology}.  A \emph{$D$--cover} of a space is then a cover to which every loop of length at most $D$ lifts.  
The precise scale $D$ at which we work is fixed in Subsection \ref{ss:noname}. 

We use this coarse theory to understand the connection between a shell around our quasi-geodesic axis $\gammax$ and the boundary $S^2 \ssm \Lambda \gammax$ (where $\Lambda \gammax$ denotes the limit set of $\gammax$ in $\partial \Xz = S^2$).
  For a large number $K$, let $S$ denote the set of points at distance $K$ from $\gammax$.  We prove that $S$ is coarsely connected (Lemma~\ref{lem:sphere_coarse_conn}), and that after $S$ is ``completed" to be a connected graph, $\pi_1^D(S)$ is isomorphic to $\mathbb{Z}$.  

More precisely, there is a projection from $\partial X \smallsetminus \Lambda\gammax$ to $S$ which induces an isomorphism between $\pi_1\left(\partial X \smallsetminus \Lambda\gammax \right)$ and $\pi_1^D(S)$.
See Theorem~\ref{t:Pi_isomorphism} and Corollary~\ref{cor:pi_1_Z} for precise statements.

Moreover, if $T$ is the collection of points in $X$ at distance at most $K$ from $\gamma_0$ then $X \smallsetminus T$ coarsely deformation retracts (see Definition~\ref{def:def retract}) onto $S$ (see Lemma~\ref{lem:cdr}).  This implies their coarse fundamental groups are  canonically isomorphic (see Lemma~\ref{prop:def retract on pi_1}).

In Section~\ref{sec:unwrap}, we fix a number $R_0$ large enough that the preceding results hold for $K = R_0$, and additionally so that every loop representing a nontrivial element of $\pi_1^D(S)$ is very long.

The $D$--universal cover $\widetilde{S}$ of $S$ is a coarse plane, upon which a group $E$ (either $\mathbb{Z}^2$ or the fundamental group of a Klein bottle) acts geometrically, see Lemma~\ref{lem:pi_CS/g}.  The $D$--universal cover of $X \smallsetminus T$ contains a copy of $\widetilde{S}$.  The space obtained by gluing a combinatorial horoball onto the copy of $\widetilde{S}$ is Gromov hyperbolic, and has boundary a $2$--sphere (see Proposition~\ref{prop:hatY_modeled} for hyperbolicity and Corollary~\ref{cor:UG0_S2} for the fact that the boundary is a $2$--sphere).

We then proceed in Section~\ref{s:unwrap family} to inductively build spaces by unwrapping more and more lifts of complements of tubes around translates of $\gammax$ and gluing horoballs onto the resulting coarse planes.  These spaces are all Gromov hyperbolic for a uniform hyperbolicity constant (this uses Lemma~\ref{lem:app_of_CCH}), and the visual metrics on their boundaries are uniformly linearly connected (this uses Corollary~\ref{cor:modeled_on_hatX_LC}.)  
In Proposition~\ref{prop:unwrap_S^2} we show all these boundaries are homeomorphic to $S^2$.

We thus obtain a sequence of based finite-valence graphs, and for any given $N$ the $N$--ball about the basepoint stabilizes in this sequence.  The limiting graph $\widehat{Y}$ is also Gromov hyperbolic (see Corollary~\ref{cor:Yhat_hyp}), and admits a geometrically finite action of a group $\widehat{G}$, which is hyperbolic relative to (a natural copy of) the group $E$, and is the required drilling of $G$ along $\langle g \rangle$.  This is shown in Theorem~\ref{GPrelhyp}.

Realizing the boundary of $\widehat{Y}$ as a weak Gromov--Hausdorff limit (Definition~\ref{def:weakGromovHausdorff}) of the boundary two-spheres of the partially unwrapped and glued spaces, we use the characterization of the two-sphere from Theorem~\ref{thm:limit_sphere} to prove that $\partial \widehat{Y}$ is a two-sphere.  Finally, we prove the statement about torsion in $G$ and $\widehat{G}$ in Proposition~\ref{prop:torsion}, thus completing the proof of Theorem~\ref{t:Drill} (see Section~\ref{sec:proof_of_main_theorem}).

\section{Background}
\label{sec:background} 

For background on hyperbolic spaces and groups and their Gromov boundaries we refer to \cite[III.H]{BH}.
\subsection{Relative hyperbolicity}
There are a number of equivalent definitions of relative hyperbolicity in the literature; see \cite{Hru-relqconv} for a survey.  It is most convenient for us to use Gromov's original definition \cite[8.6]{Gro87}.  In this definition $(G,\mc{P})$ is relatively hyperbolic if $G$ has a \emph{cusp uniform} action on a hyperbolic space with $\mc{P}$ a choice of conjugacy classes of horoball stabilizers.  We spell out what this means in the next few paragraphs.
\begin{definition}[Busemann function, horofunction]\cite[8.1]{GdlH}
  Let $\Upsilon$ be a Gromov hyperbolic space, and let $p\in \partial \Upsilon$.
  For $x\in \Upsilon$ and $\gamma: [0,\infty) \to \Upsilon$ a geodesic ray tending to $p$, define
  \[ h_\gamma(x) = \limsup_{t\to\infty}\left(d_\Upsilon(x,\gamma(t)) - t\right).\]
  The function $h_\gamma: X\to \R$ is a \emph{horofunction} based at $p$.  
  
  We obtain a function $\beta_p: \Upsilon\times\Upsilon\to \R$ by
  \[ \beta_p(x,y) = \sup\left\{h_\gamma(x)\mid \gamma(0) = y \right\} \]
  This is called a \emph{Busemann function based at $p$}.  
\end{definition}
Roughly speaking $\beta_p(x,y)$ measures how much ``farther'' $x$ is from $p$ as compared to $y$.  Note that in some sources ``Busemann function'' refers to what we are here calling a horofunction
(see for example \cite[III.H.3.4]{BH}).
\begin{definition}[Horoball]\label{def:horoball}
  Let $\Upsilon$ be a Gromov hyperbolic space, and let $p\in \partial \Upsilon$.
  A \emph{horoball centered at $p$} is any set $\mc{H}$, so that there is a horofunction $h$ based at $p$ and real numbers $a\le b$, so that
 \[ h^{-1}\left( (-\infty,a]\right) \subset \mc{H} \subset h^{-1}\left( (-\infty,b]\right). \]
\end{definition}

\begin{definition}\label{def:gromov RH}
  Let $H$ be a group, and $\mc{P}$ a finite family of subgroups of $H$.  Suppose that $H$ acts properly on a hyperbolic space $\Upsilon$,
  preserving a family of disjoint open horoballs, and acting cocompactly on the complement of the union of that family.  Suppose further that there is a collection of orbit representatives of horoballs in the family so that $\mc{P}$ is the collection of their stabilizers.  Then the pair $(H,\mc{P})$ is said to be \emph{relatively hyperbolic}.  If $\mc{P} = \{ P \}$ consists of a single subgroup, we say that $(H,P)$ is relatively hyperbolic.
\end{definition}
There may be many spaces $\Upsilon$ which certify the relative hyperbolicity of a pair $(H,\mc{P})$, but they all have $H$--equivariantly homeomorphic boundary at infinity, by a result of Bowditch \cite[Theorem 9.4]{bowditch12}.  Accordingly, $\partial \Upsilon$ is referred to as the \emph{Bowditch boundary} of the pair $(H,\mc{P})$.  

\subsection{Visibility}

Recall the following definition and result from \cite{GMS}.

\begin{definition}\label{defn:visiblespace} \cite[Definition 3.4]{GMS}
  Let $\mu>0$.  A geodesic metric space $Z$ is \emph{$\mu$--visible} if, for every $a,b \in Z$, there is a geodesic ray based at $a$ passing within $\mu$ of $b$. 
\end{definition}

\begin{proposition}\label{lem:local_visibility} \cite[Proposition 6.3]{GMS}
 For every $\nu\geq 1$ the following holds. Let $\Upsilon$ be a proper $\nu$--hyperbolic space and suppose that for 
 all $p,q \in \Upsilon$ with $d(p,q)\leq 100\nu$ there exists a geodesic $[p,q']$ of length at least $200\nu$ with $d(q,[p,q'])\leq \nu$. Then $\Upsilon$ is $5\nu$--visible.
\end{proposition}

\begin{remark}
    Some sources say ``almost geodesically complete'' or ``almost extendible'' instead of ``visible'' (see for example \cite{Ontaneda}).  The fact that hyperbolic groups are visible is due to Mihalik.  See \cite[Lemma 3.1]{BesMes} for a proof, which extends easily to any proper hyperbolic space $\Upsilon$ whose isometry group is co-compact.  (In fact it is enough that the action on the space of pairs of distinct points in $\partial \Upsilon$ is co-compact.)
\end{remark}

\subsection{Visual metrics}
Here we recall some notions and facts concerning metrics on the Gromov boundary. For details see e.g.\ \cite[Sec.\ 7.3]{GdlH} or \cite[Chap.\ III.H]{BH}.

\begin{definition}
Let $\Upsilon$ be a Gromov hyperbolic space.
Suppose that $x,y,w \in \Upsilon$.  The \emph{Gromov product} of $x$ and $y$ with respect to $w$ is
\[  \left( x \mid y \right)_w = \frac{1}{2} \left( d_\Upsilon(x,w) + d_\Upsilon(y,w) - d_\Upsilon(x,y) \right)  .   \]

Suppose that $\alpha,\beta$ are geodesic rays with common basepoint $w$.  The \emph{Gromov product} of $\alpha$ and $\beta$ with respect to $w$ is
\[	\left( \alpha \mid \beta \right)_w = \liminf_{s,t\to \infty} \left( \alpha(s) \mid \beta(t) \right)_w.	\] 
\end{definition}

\begin{definition} Let $\Upsilon$ be Gromov hyperbolic space, with basepoint $w$.  A {\em visual metric on $\partial \Upsilon$, based at $w$, with parameters $\epsilon, \kappa$} is a metric $\rho(\cdot, \cdot)$ which is $\kappa$--bi-Lipschitz to $e^{-\epsilon (\cdot \mid \cdot)_w}$. \end{definition} 

\begin{proposition}\label{prop:ek} \cite[III.H.3.21]{BH} Let $\delta >0$. For any positive $\epsilon \leq \frac{1}{6 \delta}$, and $\kappa(\epsilon,\delta) = \left(3-2 e^{2\delta\epsilon}\right)^{-1}$ we have the following:

If $\Upsilon$ is a $\delta$--hyperbolic space and $w \in \Upsilon$, then $\partial \Upsilon$ has a visual metric based at $w$ with parameters $\epsilon$, $\kappa$. 
\end{proposition} 

\begin{definition}\label{def:delta_adapted} A \emph{$\delta$--adapted visual metric at $w$} is a visual metric as above where $\epsilon(\delta) = \frac{1}{6\delta}$ and $\kappa = \left(3-2e^{1/3}\right)^{-1} = 4.78984\ldots$ are determined as in Proposition \ref{prop:ek}.  If the basepoint is understood or unimportant, we just say that the metric is \emph{$\delta$--adapted}.
\end{definition} 

\subsection{Strong convergence and weak Gromov-Hausdorff convergence}
The use of the following two types of convergence for metric spaces is important for this paper. For this we mostly follow \cite{GMS}.

\begin{definition}\label{def:weakGromovHausdorff}
	Let $(Z_i)_{i\in\mathbb N}$ and $Z$ be metric spaces. We say that $Z$ is a \emph{weak Gromov--Hausdorff limit} of the sequence $(Z_i)$ if there exists $\lambda\geq 1$ and a sequence of $(\lambda,\epsilon_i)$--quasi-isometries $Z\to Z_i$, with $\epsilon_i\to 0$ as $i\to \infty$.
\end{definition}

\begin{definition}\label{def:strongconverge}
	Let $(Z,p)$ be a pointed metric space. We say that the sequence of pointed metric spaces $(Z_i,p_i)$ \emph{strongly converges} to $(Z,p)$ if the following holds:
	For every $R>0$, for all but finitely many $i$ there are isometries $\phi_i\co B_R(p)\to B_R(p_i)$ so that $\phi_i(p)=p_i$.
\end{definition}

The connection between these two notions of convergence is given in the following.
\begin{proposition}\cite[Proposition 3.5]{GMS}\label{prop:GMS3.5}
Let $\delta>0$.
  Let $(\Upsilon_i,p_i)$ be a sequence of pointed $\delta$--hyperbolic $\delta$--visible metric spaces which strongly converge to the pointed $\delta$--hyperbolic, $\delta$--visible metric space $(\Upsilon,p)$.  Let $\rho$ be a $\delta$--adapted visual metric at $p$ on $\partial \Upsilon$ and, for each $i$, let $\rho_i$ be a $\delta$--adapted visual metric at $p_i$ on $\partial \Upsilon_i$.  The space $(\partial \Upsilon,\rho)$ is a weak Gromov-Hausdorff limit of the spaces $(\partial \Upsilon_i,\rho_i)$.
\end{proposition}

\subsection{Guessing geodesics}

In order to verify that a space $\Xz^{\mathrm{cusp}}(K)$ we build in Subsection \ref{ss:noname} is hyperbolic, we need the following result of Bowditch. For a connected (simplicial) graph $A$, from now on we always consider a path metric $d_A$ on its vertex set.

\begin{proposition} \cite[Proposition 3.1]{bowditch:uniformhyp} \label{prop:Bowditch guess}
Given $h \ge 0$ there exists $k \ge 0$ with the following property.  Suppose that $\Gamma$ is a connected graph, and that for each $x,y \in V(\Gamma)$ there is an associated connected sub-graph $\Path(x,y) \subseteq \Gamma$ so that
\begin{enumerate}
\item For all $x,y,z \in V(\Gamma)$
\[	\Path(x,y) \subseteq N_h \left(\Path(x,z) \cup \Path(z,y) \right)	;	\]
\item For any $x,y \in V(\Gamma)$ with $d_\Gamma(x,y) \le 1$, the diameter of $\Path(x,y)$ in $\Gamma$ is at most $h$.
\end{enumerate}
Then $\Gamma$ is $k$--hyperbolic.  In fact, it suffices to take any $k \ge \frac{1}{2} (3m-10h)$, where $m$ is any positive real number so that
\[	2h(6 + \log_2(m+2)) \le m	.	\]
Moreover, for all $x,y \in V(\Gamma)$, the Hausdorff distance between $\Path(x,y)$ and any geodesic from $x$ to $y$ is at most $m - 4h$.
\end{proposition}

\section{Coarse Topology}\label{s:coarse_topology}

In this section we introduce some tools used in Sections~\ref{s:shells} and \ref{sec:unwrap}. We are working with graphs, so we assume that $D$, $Q \in \mathbb{N}$.

\begin{definition} \label{def:coarse_fund}(Coarse fundamental group) Let $\Gamma$ be a connected graph and let $D\geq 0$.
Let $\Gamma^{D}$ be the space obtained from $\Gamma$ by gluing disks to all edge-loops of length $\leq D$.

Set $$\pi_1^{D}(\Gamma,a_0) = \pi_1(\Gamma^{D},a_0).$$

The {\em $D$--universal cover} of $\Gamma$
is the natural preimage of $\Gamma$ in the universal cover of $\Gamma^{D}$. More generally, a {\em $D$--covering} map of $\Gamma$ is the restriction to the preimage of $\Gamma$ of a covering of $\Gamma^{D}$.
\end{definition}

Observe that a $D$--covering map of $\Gamma$ is a covering for which all the loops of length at most $D$ in $\Gamma$ lift.
The following lemma is a direct consequence of the definition.

\begin{lemma}
Suppose that $\Gamma$ is a connected graph, let $D>0$, let $\widetilde{\Gamma}$ be the $D$--universal cover of $\Gamma$, and let $p \in \widetilde{\Gamma}$.  Then $\pi_1^D(\widetilde{\Gamma},p) = \{ 1 \}$.
\end{lemma}

The coarse Cartan--Hadamard Theorem below uses the language of the following Definition \ref{d:CSC}, but our focus is on the length rather than the diameter of loops.  We relate these two notions in Lemma \ref{lem:csc v pi_1^D} below.

\begin{definition}\cite[Appendix A]{Coulon14} \label{d:CSC}
  Let $D>0$.  A connected graph $\Gamma$ is said to be {\em $D$--simply connected} if $\pi_1(\Gamma)$ is generated by loops freely homotopic to loops of diameter at most $D$.  If $\Gamma$ is $D$--simply-connected for some $D > 0$ we say that $\Gamma$ is {\em coarsely simply-connected}.
\end{definition}

Basic examples of coarsely simply-connected spaces are hyperbolic spaces.  The following is an immediate consequence of \cite[Lemma III.H.2.6]{BH}.
\begin{lemma}\label{lem:hyp_simply_conn}.
 Let $\Upsilon$ be a $\delta$--hyperbolic graph and, $D\geq 16\delta$. For all $y \in \Upsilon$ we have $\pi_1^{D}(\Upsilon,y)=\{1\}$.
\end{lemma}

\begin{lemma} \label{lem:csc v pi_1^D}
  Let $\Gamma$ be a graph, and $D>0$. If $\Gamma^{2D}$ is simply connected, then $\Gamma$ is $D$--simply connected.  Conversely, if $\Gamma$ is $D$--simply connected, then $\Gamma^{2D+1}$ is simply connected.
\end{lemma}

\begin{proof} 
	Suppose $\Gamma^{2D}$ is simply connected, and let $\alpha$ be an edge-loop in $\Gamma$. 
	There exists a (singular) disk diagram $\triangle \to \Gamma^{2D}$ for $\alpha$, that is, $\partial \triangle$ is mapped to $\alpha$ and $2$--cells in $\triangle$ are mapped to $2$--cells in $\Gamma^{2D}$. The boundary of each such $2$--cell has diameter at most $D$, hence the class of $\alpha$ in $\pi_1(\Gamma)$ is generated by loops of such diameter, proving the first implication.
	
	Now suppose that $\Gamma$ is $D$--simply connected. 
 Consider an edge-loop $\alpha$ in $\Gamma^{2D+1}$.  Such a loop is in fact a loop in $\Gamma$. The class of $\alpha$ in $\pi_1(\Gamma)$ is generated by loops of diameter at most $D$ hence
	there exists a map $\triangle^{(1)}\to \Gamma$ from the $1$--skeleton of a combinatorial (singular) disk $\triangle$ such that $\partial \triangle$ is mapped onto $\alpha$ and the boundary of every $2$--cell in $\triangle$ is mapped onto a loop of diameter at most $D$. Observe that any loop of such diameter can be decomposed into
	finitely many loops of length at most $2D+1$, hence the the map $\triangle^{(1)}\to \Gamma$ induces a map 
	$\triangle'\to \Gamma^{2D+1}$ from a subdivision $\triangle'$ of $\triangle$, justifying triviality of the homotopy class of $\alpha$ in $\Gamma^{2D+1}$, and finishing the proof.
\end{proof} 

When calculating constants later, we implicitly use the following immediate consequence of Lemma~\ref{lem:csc v pi_1^D}.  Note that constants are not optimal.

\begin{corollary}
Let $\Gamma$ be a connected graph, $a_0 \in \Gamma$, and $D > 0$.  If $\pi_1^D(\Gamma,a_0) = \{ 1\}$ then $\Gamma$ is $D$--simply-connected.
\end{corollary}

As mentioned above, the reason we wish to control the scale of simple connectivity is so we can apply the following theorem.
\begin{theorem}[Coarse Cartan--Hadamard] \cite[Theorem A.1]{Coulon14}\label{t:CCH}
  Let $\nu\geq 0$, and let $\sigma \geq 10^7\nu$.  Let $Z$ be a geodesic space.  If every ball of radius $\sigma$ in $Z$ is $\nu$--hyperbolic and if $Z$ is $10^{-5}\sigma$--simply-connected, then $Z$ is $300\nu$--hyperbolic.
\end{theorem}

\subsection{Combinatorial horoballs}

\begin{definition}\cite[Definition 3.1]{rhds}\label{def:combhoroball}
 Let $\Gamma$ be a simplicial graph.  The \emph{combinatorial horoball  based on $\Gamma$}, is the graph $\horba(\Gamma)$ with vertex set $\Gamma \times \mathbb{Z}_{\geq 0}$ and two types of edges: 
\begin{itemize} 
\item vertical edges connecting $(x, n)$ to $(x,n+1)$ 
\item horizontal edges connecting $(x,n)$ to $(y,n)$ if $d_\Gamma(x,y) \leq 2^n$. 
\end{itemize} 
Vertices $(x,n)$ are \emph{at depth $n$}.
\end{definition} 

The graph $\Gamma$ naturally includes as a subgraph of $\horba(\Gamma)$ but its metric is distorted exponentially.  More precisely we have the following.
\begin{lemma}\label{lem:horbadistort}
   Let $\Gamma$ be a connected simplicial graph, and let $\horba = \horba(\Gamma)$ be the horoball based on $\Gamma$.  For any distinct vertices $v,w\in \Gamma$, we have
   \begin{equation}\label{eq:horbadistort}
       \frac{1}{2} d_\horba((v,0),(w,0))-2 <\log_2(d_\Gamma(v,w)) < \frac{1}{2}d_\horba((v,0),(w,0)) + 1.
   \end{equation}  
\end{lemma}
\begin{proof}
    As pointed out in \cite[Lemma 3.10]{rhds}, any two vertices of $\horba$ are joined by a \emph{regular} geodesic, which consists of an initial vertical subsegment, followed by a horizontal subsegment of length at most $3$, followed by a final vertical subsegment.  If both the vertical subsegments are nondegenerate, then the horizontal subsegment has length at least $2$.  If $d_\Gamma(v,w)<6$, then $d_\Gamma(v,w) = d_\horba((v,0),(w,0))$ and the inequalities~\eqref{eq:horbadistort} are easily verified.

    Otherwise, the vertical segments are nondegenerate.  Let $h\in\{2,3\}$ be the length of the horizontal subsegment and let $D$ be the length of each vertical subsegment, so $d_\horba((v,0),(w,0)) = 2D + h$.  We have
    \[ 2^D < d_\Gamma(v,w) \le h\cdot 2^D. \]  Taking logs and using the relationship between $D$ and $d_\horba((v,0),(w,0))$ yields the inequalities~\eqref{eq:horbadistort}.
    \end{proof}

\begin{lemma} \label{lem:horoballs csc}
Let $\Gamma$ be a connected simplicial graph, and let $\horba(\Gamma)$ be the combinatorial horoball based on $\Gamma$.  For any $D \ge 5$ and any $p \in \horba(\Gamma)$ we have $\pi_1^D(\horba(\Gamma),p) = \{ 1 \}$.
\end{lemma} 
\begin{proof}
In \cite[Definition 3.1]{rhds} the combinatorial horoball is defined by including $2$--cells which are triangles, squares, or pentagons.  In \cite[Proposition 3.7]{rhds} it is proved that this complex is simply-connected.
\end{proof}

\subsection{Coarse topology of glued spaces}

\begin{proposition} \label{prop:CvK}
 Let $\Gamma$ be a connected simplicial graph, and let $D\geq 5$.  Suppose that $\Gamma^D$ is simply connected.   Let $\{\Gamma_i\}$ be disjoint connected sub-graphs of $\Gamma$. 
 Let $\widehat \Gamma$ be obtained from $\Gamma$ by gluing a combinatorial horoball $\horba_i$ based on $\Gamma_i$ to each $\Gamma_i$.
 Then $\widehat{\Gamma}$ is connected and $\widehat{\Gamma}^D$ is simply connected.
\end{proposition}

\begin{proof} 
Since $\widehat{\Gamma}$ is obtained from gluing together connected graphs along connected sub-graphs, it is clear that $\widehat{\Gamma}$ is connected.  Therefore, the basepoint is irrelevant, and we assume that the basepoint $a$ is in $\Gamma$.

We first note since $\Gamma_i$ is connected, any path $p$ in $\horba_i$ with endpoints in $\Gamma_i$ is $D$--homotopic rel endpoints to a path in $\Gamma_i$, by Lemma~\ref{lem:horoballs csc}.

Consider now a loop $\ell$ based at $a \in \Gamma$. Up to homotopy in $\widehat{\Gamma}^D$, we can replace each maximal sub-path of $\ell$ contained in some $\horba_i$ with a path contained in $\Gamma_i$, thereby obtaining a loop $\ell'$ in $\Gamma$. Since $\pi_1^{D}(\Gamma,a)=\{1\}$, we have that $\ell'$, hence $\ell$, is $D$--homotopically trivial, as required.
\end{proof}

\subsection{Coarse deformation retractions}

\begin{definition}\label{def:def retract}
 Let $\Xi$ be a connected graph and $\Gamma \subset \Xi$ a subset.  A \emph{$Q$--deformation retraction} of $\Xi$ onto $\Gamma$ is a sequence of maps $f_0,\ldots,f_n,\ldots:\Xi\to \Xi$ satisfying the following properties:
 \begin{enumerate}
  \item\label{eq:0Id} $f_0$ is the identity,
  \item\label{eq:evId} $f_i|_{\Gamma}$ is the identity for all $i$,
  \item\label{eq:CL} for all $b_1,b_2\in \Xi$ so that $d_\Xi(b_1,b_2) \le 1$ and all $i$ we have $$d_\Xi(f_i(b_1),f_i(b_2))\leq Q,$$
  \item\label{eq:speed} for all $b \in \Xi$ and all $i$ we have $d_\Xi(f_i(b),f_{i+1}(b))\leq 1$,
  \item\label{stability} for each $b\in \Xi$ there is some $n\geq 0$ and $a\in \Gamma$ so that $f_i(b)=a$ for all $i\geq n$.  
  \end{enumerate}
  Define $f(b)$ to be the point $a$ as in item \eqref{stability}. We call $f$ the \emph{stable map}.
\end{definition}
The $Q$--deformation retractions we require are built in Lemma \ref{lem:cdr} below.  A natural example of a $Q$--deformation retraction is given by a projection of hyperbolic graphs onto a neighborhood of a quasi-convex sub-graph, see Lemma~\ref{lem:easy retract}.

It is important to observe that in the following proposition the conditions required on the $Q$--deformation retraction are stronger than those implied by the definition.  In particular, notice that in Item~\eqref{item:defret0} the $d_\Gamma$--distance between $f(b_1)$ and $f(b_2)$ is bounded by $Q$.
\begin{proposition}\label{prop:def retract on pi_1}
Suppose that $\Xi$ is a connected graph, and that $\Gamma \subset \Xi$ is a connected sub-graph.  
Let $a \in \Gamma$, and $Q > 0$.  
Suppose
\begin{enumerate}
\item\label{item:defret0} there exists a $Q$--deformation retraction $f_0, \ldots , f_n , \ldots : \Xi \to \Xi$ of $\Xi$ onto $\Gamma$ with stable map $f$, so that for all $b_1,b_2 \in \Xi$ for which $d_\Xi(b_1,b_2) \le 1$ we have $d_\Gamma(f(b_1),f(b_2)) \le Q$; and 
\item\label{item:defret1} there exists $D > 2Q+2$ so that all $\Gamma$--loops of length at most $Q D$ represent the (conjugacy class of the) identity element of $\pi_1^D(\Gamma,a)$.
\end{enumerate}
Then the inclusion $\iota \co \Gamma \into \Xi$ induces an isomorphism:
\[	\iota_{\ast} \co \pi_1^D(\Gamma,a) \to \pi_1^D(\Xi,a)	.	\]
Moreover, for any $k$, given any loop $\sigma$ in $\Xi$ of length $k$ (possibly not based at $a$), with consecutive vertices $b_0,b_1,\ldots,b_k=b_0$, let $f(\sigma)$ be a loop in $\Gamma$ obtained by connecting $f(b_i)$ with $f(b_{i+1})$ by a path of length at most $Q$ in $\Gamma$, for all $i$ (this is possible by condition \eqref{item:defret1} above). Then $f(\sigma)$ represents the same conjugacy class as $\sigma$ in $\pi_1^D(\Xi,a) \cong \pi_1^D(\Gamma,a)$.
\end{proposition}
\begin{proof}
Any inclusion $\iota \co \Gamma \into \Xi$ between connected graphs induces a homomorphism $\iota_{\ast} \co \pi_1^D(\Gamma,a) \to \pi_1^D(\Xi,a)$, for any $a \in \Gamma$.

Take a loop $\sigma$ in $\Xi$, and construct a loop $f(\sigma)$ as in the last statement of the theorem. If $\sigma$ was based at $a$, then so is $f(\sigma)$.
Different choices of paths between $b_i$ and $b_{i+1}$ give loops differing by a collection of loops of length at most $2Q \le Q D$, so give the same element of $\pi_1^D(\Gamma,a)$.  Moreover, a relation in $\pi_1^D(\Xi,a)$ is a loop of length at most $D$ in $\Xi$, which is mapped to a loop of length at most $Q D$ in $\Gamma$, which implies (by Condition~\eqref{item:defret1}) that there is a well-defined induced map
\[	f_{\ast} \co \pi_1^D(\Xi,a) \to \pi_1^D(\Gamma,a)		.	\]
Since $f$ is a $Q$--deformation retraction, we clearly have $f \circ \iota = \mathrm{Id}\mid_\Gamma$, which means that 
$f_{\ast} \circ \iota_{\ast} = \mathrm{Id}\mid_{\pi_1^D(\Gamma,a)}$, and thus $\iota_\ast$ is injective.

However, if $\sigma$ is any $\Xi$--loop based at $a$, the fact that $D > 2Q+2$ means that the maps $( f_i)$ provide a homotopy in $\Xi^D$ between $\sigma$ and $f(\sigma)$, which means that $\iota_\ast$ is surjective.

The final statement of the result follows immediately from the above argument.
\end{proof}

\section{Linear connectivity}\label{sec:linear_connectedness}

In this section we investigate linear connectedness of boundaries.  In particular, we link ``uniform" linear connectedness to (weak) cut points (see Theorem~\ref{thm:lin_conn_cut_point} and also Proposition~\ref{prop:cut point implies not uniformly linearly connected} and~\ref{prop:no cut points}).  We also define a condition on graphs called ``spherical connectivity" which allows us to prove a local-to-global result about linearly connected boundaries (see Lemma~\ref{lem:sphere_conn_implies_lin_conn}).  In later sections, we use this criterion to verify that the boundaries of our partially unwrapped-and-glued spaces are uniformly linearly connected, which is crucial in our calculation that the Bowditch boundary of the drilled group is a two-sphere (see Corollary~\ref{cor:modeled_on_hatX_LC}).

\begin{definition} \label{def:lc} A metric space $M$ is $L$--linearly connected if every pair $\lbrace x,y \rbrace$ of points is contained in a connected set $J$ of diameter at most $L \cdot d(x,y)$.  A metric space is {\it linearly connected} if it is $L$--linearly connected for some $L$. \end{definition} 
This condition is also called $L$-$LLC_1$ (see, for example \cite[Section 2]{BonkKleiner:QS}), and has a companion property $L$-$LLC_2$, which together make the property \emph{$L$--locally linearly connected}. Connected boundaries of hyperbolic groups are locally linearly connected, but we only need the property from Definition~\ref{def:lc} in this paper.

As noted by Mackay \cite{MackayQuasiArcs}, when $M$ is locally compact and locally connected, by increasing $L$ by an arbitrarily small amount we may assume $J$ is an arc. 

We use the following notation several times in this section. If $a$, $b$ are two positive quantities, we write $a\lesssim b$ or $b\gtrsim a$ if there is a constant $u= u(\delta)>0$ such that $a\le u b$.
We write $a\asymp b$ if $a\lesssim b$ and $a\gtrsim b$. 
We say that $a \leq b + O(\delta)$ if there exists non-negative number  $C(\delta)$ such that $a-b \leq C\delta$. 

\subsection{Linear connectivity versus cut points}

A \emph{continuum} is a compact, connected, Hausdorff topological space.  A \emph{cut point} of a continuum $M$ is a point $x\in M$ such that $M\smallsetminus x$ is not connected. 
\begin{definition}
 Suppose that $M$ is a continuum.  A point $x \in M$ is a {\em weak cut point} of $M$ if there exist $p,q \in M \smallsetminus \{ x \}$ so that every sub-continuum of $M$ that contains both $p$ and $q$ must also contain $x$.
\end{definition}
An example of a weak cut point that is not a cut point is any point on the vertical line in a topologist's sine curve. 

The next result is one of the main results of this section. We use it in Lemma~\ref{lem:unwrap_is_lin_conn} for proving that our first unwrapped and glued space has boundary which is uniformly linearly connected.

\begin{theorem}\label{thm:lin_conn_cut_point} 
Suppose $\Upsilon_i$ are $\delta$--hyperbolic and $\delta$--visual metric spaces with finitely many isometry types of balls of radius $R$ for every $R>0$ (not depending on $i$).   Endow each $M_i = \partial \Upsilon_i$ with a $\delta$--adapted visual metric $\rho_i$ with basepoint $p_i$, and suppose each $M_i$ is connected.  Further suppose that $M_i$ is not  $L_i$--linearly connected for $L_i \rightarrow \infty$.  

Then there exists a sequence $p'_i$ of basepoints such that a subsequence of $(\Upsilon_i, p'_i)$  strongly converges to $(\Upsilon_\infty, p'_\infty)$ with the property that $M_\infty = \partial \Upsilon_\infty$ has a weak cut point. 
\end{theorem}

We need the following result, which is \cite[Lemma 6.19]{GMS}.

\begin{lemma} \label{lem:GMS6.19}
Let $M$ be a compact metric space. Suppose that there exists $L \geq 1$ so that each $p,q \in M$ can be joined by a chain of points $p = p_1,\ldots,p_n = q$ so that $diam(\lbrace p_1,\ldots, p_n \rbrace) \leq Ld(p, q)$ and $d(p_i, p_{i+1}) \leq d(p, q)/2$. Then $M$ is $5L$--linearly connected. \end{lemma} 

We use the following lemma to find the new basepoints $p_i'$ in the proof of Theorem \ref{thm:lin_conn_cut_point}:

\begin{lemma}\label{lem:change_basepoint}
 For every $\delta >0$, there exist constants $c$,$\lambda,t_0>1$ so that the following holds. 
  Suppose that $\Upsilon$ is $\delta$--hyperbolic and $\delta$--visible, and let $p\in \Upsilon$, and $x,y\in\partial \Upsilon$, with $x \ne y$.   Let $\rho$ be a $\delta$--adapted visual metric on $\partial \Upsilon$ at $p$. There exists $p'\in \Upsilon$ so that in any $\delta$--adapted visual metric $\rho'$ on $\partial \Upsilon$ at $p'$ we have:
 \begin{enumerate}
 \item $\rho'(x,y)>1/c$,\label{item:far}
 \item for every $t \geq t_0$, the set $\mathcal A_t=\{z\in \partial \Upsilon: \rho(z,x)\geq t \rho(x,y)\}$ has $\rho'$--diameter at most $c/t$
 and for any $z\in\CalA_t$, $\rho'(x,z),\rho'(y,z)\ge 1/c$,\label{item:diam}
 \item on $\partial \Upsilon - \mathcal A_t$, $\frac{1}{\rho(x,y)}\rho$ and $\rho'$ are $\lambda t^2$--bi-Lipschitz equivalent. \label{item:bilip}
\end{enumerate}
\end{lemma}

\begin{proof} 
Let $a= e^{1/6\delta}$. By the definition of $\delta$--adapted visual metric, we have $\rho(w,z)\asymp a^{-(w|z)_p}$ for all $w,z\in \partial \Upsilon$.  
We use the fact that if $x,y$ are points on $\partial \Upsilon$, then as in \cite[III.H.3.18.(3)]{BH}, 
\[  (x|y)_p = d(p,(x,y)) + O(\delta) , \]
where $(x,y)$ is any bi-infinite geodesic from $x$ to $y$.

By \cite[Thm 2.12]{GdlH}, there exists a constant $C=C(\delta)\ge 0$ 
such that, for any $\{p,x,y, z,w\} \in \Upsilon \cup \partial \Upsilon$, $p \in \Upsilon$, there is a $(1,C)$--quasi-isometry
$$f\co  Y= [p,x)\cup [p,y)\cup [p,z)\cup [p,w)\to T$$ into a metric tree which
preserves distances to $p$. Let us write  $\bar u = f(u)$ for any $u\in Y$.  For $u\in \partial Y = \{x,y,z,w\}$ we write $\bar u$ for the corresponding point in $\partial T$.
It follows that, for any $q\in Y$ and $u,v\in Y\cup\partial Y$, we have 
$(\bar u|\bar v)_{\bar q}= d(\bar q, \bar c)= (u|v)_q + O(\delta)$
where $\bar c$ is the center of the tripod $\{\bar q, \bar u,\bar v\}$.

\medskip 
Let us consider the point $p'$ on a geodesic ray with origin $p$  which is asymptotic
to $x$ such that $d(p,p')= (x|y)_p$.  The approximating tree is a tripod where $\bar p'$ is $O(\delta)$ from the center. This implies that $\rho' (x,y)$ is comparable to 1, because $( x | y)_{p'}=O(\delta)$. This proves (\ref{item:far}), for $c\ge c_1(\delta)$
\medskip

From now on, we assume that $\bar p' = f(p')$ is the center of the tripod
$\{\bar p, \bar x,\bar y\}$.

\medskip

We turn to the proof of (\ref{item:diam}).  Let $t\ge 1$ and pick $z,w$ such that  $\rho(w,x),\rho(z,x)\ge t \rho(x,y)$ , i.e.,  $\max\{(x|w)_p, (x|z)_p\}\le (x|y)_p -\log_a t + O(\delta)$, see Figure~\ref{fig:PH1}(left). 
Thus, $$(\bar x| \bar z)_{\bar p}\le   (x|z)_p + O(\delta) \le (x|y)_p -\log_a t + O(\delta) \le  d_T(\bar p, \bar p')  -\log_a t + O(\delta)$$
and similarly $$(\bar y| \bar z)_{\bar p}\le   d_T(\bar p, \bar p')  -\log_a t + O(\delta)$$ as 
$$\rho(y,z) \ge (t-1)\rho(x,y) \hbox{ by the triangle inequality and } (y|z)_p \le  (x|y)_p - \log_a(t-1)\,.$$
Thus, we may find some $t_0>1$ large enough such that, if $t\ge t_0$, then,
in the approximate tree $T$, the point $\bar p'$ belongs to $(\bar z, \bar x)\cap (\bar z, \bar y)$. 

\begin{figure}[h!]
	\begin{center}
		\includegraphics[scale=0.6]{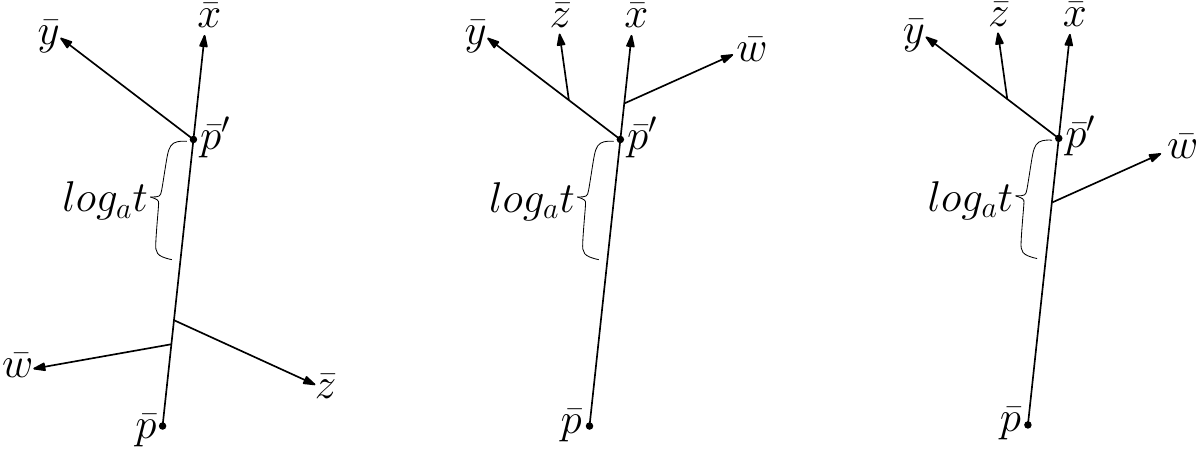}
	\end{center}
	\caption{The proof of Lemma~\ref{lem:change_basepoint}: Conclusion (\ref{item:diam}) (left); conclusion (\ref{item:bilip}), in case $(\bar x| \bar w)_{\bar p}  \ge d_T(\bar p - \bar p')$ (center); Conclusion (\ref{item:bilip}), in case $(\bar x| \bar w)_{\bar p}  \le d_T(\bar p - \bar p')$ (right).}\label{fig:PH1}
\end{figure}

We obtain, with the observation that $(\bar x|\bar w)_{\bar p}= (\bar p'|\bar w)_{\bar p}$ and
$(\bar x|\bar z)_{\bar p}= (\bar p'|\bar z)_{\bar p}$,
$$\left\{\begin{array}{l} 
(\bar x| \bar w)_{\bar p} + (\bar p| \bar w)_{\bar p'} = (\bar x|\bar y)_{\bar p}\\
(\bar x| \bar z)_{\bar p} + (\bar p| \bar z)_{\bar p'} = (\bar x|\bar y)_{\bar p}\end{array}\right.$$
Therefore,
\begin{eqnarray*}
(\bar w| \bar z)_{\bar p'} & \ge &  \min\{(\bar p| \bar z)_{\bar p'},(\bar p| \bar w)_{\bar p'}\} \\
& \ge &  (\bar x| \bar y)_{\bar p}  - \max\{(\bar x| \bar z)_{\bar p}, (\bar x| \bar w)_{\bar p}  \}\\
& \ge & \log_a t + O(\delta)\,.\end{eqnarray*}
Hence $\rho ' (z,w)\le  c_2/t$ for some $c_2=c_2(\delta)>0$.  This implies that the diameter
of $\mathcal A_t$ is bounded by $c_2/t$. Note that $\bar{p}'$ belongs to both geodesics  $(\bar x, \bar z)$ and $(\bar y,\bar z)$ so that $\rho'(x,z),\rho'(y,z) \asymp 1$ thus 
$\rho'(x,z),\rho'(y,z) \gtrsim 1$, which ends the proof of  (\ref{item:diam}), taking $c = \max\{c_1,c_2\}$.

\medskip

Let us deal with the proof of (\ref{item:bilip}). 
If $w,z \in \partial \Upsilon - \CalA_t$,  
then $\rho(w,x),\rho(z,x)\le t \rho(x,y)$, i.e.,  $\min\{(x|w)_p, (x|z)_p\}\ge (x|y)_p -\log_a t + O(\delta)$.
By the triangle inequality, we have
\begin{equation}\label{eq:lma6} |(z|w)_{p'}-(z|w)_p|\le d_{\Upsilon}(p- p')= (x|y)_p. \end{equation} Therefore
$$ \frac{ \rho'(z,w)}{\rho(x,y)^2}\gtrsim \frac{\rho(z,w)}{\rho(x,y)} \gtrsim \rho'(z,w)\,.$$
This is one of the coarse inequalities that we need.  Now we find an upper bound independent of $\rho(x,y)$.

If $\log_a t \ge (x|y)_p$ then we have essentially no constraints: 
 the inequality \eqref{eq:lma6}  implies that
$$\frac{\rho(z,w)}{\rho(x,y)}\lesssim  \frac{\rho'(z,w)}{\rho(x,y)^2}  \lesssim t^2  \rho'(z,w)\,.$$
So we restrict our attention to the condition  $\log_a t\le (x|y)_p$. Let us work on the approximate tree $T$. 
Switching $z$ and $w$ if necessary, we may assume that  $(\bar x| \bar w)_{\bar p} \le (\bar x| \bar z)_{\bar p}  $.
Let us consider two cases. If $(\bar x| \bar w)_{\bar p}  \ge d_T(\bar p - \bar p')$ (see Figure~\ref{fig:PH1}(center)) then 
$$(\bar z|\bar w)_{\bar p} = d_T({\bar p}  -\bar{p}') + (\bar z|\bar w)_{\bar p'}  = (\bar x|\bar y)_{\bar p} + (\bar z|\bar w)_{\bar p'}$$
hence $$\frac{\rho(z,w)}{\rho(x,y)} \asymp \rho'(z,w)\,.$$
If, on the contrary, $(\bar x| \bar w)_{\bar p}  \le d_T(\bar p - \bar p')$ (see Figure~\ref{fig:PH1}(right)) then we first observe that $w\in B(x,t\rho(x,y))$ implies that
$(x|w)_p \ge (x|y)_p -\log_a t + O(\delta)$ holds. Together with $(\bar z|\bar w)_{\bar p} =  (\bar x |\bar w)_{\bar p}$,
we deduce
that 
\begin{eqnarray*} (\bar w|\bar z)_{\bar p'} \le (\bar w|\bar p)_{\bar p'} & = &  (\bar x|\bar y)_{\bar p} - (\bar p'|\bar w)_{\bar p}\\
& = & (\bar x|\bar y)_{\bar p} - (\bar x|\bar w)_{\bar p}\\
& \le &    \log_a t  +O(\delta)
 \end{eqnarray*}
and hence
\begin{eqnarray*} (\bar x|\bar y)_{\bar p}   - ( \bar   z|\bar w)_{\bar p}   
  & \le  &    \log_a t  +O(\delta) \\ &  \le &   2  \log_a t - (\bar z|\bar w)_{\bar p'}  +O(\delta)  \,.
 \end{eqnarray*}
 This translates into $$\frac{\rho(z,w)}{\rho(x,y)} \lesssim  t^2 \rho'(z,w)$$
 and concludes the proof of (\ref{item:bilip}) and of the lemma.
\end{proof}

Before presenting the proof of Theorem  \ref{thm:lin_conn_cut_point}, we introduce the following terminology. A {\it fine chain} joining  two points $x$ and $y$ of a metric
space  $M$ is a finite sequence of points $m_0,\ldots, m_n$ such that $m_0=x$, $m_n=y$ and $d(m_j,m_{j+1})\le d(x,y)/2$ for all $0\le j <n$. 

\begin{proof}[Proof of Theorem \ref{thm:lin_conn_cut_point}]
We first pass to a subsequence so that $L_i\ge 5i$, so that the space $M_i$ is not $5i$--linearly connected.
 Lemma \ref{lem:GMS6.19} implies that for each $i$ we may find a pair of points $(x_i, y_i)$ in $M_i$ so that any fine chain in $M_i$
joining $x_i$ to $y_i$  meets the (non-empty) set  $\CalA_i= M_i \ssm B_{\rho_i}(x_i, i \rho_i(x_i,y_i))$.

Let $p'_i$ be chosen according to Lemma \ref{lem:change_basepoint}, applied with $x=x_i, y=y_i$, and let $\rho'_i$ be the corresponding re-centered $\delta$--adapted visual metric from the conclusion of that lemma.  
Let $(\Upsilon_\infty, p'_\infty)$ be a strong limit of a subsequence of $(\Upsilon_i, p'_i)$, which exists because there are finitely many isometry types of balls of any radius in the spaces $\Upsilon_i$.  
The  spaces $(M_i, \rho_i')$ weakly Gromov--Hausdorff converge to $(M_\infty,\rho'_\infty)$ where $M_\infty=\partial \Upsilon_\infty$ and $\rho'_\infty$ is a $\delta$--adapted visual metric, by Proposition~\ref{prop:GMS3.5}.  

Denote by $\phi_i \co (M_i,\rho'_i) \to (M_\infty,\rho'_\infty)$ choices of $(k,\epsilon_i)$--quasi-isometries that witness the weak Gromov--Hausdorff convergence. Choose $z_i$ in $\CalA_i$.  
Let $\psi_i\co M\to M_i$ be quasi-inverses for the $\phi_i$'s.  We may assume that, for each $i$, $\psi_i(\phi_i(z_i))=z_i$ and that $\psi_i$ is a  $(k,\epsilon_i)$--quasi-isometry, by replacing $\epsilon_i$ by $3k\epsilon_i$ if necessary.

Passing to a  subsequence and reindexing,  we may also assume that $\epsilon_i i^2$ tends to zero,  and that
$\phi_i(x_i)$ (resp. $\phi_i(y_i)$, $\phi_i(z_i)$) converges to a point $x_\infty$ (resp. $y_\infty$, $z_\infty$) in $M_\infty$.  
The points $x_\infty$, $y_\infty$ and $z_\infty$ are at least $(1/kc)$--apart by Lemma \ref{lem:change_basepoint} \eqref{item:far} and \eqref{item:diam}. 

\medskip

Recall that every fine chain connecting $x_i$ to $y_i$ meets $\CalA_i$. The diameters with respect to $\rho_i'$ of the $\mc{A}_i$ go to $0$ as $i \to \infty$ by Lemma \ref{lem:change_basepoint}\eqref{item:diam}, 
and $\phi_i(\CalA_i)$  converge to $z_\infty$ in $M_\infty$.  The intuitive idea is that the $\mathcal A_i$ separate $x_i$ from $y_i$, so in the limit $z_\infty$ is a weak cut point of $M_\infty$. 
We now make this more precise.

\medskip

In order to obtain a contradiction, suppose that $z_\infty$ is not a weak cut point in $M_\infty$, so that there is a sub-continuum $K$ of $M_{\infty}$ which contains $x_\infty$ and $y_\infty$ but not $z_\infty$. 

We first claim that, for $i$ large enough, $\psi_i(K)$ does not meet $\CalA_i$.  We show that there exists
$\eta>0$ such that, for any $i$ large enough, $\rho_i'(\psi_i(w),w_i) \ge \eta$  for any $w\in K$ and $w_i\in \CalA_i$:
\begin{eqnarray*}
\rho_i'(\psi_i(w),w_i) & \ge &  \rho_i'(\psi_i(w), \psi_i(\phi_i(z_i))) - \rho'_i(z_i,w_i)\\
& \ge & \frac{1}{k} \rho'_\infty(w,\phi_i(z_i)) - \epsilon_i - \hbox{\rm diam}_{\rho_i'} \CalA_i \\
& \ge &  \frac{1}{k} \rho'_\infty(w,z_\infty) - \left(\frac{1}{k} \rho_\infty'(\phi_i(z_i),z_\infty) +\epsilon_i + c/i\right) \\
& \ge &  \frac{1}{k} \rho'_\infty(z_\infty,K) - \left(\frac{1}{k} \rho_\infty'(\phi_i(z_i),z_\infty) +\epsilon_i + c/i\right)
\end{eqnarray*}
Thus, if $i$ is sufficiently large, then $\eta= \rho'_\infty(z_\infty,K)/(2k)$ bounds $\rho_i'(\psi_i(w),w_i)$ from below.

For any $i\ge 0$, we may find a finite subset $F_i\subset K$ such that $K\subset \cup_{x\in F_i} B_{\rho_\infty'}(x, \epsilon_i/k)$. Since $K$ is a continuum,  the nerve of this cover is connected.
Applying $\psi_i$, we get that $\psi_i(K) \subset \cup_{x\in F_i} B_{\rho_i'}(\psi_i(x),  2\epsilon_i)$ and the nerve of this cover is also connected.
For $i$ large enough, $\rho_i'(\psi_i(K),\CalA_i) \ge \eta >  2\epsilon_i$, so Lemma \ref{lem:change_basepoint} \eqref{item:bilip} implies that, for some $\lambda >0$ and each $x\in F_i$,
\[B_{\rho_i'}(\psi_i(x),  2\epsilon_i) \subset  B_{\rho_i}(\psi_i(x),  2\lambda  i^2\epsilon_i\rho_i(x_i,y_i))\]
so that \[ \psi_i(K) \subset \cup_{x\in F_i} B_{\rho_i}(\psi_i(x),  2\lambda  i^2\epsilon_i\rho_i(x_i,y_i)).\]
The nerve of the cover $\{B_{\rho_i}(\psi_i(x),  2\lambda  i^2\epsilon_i\rho_i(x_i,y_i))\}_{x\in F_i}$ of $\psi_i(K)$ is still connected.

Since we made sure that $\lim i^2\epsilon_i=0$, we may pick $i$ large enough so that $ 2\lambda  i^2\epsilon_i\le 1/4$. Thus, $\psi_i(F_i) \subset \psi_i(K)$ contains a fine chain connecting $x_i$ and $y_i$.  Since $\psi_i(K)\subset M_i \smallsetminus \mc{A}_i$, we obtain a contradiction.
\end{proof}

In the opposite direction we have the following statement about hyperbolic spaces with cut points in their boundary; the visual metrics on the boundary of such a space cannot be uniformly linearly connected.  (A collection of metrics is \emph{uniformly} linearly connected if there is some constant of linear connectedness which works for all of them.)
\begin{proposition}\label{prop:cut point implies not uniformly linearly connected}
Let $X$ be a proper geodesic $\delta$--hyperbolic space whose boundary is connected and has a cut point.  There is a sequence of points $\{x_i\}$ in $X$ satisfying the following.   If for each $i$, $\rho_i$  is a $\delta$--adapted visual metric based at $x_i$ then the metrics $\{\rho_i\}$ are not uniformly linearly connected.
\end{proposition}
\begin{proof}
  Let $p\in \partial X$ be a cut point, and let $a,b$ be in different components of $\partial X\ssm\{p\}$.  Since every curve joining $a$ and $b$ has to go through the point $p$, it suffices to find points $x_i$ so that the Gromov products $(a|p)_{x_i}$ and $(b|p)_{x_i}$ are bounded above, but the Gromov products $(a|b)_{x_i}$ tend to infinity.
  
  Let $\gamma$ be any geodesic ray in $X$ tending to $p$.  Then the points $x_i = \gamma(i)$ satisfy this property.
\end{proof}

For the Bowditch boundary of a relatively hyperbolic pair, existence of a cut point prevents even non-uniform linear connectivity.

\begin{proposition} \label{prop:no cut points}
 Let $(H,\mathcal P)$ be a relatively hyperbolic pair, with each $P\in\mathcal P$ infinite.
 If the Bowditch boundary (when endowed with any visual metric) is linearly connected, then it does not have cut points.\end{proposition}
\begin{proof}
 Suppose that the Bowditch boundary $\partial(H,\mc{P})$ has a cut point $p$.  Then $p$ is a parabolic point by \cite[Theorem 1.1]{Dasgupta_Hruska}.
 Let $x$ and $y$ lie in different components of $\partial(H,\mc{P})\ssm\{p\}$. Let  $J$ be any arc connecting $x$ and $y$, so containing $p$. It follows from Lemma \ref{lem:cvparab} below that, for any visual distance $\rho$,
 we have for all $h\in \Stab(p)$,  $$\frac{\rho(hx,hy)}{\hbox{\rm diam}_\rho h(J)} \le   \frac{\rho(hx,hy)} {\rho(p, hx) }   \lesssim  \rho(p, hy ) \,.$$
Considering an infinite  sequence  $(h_n)$, the right-hand side tends to $0$, so $\partial (H,\mc{P})$ cannot be linearly connected.
\end{proof}

\begin{lemma}\label{lem:cvparab} Let $\Upsilon$ be a proper geodesic  $\delta$--hyperbolic space and let $P$ be a parabolic group of isometries of $\Upsilon$ fixing the parabolic point $p \in \partial \Upsilon$.  Let $\rho$ be a $\delta$--adapted visual metric on $\partial \Upsilon$, and let $x,y\in \partial \Upsilon\ssm\{p\}$.  

There is a $C>0$ so that, for all $g\in P$,
$$\rho(gx,gy) \le C  \rho(p,gx)  \rho(p, gy) .$$

\end{lemma}

\begin{proof} 
 It  follows from 
\cite[Prop.\,8.11, Rem.\,8.13.ii, Thm.\,8.16]{GdlH} that horoballs based at $p$ are quasi-preserved by $P$: there exists a constant
$C_0 \ge 0$ depending only on $\delta$ such that, for any $x\in \Upsilon$ and any $g\in P$, 
\begin{equation}\label{eq;buse} |\beta_p(x,g(x))|\le C_0\end{equation} 
where $\beta_p$ denotes a  Busemann function at $p$.

\medskip

Let $w$ be the basepoint for the visual metric $\rho$, and write $B=\max \{ (x|p)_w,(y|p)_w\}$; fix $g\in P$. 
We  claim that there is a constant $C_1=C_1(\delta) \ge 0$ such that  
$$ (g(x) | g(y))_w \ge  (g(x)|p)_w +   (g(y)|p)_w\   -(   C_1  +2B)\,. $$
From the claim, we may conclude: 
$$\rho(gx,gy)\lesssim e^{-\epsilon (g(x)|g(y))_w}   \lesssim e^{-\epsilon ( ( g(x)|p)_w+(g(y)|p)_w )}  \lesssim  \rho(p,gx)  \rho(p, gy)\,.$$

\medskip
 
 To prove the claim, we first observe that, given any 
 four points  $x,y,w,p\in (\Upsilon\cup\partial \Upsilon)$ with $w\in \Upsilon$ and $p\in\partial \Upsilon$ 
 \begin{equation}\label{eq:config4}
\left| \left(\sup_{z\in (xy)}\beta_p(w,z)+ (x|y)_w\right) -((x|p)_w + (y|p)_w  )\right|\le C_2.\end{equation}
for some constant $C_2(\delta)\ge 0$. 
This may be checked by approximating these points by a tree \cite[Thm 2.12]{GdlH}.
 
Thus, we get
 $$ \sup_{z\in (xy)}\beta_p(w,z) \le C_2+2B\,.$$
 Since $g\in P$, we have by the quasi-cocycle property of Busemann functions \cite[Prop.\,8.2.iii]{GdlH} and by \eqref{eq;buse} $$\left| \sup_{z\in (xy)}\beta_p(w,z)  -  \sup_{z\in (g(x) g(y))}\beta_p(w,z)\right|  \le C_0 + C_3$$
 for some constant $C_3=C_3(\delta)$.
Hence, with a second application of  \eqref{eq:config4} we obtain
$$(g(x)|g(y))_w \ge  (g(x)|p)_w  +  (g(y)|p)_w   -  C_0 - C_3 - 2B -2 C_2$$
proving the claim.
\end{proof}
 
\subsection{Linear connectivity versus spherical connectivity}
In this subsection we introduce our local-to-global criterion for linear connectedness of the boundary.

\begin{definition}
Let $\Delta,R \ge 0$,
let $\Upsilon$ be a $\delta$--hyperbolic and $\delta$--visible metric space, and let $y \in \Upsilon$.  We say that $\Upsilon$ is {\em $(\Delta,R)$--spherically connected at $y$} if for any points $p,q\in \Upsilon$ 
with $d(p,y)=d(q,y)=R$, the following holds. There is a sequence $p=p_0,\dots,p_n=q$ so that $d(y,p_i)=R$, $(p_i|p_{i+1})_y \geq R - 5\delta$ and $(p | p_i)_y\geq (p | q)_y-\Delta$.
\end{definition}

\begin{lemma}\label{lem:lin_conn_implies_sphere_conn}
For every $\delta, L$ there exists $\Delta=\Delta(\delta,L)$ such that the following holds.  Let $\Upsilon$ be $\delta$--hyperbolic and  $\delta$--visible and $y \in \Upsilon$ be 
such that any $\delta$--adapted visual metric on $\partial \Upsilon$ with basepoint $y$ is $L$--linearly connected.  For every $R$ we have that $\Upsilon$ is $(\Delta,R)$--spherically connected at $y$.
\end{lemma}

\begin{proof} Since the only base point considered here is  the point $y$, we write $(\cdot|\cdot)_y=(\cdot|\cdot)$.
We first gather some estimates that are used in the proof. 

Let $z\in \Upsilon$ and $\xi$ be the endpoint of a ray $[y,\xi)$ that passes at distance $\delta$ from $z$. Then the $\delta$--hyperbolicity implies 
\begin{equation}\label{eq:lcisc1} (z|\xi) \ge d(y,z)- 2\delta\,.\end{equation}

Let $\xi,\zeta\in\partial \Upsilon$, $R>0$.  Let us consider points $p$ and $q$ at distance $R$ from the base point $y$ that are on rays defining $\xi$ and $\zeta$ respectively.
Then   $$ (p|q) \ge \min\{(p|\xi), (\xi| \zeta),(\zeta|q)\} -2\delta \ge \min\{R, (\xi| \zeta)\} - 2\delta  \ge  (\xi| \zeta)- 2\delta$$
since $(p|q)\le R$, and, by symmetry, we get 
\begin{equation}\label{eq:lcisc2}  (p|q) -2\delta \le (\xi| \zeta)\le  (p| q) + 2\delta \,. \end{equation}

\medskip

We now establish the spherical connectivity. Let $p$ and $q$ be on the sphere of radius $R$ around $y$. If $(p|q) \ge R-5\delta$, then we are done. Otherwise,
 we approximate $p$ and $q$ by rays $\gamma_p$ and $\gamma_q$ that pass within $\delta$ of $p$ and $q$, respectively.  
The associated points on $\partial \Upsilon$ are $\xi_p$ and $\xi_q$. 

The hyperbolicity and \eqref{eq:lcisc1} implies 
\begin{eqnarray*} 
(\xi_p| \xi_q)  & \ge & \min\{ (\xi_p| p), (p|q), (q| \xi_q)\} -2\delta\\
& \ge & \min\{R-2\delta, (p|q)\} - 2\delta\end{eqnarray*}
Since $(p|q)\le R-5\delta$ we have
\begin{equation}\label{eq:lcisc4}
(\xi_p|\xi_q) \ge (p|q) - 2\delta\,. \end{equation}

Let $J$ be an arc connecting $\xi_p$ and $\xi_q$ such that $\rho (\xi,\xi_p) \le L\rho (\xi_p,\xi_q)$ holds for  every point $\xi$ in $J$ and let $\eta=  e^{-\epsilon R}/\kappa$.  
Since $J$ is an arc, there is a sequence of points $(\xi_j)_{1\le j\le n-1}$ on $J$ represented by rays $\gamma_p = \alpha_1,\ldots,\alpha_{n-1} = \gamma_q$  with $\rho(\xi_j,\xi_{j+1})\le \eta$. 
Let $p_j$ be the intersection of $\alpha_j$ with the sphere of radius $R$ about  $y$, and write $p = p_0$ and $q = p_n$. Let us check that this chain satisfies the conditions of spherical connectedness.

We first note that $(p_0|p_1),\, (p_{n-1}|p_n)\ge R-\delta$ by construction; 
for $1\le j\le n-2$, the estimate \eqref{eq:lcisc2} shows us that
\begin{eqnarray*} 
(p_j|p_{j+1}) & \ge  &  (\xi_j | \xi_{j+1})-  2\delta \\
& \ge &  \frac{-1}{\epsilon} \log(\kappa\eta)  -2\delta \\
& \ge  & R-2\delta \,.
\end{eqnarray*}

If $j= 1,n$, then  $(p|p_j) \ge \min\{(p|q), R -2\delta\} \ge (p|q)$ by \eqref{eq:lcisc1} and  since $(p|q)\le R-5\delta$. 
If $2\le j\le n-1$, then by \eqref{eq:lcisc2}  and\eqref{eq:lcisc4}, 
\begin{eqnarray*} 
(p_1 |p_{j}) & \ge  &   (\xi_p  | \xi_{j}) -  2\delta \\
& \ge &   \frac{-1}{\epsilon} \log \kappa \rho(\xi_p,\xi_j) -2\delta \\
& \ge  & \frac{-1}{\epsilon} \log \kappa  L \rho(\xi_p,\xi_q)   -2\delta \\
& \ge  & \frac{-1}{\epsilon} \log (\kappa^2L) + (\xi_p|\xi_q)  -2\delta \\
& \ge &   (p|q)  - \frac{1}{\epsilon} \log (\kappa^2L) - 4\delta
\end{eqnarray*}
so  $$(p|p_j) \ge  \min\{(p|p_1),(p_1|p_j)\} -\delta \ge   (p|q)  - \frac{1}{\epsilon} \log (\kappa^2L) - 5\delta $$
since $(p|p_1) \ge R- \delta \ge R-5\delta\ge (p|q)$.   Therefore, we may set $$\Delta= \frac{1}{\epsilon} \log (\kappa^2L) + 5\delta$$
that only depends on $\delta$ and $L$.
\end{proof}

We now prove a kind of converse to Lemma \ref{lem:lin_conn_implies_sphere_conn}.

\begin{lemma}\label{lem:sphere_conn_implies_lin_conn}
For every $\Delta, \delta$ there exists $R_0 = R_0(\Delta,\delta)$ and $L = L(\Delta,\delta)$ so that the following holds.  Suppose that $\Upsilon$ is a $\delta$--hyperbolic and $\delta$--visible metric space, and that there exists $R \ge R_0$ so that for every $y \in \Upsilon$ the space $\Upsilon$ is $(\Delta,R)$--spherically connected at $y$.  Then any $\delta$--adapted visual metric on $\partial \Upsilon$ with any basepoint is $L$--linearly connected.
\end{lemma}

 \begin{proof} 
We fix a base point $y\in \Upsilon$ and let $\rho$ denote a  $\delta$--adapted visual distance.  We check the criterion of Lemma  \ref{lem:GMS6.19}. 
Let $\xi, \zeta \in \partial \Upsilon$ with $\xi \ne \zeta$ and assume that $R\ge R_0$ for some constant $R_0$ that will be fixed later on. 

Let us apply Lemma \ref{lem:change_basepoint} to the triple $(y,\xi,\zeta)$ and denote by $y'$ the new basepoint with associated visual distance $\rho'$.

Let $p$, $q$ be at distance $R$ from the base point $y'$ on rays $[y',\xi)$ and $[y',\zeta)$ respectively. Therefore, the argument of \eqref{eq:lcisc2} leads us to
\begin{equation}\label{eq:scilc1}   (p|q)_{y'}-2\delta \le (\xi|\zeta)_{y'} \le (p|q)_{y'} + 2\delta \,. \end{equation}
Since $\Upsilon$ is $(\Delta,R)$--spherically connected at $y'$, we may find
$p=p_0,\dots,p_n=q$ so that $d(y',p_j)=R$, $(p_j|p_{j+1})_{y'} \geq R - 5\delta$ and $(p | p_j)_{y'}\geq (p | q)_{y'}-\Delta$.
  
For any $ 1\le j\le n-1$, we may find  a ray that passes at distance at most $\delta$ from $p_j$ with end point $\xi_j\in \partial \Upsilon$. 

On the one hand, we have 
\begin{eqnarray*}  (\xi | \xi_j)_{y'}  & \ge & \min \{ (\xi|p)_{y'}, (p|p_j)_{y'}, (p_j|\xi_j)_{y'} \} - 2\delta\\
&  \ge  & \min\{R, (p|q)_{y'}-\Delta, R-2\delta\} - 2\delta\\
& \ge & (p|q)_{y'}- \Delta - 2\delta\\
& \ge & (\xi|\zeta)_{y'} - \Delta - 4\delta
\end{eqnarray*}
so that 
\begin{equation}\label{eq:comparison}
    \rho'(\xi,\xi_j) \le \kappa e^{-\epsilon(\xi | \xi_j)_{y'}} \le \kappa^2 e^{\epsilon(\Delta + 4\delta)} \rho'(\xi,\zeta)\,.
\end{equation}
This bound shows that there is a choice of $t$, depending only on $\delta$ and $\Delta$, so that for each $j$ we have $\xi_j \not\in \mathcal{A}_t$, where $\mathcal{A}_t$ is the set defined in Lemma~\ref{lem:change_basepoint}.\eqref{item:diam} with $x = \xi$, $y = \zeta$.
Lemma \ref{lem:change_basepoint}.\eqref{item:bilip} tells us that there is a constant $\lambda_0=\lambda t^2 \ge 1$, where $\lambda_0$ depends only on $\delta$ and $\Delta$ such that
$$\frac{1}{\lambda_0} \frac{ \rho(\xi_i,\xi_j)}{\rho(\zeta,\xi) } \le \rho'(\xi_i,\xi_j) \le  \lambda_0\frac{ \rho(\xi_i,\xi_j)}{\rho(\zeta,\xi) }$$ holds for all $0\le i,j\le n$.

On the other hand, the same argument as for \eqref{eq:lcisc4}  leads to $$(\xi_j|\xi_{j+1})_{y'} \ge (p_j|p_{j+1})_{y'}  -2\delta \ge R- 7\delta$$
so that, by definition of $R$, we have  $$\rho'(\xi_j,\xi_{j+1}) \le \kappa  e^{-\epsilon(R_0-7\delta)}\,.$$

Hence, Lemma \ref{lem:change_basepoint}.\eqref{item:bilip} (with the same choice of $t$, and hence $\lambda_0$, as above) implies that
\begin{eqnarray*} \rho(\xi_j,\xi_{j+1})   & \le & \lambda_0 \rho'(\xi_j,\xi_{j+1})  \rho(\zeta,\xi) \\
 & \le & \lambda_0\kappa  e^{-\epsilon(R_0-7\delta)}  \rho(\zeta,\xi) \\
 &  \le   & \rho (\xi,\zeta)/2 ,
\end{eqnarray*}
if we pick $R_0$ large enough.
This implies that the assumptions of Lemma \ref{lem:GMS6.19} hold and we may deduce that $\partial \Upsilon$ is $L$--linearly connected
where $L$ only depends on $\delta$ and $\Delta$.
\end{proof}

\section{Shells and tube complements} \label{s:shells}

In this section we study the geometry of large tubes and complements of large tubes around bi-infinite geodesics in hyperbolic spaces, in particular when the boundary of the hyperbolic space is $S^2$. One of the main results of the section is Corollary~\ref{cor:pi_1_Z}, which shows that the coarse fundamental group of the boundary of a tube is isomorphic to the fundamental group of the boundary at infinity minus the limit points of the tube, which is isomorphic to $\mathbb Z$. The isomorphism arises from a closest point projection map, see Definition \ref{def:proj}. 

Another result important for the later parts of the paper is Theorem \ref{thm:cusp_tube_hyp}, which roughly speaking says that removing a large tube around bi-infinite geodesic and adding a combinatorial horoball on the corresponding shell results in a hyperbolic space; this space will serve as a ``local" model for other hyperbolic spaces we construct later on.

Two key technical constructions in this section are completed shells and completed tube complements, Definitions \ref{def:completed shell} and \ref{def:CTC}, which allow us for example to make sense of the cusped space described above.

We now fix some notation and hypotheses which we use throughout this section. The reader might want to keep in mind a space $\Xz$ acted upon geometrically by a hyperbolic group $G$, as will be the case in Section \ref{s:unwrap family} below.

\begin{assumption} \label{ass:X} 
Let $\Xz$ be a locally finite graph and suppose that $\Xz$ is $\delta_0$--hyperbolic and $\delta_0$--visible for some $\delta_0$.  

Let $\lambda_0 \ge 0$ be a constant, and let $\Yx$ denote a $\lambda_0$--quasi-convex subset of $\Xz$ such that $\partial \Xz \ssm \Lambda \Yx$ is path connected, where $\Lambda \Yx$ is the limit set of $\Yx$ in $\partial \Xz$.  Suppose further that if $y \in \Yx$ and $\alpha$ is a geodesic in $\Xz$ starting at $y$ and limiting to a point in  $\Lambda \Yx$ then $\alpha$ stays entirely in the $\lambda_0$--neighborhood of $\Yx$. Fix a basepoint $w_0 \in \Yx$.

Suppose further that some $\delta_0$--adapted visual metric $\rho_0$ on $\partial \Xz$ based at $w_0$ is $L_0$--linearly connected and $\doub$--doubling.
\end{assumption}

\subsection{Shells and projections}

\begin{definition} \label{def:shell}
Suppose that $Z$ is a metric space, $W \subset Z$, and  let $K > 0$ be an integer.

  The {\em $K$--shell of $W$ in $Z$} is $S_K^Z(W)=\{z\in Z \mid d(z,W)=K\}$, and the {\em (open) $K$--tube of $W$ in $Z$} is $T_K^Z(W) = \{z \in Z\mid d(z, W) < K\}$. 
  The \emph{(closed) $K$--neighborhood of $W$ in $Z$} is $N_K^Z(W) = \{ z \in Z \mid d(z,W) \le K \}$.
  
  Often, the space $Z$ is implicit, and we write $S_K(W)$, $T_K(W)$, and $N_K(W)$.
\end{definition}
We remark that if $Z$ is a graph and $W$ is a sub-graph, then $S_K^Z(W)$ is in the $0$--skeleton of $Z$, since $K$ is an integer. This will be important later.

Recall that the pair $(\Xz,\Yx)$ satisfies the conditions given in Assumption~\ref{ass:X}.  Recall also that we fixed $w_0\in\Yx$. Recall that the limit set of $\Yx$ in $\partial \Xz$ is denoted by $\Lambda \Yx$.

\begin{definition} \label{def:proj}
  For $K \geq 0$, define the map $\Pi_K=\Pi_{S_K(\Yx)} \co \partial \Xz\ssm \Lambda\Yx\to S_K(\Yx)$ as follows. For each $p\in\partial \Xz \ssm \Lambda\Yx$, choose a representative ray $\ray{p}$ from $w_0$ to $p$.  Then $\Pi_K(p)$ is the last point on $\ray{p}$ at distance $K$ from $\Yx$.
\end{definition}
Note that the map $\Pi_K$ involves some choices, such as $w_0$ and $\ray{p}$.  For large choices of $K$, Lemma \ref{lem:piK well-defined} shows that the choices don't matter much.  We use the chosen paths $\ray{p}$ later.

The next lemma shows that a geodesic ray starting on $\Yx$ which leaves $T_K(\Yx)$ does not stay near $S_K(\Yx)$ for very long.
\begin{lemma} \label{lem:straightgeodesic} Suppose $K \ge \lambda_0 + 3\delta_0$.  Let $\alpha$ be a geodesic ray from $x\in\Yx$ tending to a point in $\partial \Xz \ssm \Lambda\Yx$.  Let $y\in \alpha \cap S_K(\Yx)$.  If $z$ is any point on $\alpha$ with $d(z, S_K(\Yx)) \le 2 \delta_0$, then $d(z,y) \le 6 \delta_0$. \end{lemma} 

\begin{proof} Let $u$ be a point of $\alpha$ so that $z$ and $y$ are both between $u$ and $x \in \Yx$.  Project $u$ to $\Yx$ and call the projection $u'$.  Let $\Delta$ be a geodesic triangle with one side that part of $\alpha$ from $x$ to $u$, and third vertex $u'$.  Since the side $xu'$ lies within $\lambda_0$ of $\Yx$, and $K \ge \lambda_0 + 3\delta_0$, $z$ and $y$ are at most $\delta_0$ from points $z'$ and $y'$ on the $uu'$ side of this triangle.  However, the geodesic $uu'$ travels as quickly to $\Yx$ as possible, so since $|d(z',\Yx) - K| \le 3\delta_0$ and $| d(y,\Yx) - K| \le \delta_0$, we have $d(z',y') \le 4\delta_0$.
Thus $d(z,y) < \delta_0 + \delta_0 +4 \delta_0$, as required.  
\end{proof} 

\begin{lemma}\label{lem:piK well-defined}
  For any $K \ge \lambda_0 + 3\delta_0$, and any $p \in \partial \Xz \smallsetminus \Lambda\Yx$, changing the ray $\ray{p}$ or the basepoint $w_0$ in Definition~\ref{def:proj} changes the location of $\Pi_K(p)$ by at most $8\delta_0$.
\end{lemma}

\begin{proof} Let $w$ and $w'$ be two arbitrary points on $\Yx$.  Let $g$ and $g'$ be two geodesics from $w$ and $w'$ limiting at $p$.  There is a partially ideal triangle with vertices $w,w',p$, and sides $g,g'$ and a geodesic $[w,w']$.  Let $y$ and $y'$ be the points on $g$ and $g'$ at distance $K$ from $\Yx$. Since $K > \lambda_0 +3\delta_0$, the point $y'$ on $g'$ lies within $2 \delta_0$ of some point $y''$ on $g$.  By Lemma \ref{lem:straightgeodesic}, this point $y''$  is within $6 \delta_0$ of $y$. 
\end{proof}
\subsection{The coarse topology of shells}
Throughout this subsection, we let $S_K(\Yx)$ denote $S_K^{\Xz}(\Yx)$, and similarly drop the superscript for tubes. For our projections, we continue to work with a fixed $w_0$ on $\Yx$. 

The following is well-known, we include an argument for completeness.

\begin{lemma} \label{lem:tubequasiconvex} 
For any $K \ge \lambda_0$ the tube $T_K(\Yx)$ is $2\delta_0$--quasi-convex.
\end{lemma}

\begin{proof}
 Let $x,y\in T_K(\Yx)$ and consider a geodesic $[x,y]$ between them. Also, let $x',y'\in\Yx$ be closest points to $x,y$ respectively. Consider any $z\in[x,y]$, and we have to show that it lies within $2\delta_0$ of $T_K(\Yx)$.
 By hyperbolicity, $z$ lies $2\delta_0$ close to a geodesic $[x,x'],[y,y'],$ or $[x',y']$. The first two geodesics are contained in $T_K(\Yx)$, while the third one is contained in $T_{\lambda_0}(\Yx)\subseteq T_K(\Yx)$, so we are done.
\end{proof}

\begin{lemma}\label{lem:project_from_infty}
For any $K \ge \lambda_0 + 3\delta_0$, the map $\Pi_K$ is $8\delta_0$--coarsely surjective. Furthermore, any $p\in \partial \Xz\ssm \Lambda\Yx$ has an open neighborhood $U_p$ so that for every $q\in U_p$ we have $d(\Pi_K(p),\Pi_K(q))\leq 8\delta_0$.
\end{lemma}
\begin{proof}
 By $\delta_0$--visibility, any point $x$ on $S_K(\Yx)$ lies $\delta_0$--close to a geodesic ray $r$ from $w_0$ to some $q\in\partial \Xz$. It follows that $x$ is within $\delta_0$ of a point $y$ on $r$.  Let $\ray{q}$ be the geodesic ray used to define $\Pi_K(q)$ starting at $w_0$. We can form a geodesic triangle between $w_0$, and two points on $r$ and $\ray{q}$ which are far away from $S_K(\Yx)$.  Then $y$ is distance at most $\delta_0$ from a point $z$ on $\ray{q}$.  It follows from Lemma \ref{lem:straightgeodesic} that $x$ is within $8 \delta_0$ of $\Pi_K(q)$.   This proves the first assertion.
 
 Let $p$ be in $\partial \Xz \ssm \Lambda\Yx$ and consider the geodesic $\ray{p}$  used to define $\Pi_K(p)$.  Pick a neighborhood $U_p$ of $p$ in $\partial \Xz$ such that any geodesic from $w_0$ to a point $q$ in $U_p$ fellow travels $\ray{p}$ until both geodesics are a distance at least $2K$ from $\Yx$.  Since (partially) ideal triangles are $2\delta_0$--thin, $\Pi_K(q)$ is within $2 \delta_0$ of a point on $\ray{p}$. Then using Lemma \ref{lem:straightgeodesic}, $d\left( \Pi_K(p), \Pi_K(q) \right) \le 8\delta_0$. 
\end{proof}

\begin{definition}
Let $(A,d)$ be a metric space.  An \emph{$s$--path in $A$} is a function $f\co \{0,1,\ldots,n_f\}\to A$ so that $d(f(i-1),f(i))\leq s$ for each $i\in\{1,\ldots,n_f\}$. The number $\sum\limits_{i=1}^{n_f} d\left( f(i-1), f(i)\right)$ is called its \emph{length}. 

We call a space \emph{$s$--path-connected} if every pair of points is joined by an $s$--path.
\end{definition}

Notice that for any $x, y \in A$ and any $s > 0$ the length of any $s$--path between $x$ and $y$ is at least $d(x,y)$. 

\begin{lemma}\label{lem:sphere_coarse_conn}
For any $K \ge \lambda_0 + 3\delta_0$ the shell $S_K(\Yx)$ is $8\delta_0$--path connected.\end{lemma}

\begin{proof}
Let $x, y \in S_K(\Yx)$.  By the first part of Lemma \ref{lem:project_from_infty}, there exist $p, q \in \partial \Xz \smallsetminus \Lambda\Yx$ so that $d(x,\Pi_K(p)), d(y,\Pi_K(q)) \le 8\delta_0$.  Since $\partial \Xz\ssm \Lambda\Yx$ is path-connected, we can consider a path from $p$ to $q$ in it. We can then cover the path with finitely many open neighborhoods as described in the second part of Lemma \ref{lem:project_from_infty}, and apply $\Pi_K$ to obtain a $8\delta_0$--path between $\Pi_K(p)$ and $\Pi_K(q)$.
\end{proof}

The following result is a more precise version of \cite[Lemma 6.21]{GMS}, and we explain how to adapt the proof of \cite{GMS}.
Recall that $\epsilon = \epsilon(\delta_0)$ and $\kappa$ are the constants from Definition \ref{def:delta_adapted}, and that $\partial \Xz$ is equipped with a $\delta_0$--adapted visual metric $\rho_0$ based at $w_0$ which is $L_0$--linearly connected and $\doub$--doubling.
\begin{lemma} \label{lem:GMS}
For any $B > 0$ there exists $N=N(\delta_0,L_0,\doub,B)$ such that for every pair of rays $\eta_1, \eta_2$ based at $w_0$ with different endpoints in $\partial \Xz$ there exists $n\leq N$, and a sequence of rays $\eta_1= \alpha_1, \ldots , \alpha_n = \eta_2$ starting at $w_0$ with $(\alpha_i | \alpha_{i+1})_{w_0} \geq (\eta_1 | \eta_2)_{w_0} + B$ and $(\alpha_i | \eta_1)_{w_0} \geq (\eta_1 | \eta_2)_{w_0} -\Delta_0$,
where $\Delta_0 = \frac{\log(2\kappa^2L_0)}{\epsilon} + 20\delta_0$.
\end{lemma} 

\begin{proof}
The statement of the lemma differs from that of \cite[Lemma 6.21]{GMS} in two ways. First, the constant $R$ (which is our $\Delta_0$) in the statement of \cite[Lemma 6.21]{GMS} is allowed to depend on the hyperbolic space rather than just on $\delta_0$ and $L_0$. Second,  \cite[Lemma 6.21]{GMS} provides no bound on $n$.

Regarding the first difference, the explicit constant $\Delta_0$, which depends only on $L_0$ and $\delta_0$, is actually given in the proof in \cite{GMS} (third line). 

Regarding the bound on $n$, we first outline the proof of \cite[Lemma 6.21]{GMS}. First of all, in said proof the limit points $p_1,p_2$ of $\eta_1,\eta_2$ get connected by an arc $I$ of diameter at most $2L_0\rho_0(p_1,p_2)$. The sequence of rays $\alpha_i$ is then obtained as any chain of rays with limit points $a_i$ in $I$ where the $a_i$ form a chain that starts at $p_1$, ends at $p_2$, and any two consecutive points are sufficiently close. Explicitly, we need

$$\rho(a_i,a_{i+1})\leq e^{-\epsilon B}\rho(p_1,p_2)/\kappa^2,$$

which guarantees
$$e^{-\epsilon (\alpha_i | \alpha_{i+1})_{w_0}} \leq \kappa \rho(a_i,a_{i+1})\leq e^{-\epsilon B} \rho(p_1,p_2)/\kappa \leq  e^{-\epsilon B}e^{-\epsilon  (\eta_1 | \eta_2)_{w_0}},$$

and in turn this yields the required $(\alpha_i | \alpha_{i+1})_{w_0} \geq (\eta_1 | \eta_2)_{w_0} + B$.

To get the desired bound on $n$, we observe that subspaces of an $\doub$--doubling space are $\doub^2$--doubling (the intersection of a ball with a subspace is contained in a ball of twice the radius in the induced metric). From this, we see that $I$ can be covered by at most $\doub^{2\lceil\log_2(4L_0e^{-\epsilon B}/(2\kappa^2))\rceil}$ balls of radius $e^{-\epsilon B}\rho(p_1,p_2)/(2\kappa^2)$. Therefore, we can extract a chain of points with the required bound on the length from the centers of these balls.
\end{proof}

  Recall the hyperbolicity constant $\delta_0$ and the linear connectedness constant $L_0$ were fixed in Assumption~\ref{ass:X}, as was the quasi-convexity constant $\lambda_0$ of $\Yx$.  In the following we also fix the constant $\Delta_0$ from Lemma~\ref{lem:GMS}.

\begin{lemma}\label{lem:proper_dist}
  There is a function $\Phi\co \R_{\ge 0}\to \R_{\ge 0}$ so that for any $C >0$ and $K \ge C + \Delta_0 + \lambda_0 + 3\delta_0$, if $x,y \in S_K(\Yx)$ satisfy $d(x,y) \le C$ then there exists an $8\delta_0$--path in $S_K(\Yx)$ from $x$ to $y$ of length at most $\Phi(C)$.  
  \end{lemma}

\begin{proof}  
We use the same idea as in Lemma \ref{lem:sphere_coarse_conn}.  Let $C > 0$, and suppose that $K > C + \Delta_0 + \lambda_0 + 3\delta_0$.  Let $x,y \in S_K(\Yx)$ satisfy $d(x,y) \le C$.  Lemma~\ref{lem:project_from_infty} says that $\Pi_K$ is $8\delta_0$--coarsely surjective, so by moving $x$ and $y$ at most $8\delta_0$, we may assume that $x = \Pi_K(p)$ and $y=\Pi_K(q)$, for some $p, q \in \partial \Xz \smallsetminus \Lambda\Yx$.

Let $\ray{p}$ and $\ray{q}$ denote the geodesic rays from $w_0$ defining $\Pi_K(p)$ and $\Pi_K(q)$, respectively.  Note that because $x,y \in S_K(\Yx)$ and $d(x,y) \le C$, we know that there are points $y_1 \in \ray{p}$ and $y_2 \in \ray{q}$ which lie at distance at least $K-(1/2 C+\delta_0)$ from $\Yx$  and so that $d(y_1,y_2) \le \delta_0$. As in the conclusion of Lemma \ref{lem:GMS}, choose a sequence $\ray{p} = \alpha_1, \ldots, \alpha_n = \ray{q}$ of geodesic rays so that $(\alpha_i\mid \alpha_{i+1})_{w_0} \ge K + 10\delta_0$ for each $i$, and $(\alpha_i \mid \ray{p})_{w_0} \ge (\ray{p} \mid \ray{q})_{w_0} - \Delta_0$. By the same lemma, we can also require that $n$ is bounded in terms of $\delta_0,L_0,\doub,C$ only (as $(\ray{p} \mid \ray{q})_{w_0}$ is at least $K$ minus a constant depending on $C$ and $\delta_0$). It now follows that there exists a point $z_i$ on each $\alpha_i$ at distance at least $K - ((1/2) C + \delta_0) - \Delta_0 \ge \lambda_0 + 2\delta_0$ from $\Yx$.  By the assumption on $\Yx$ in Assumption~\ref{ass:X}, any geodesic ray from $w_0$ which leaves the $\lambda_0$--neighborhood of $\Yx$ cannot limit on $\Lambda\Yx$.  Thus, the sequence of rays $\alpha_i$ each limit on a point $p_i$ of $\partial \Xz \smallsetminus \Lambda\Yx$.

As in the proof of Lemma \ref{lem:sphere_coarse_conn}, the points $x =  \Pi_K(p_1), \ldots , \Pi_K(p_n) = y$ form a $8\delta_0$--path.
\end{proof}

\subsection{Completed shells}

\begin{definition} \label{def:completed shell}
Let $\Gamma$ be a graph, and suppose $\Xi \subset \Gamma$ is a subgraph.  Let $K,s$ be positive integers, and let $S^\Gamma_K(\Xi)$ be the $K$--shell about $\Xi$ in $\Gamma$ (which consists of vertices).  The {\em completed $K$--shell about $\Xi$ in $\Gamma$ at scale $s$}, denoted $\CS^\Gamma_{K,s}(\Xi)$ is the graph obtained from $S^\Gamma_K(\Xi)$ by joining each pair of  points $a,b \in S^\Gamma_K(\Xi)$ so that $d_\Gamma(a,b) \le s$ with an edge-path of length $d_\Gamma(a,b)$.
\end{definition}

\begin{convention}\label{conv:CS}
In applying Definition \ref{def:completed shell} above, 
we always use $s = 8\delta_0$.  Thus, we use $\CS_K^\Gamma(\Xi)$ to denote $\CS_{K,8\delta_0}^\Gamma(\Xi)$.

Our default space is $\Xz$, so we use $\CS_K(\Xi)$ to denote $\CS_K^{\Xz}(\Xi) = \CS_{K,8\delta_0}^{\Xz}(\Xi)$.
\end{convention}

The following is an immediate corollary of Lemma \ref{lem:sphere_coarse_conn}.
\begin{lemma} \label{lem:shcon}
If $K \ge \lambda_0 + 3\delta_0$ the completed shell $\CS_K(\Yx)$ is connected.
\end{lemma}

\begin{definition} \label{def:CTC}
Let $\Gamma$ be a graph and $\Xi \subset \Gamma$ a subgraph.  Let $K,s $ be positive integers. The {\em completed $K$--tube complement of $\Xi$ in $\Gamma$ at scale $s$}, denoted $\CTC^\Gamma_{K,s}(\Xi)$ is the space
\[	\CTC^\Gamma_{K,s}(\Xi) = \left( \Gamma \ssm T^\Gamma_K(\Xi) \right) \sqcup_{S^\Gamma_K(\Xi)} \CS^\Gamma_{K,s}(\Xi)	,	\]
obtained from `completing' the copy of $S^\Gamma_K(\Xi)$ in $\Gamma \ssm T^\Gamma_K(\Xi)$ at scale $s$.
\end{definition}
Note that $\CTC^\Gamma_{K,s}(\Xi)$ and $\CS^\Gamma_{K,s}(\Xi)$ are graphs.

\begin{convention}\label{conv:CTC}
As in Convention~\ref{conv:CS}, when completing complements of tubes about $\Yx$ in $\Xz$, we always use $s = 8\delta_0$ and denote the space $\CTC^\Gamma_{K,s}(\Xi)$ by $\CTC^\Gamma_K(\Xi)$, and when $\Gamma = \Xz$ we omit the superscript and simply write $\CTC_K(\Xi)$.
\end{convention}

\begin{lemma} \label{lem:CTC connected}
Suppose that $K \ge \lambda_0 + 3\delta_0$.  Then $\CTC_K(\Yx)$ is a connected graph. 
\end{lemma}
\begin{proof}
By Lemma \ref{lem:shcon}, $\CS_K(\Yx)$ is connected, so we only need to show that every point of $\CTC_K(\Yx)$ is connected to the completed shell.  Take a point in $ \Xz \ssm T^Z_K(\Yx)$.  It is connected to $\Yx$ by a path in $\Xz$  (since $\Xz$ is a connected graph). This path hits $S^{\Xz}_K(\Yx) \subset \CS_K(\Yx)$. 
\end{proof}

\begin{convention}
Whenever we talk about a completed tube complement, we ensure that Lemma \ref{lem:CTC connected} applies, and the implied metric is the path metric.
\end{convention}

\subsection{Geometry of cusped spaces} 
 The main result in this subsection is Theorem~\ref{thm:cusp_tube_hyp}, which gives us control over the geometry of a space obtained by attaching a horoball to a $K$--shell in a tube complement.

\begin{definition} \label{def:projection}
Suppose $Z$ is a metric space, $W \subset Z$ and $z \in Z$.  Then $P_W(z)$ denotes a point in $W$ which is closest to $z$ (amongst all points in $W$).
\end{definition}

\begin{remark}All of our spaces are locally compact and $\delta$--hyperbolic, and our subsets $W$ are closed, so such points $P_W(z)$ exist, though they typically are not unique.  Nothing that we do depends on the particular choice of $P_W(z)$. 
\end{remark}

\begin{lemma} \label{lem:little}  Let $\Upsilon$ be a $\delta$--hyperbolic metric space and $A$ a $\lambda$--quasi-convex subset.  Suppose $K \ge 4\delta + \lambda$ and let $y,z \in \Upsilon$, and suppose there is a geodesic $[z,y] \subset \Upsilon \ssm N_K(A)$.  Suppose $y' \in [y,P_A(y)] \cap S_K(A)$ and $z' \in [z,P_A(z)] \cap S_K(A)$ (for some choices of geodesics $[y,P_A(y)]$, $[z,P_A(z)]$).  Then $d(y',z') < 8 \delta$. 
\end{lemma} 

\begin{proof}
  Choose a geodesic $[P_A(y),P_A(z)]$ and consider the geodesic quadrilateral
  $[y, P_A(y)]\cup[P_A(y),P_A(z)]\cup[P_A(z),z]\cup[y, z]$.
  Since $K \ge 4 \delta + \lambda$, every point of $[y, z]$ is distance at least $4 \delta$ from $[P_A(y),P_A(z)]$.  Choose a geodesic $[P_A(y),z]$ and consider the geodesic triangle $[y, P_A(y)]\cup[y, z]\cup[P_A(y), z]$.  Since $y' \in [y, P_A(y)]$, $y'$ is within $\delta$ of either $[y,z]$ or $[P_A(y),z]$.  If $y'$ is within $\delta$ of some point $q$ on $[P_A(y),z]$ then $q$ is not within $\delta$ of $[P_A(y),P_A(z)]$, so it is distance $\delta$ from a point $q'$ on $[z, P_A(z)]$.  We have $|d(q',A)-K| \le 2\delta$, since $d(y',A) = K$.  Hence the distance from $q'$ to $z'$ is at most $2\delta$.  So $d(y',z') < 4 \delta$ in this case.

  Therefore, we assume that $y'$ is within $\delta$ of some point  $u\in [y, z]$.  By a symmetric argument we can assume that $z'$ is within $\delta$ of some  $v\in [y, z]$.
  If $u=v$ we are done.  We assume that the order along $[y,z]$ is $y,u,v,z$ (else it follows easily that $d(y',z') < 4 \delta$).

  Using $\lambda$--quasi-convexity, $[y, z]$ lies outside of the $4 \delta$--neighborhood of $[P_A(y), P_A(z)]$, so every point of the subsegment $[u, v]\subseteq [y,z]$ lies within $2 \delta$ of $[y, P_A(y)] \cup [z, P_A(z)]$.  Let $p$ be a point on $[u,v]$ which is distance at most $2\delta$ from some point $y_p$ on $[y,P_A(y)]$ and distance at most $2 \delta$ from a point $z_p$ on $[z, P_A(z)]$. 

 If $z_p$ is  closer to $A$ than $z'$, it is at most $2 \delta$ closer, so $d(z', z_p) < 2 \delta$.  Then $d(z',p) < 4 \delta$.  If $z_p$ is further away from $A$ than $z'$, then $z_p$ is on $[z', z]$.  Consider the triangle $[z',z],[z,v],[v,z']$, where the length of $[v,z']$ is at most $\delta$.  If $z_p$ is $\delta$--close to $[v, z']$, then $d(z_p,z') < 2 \delta$ and $d(z',p)< 4 \delta$.  If $z_p$ is $\delta$--close to $[z,v]$, $v$  is at most $3 \delta$ from $p$, so again $d(z',p) < 4 \delta$.  So in any case $d(z',p)< 4 \delta$.  A similar argument shows $d(y',p) < 4 \delta$, so $d(y',z') < 8 \delta$, as required.  
\end{proof}

\begin{lemma}\label{lem:constant_progress}
  Let $\Upsilon$ be a $\delta$--hyperbolic space, let $A$ be a $\lambda$--quasi-convex subset of $\Upsilon$, and suppose $\alpha$ is a geodesic ray starting at $w\in A$ and not converging to a point in $\Lambda A$.  For any $y,z\in \alpha$ which are distance greater than $6\delta+\lambda$ from $A$, we have
  \[ d(y,z) \sim_{10\delta} |d(y,A)-d(z,A)| .\]
\end{lemma}
\begin{proof}
  Without loss of generality we assume that $d(y,w)>d(z,w)$.
  
  An easy argument shows that under the hypotheses, the subsegment of $\alpha$ joining $z$ to $y$ lies entirely outside the $(5\delta+\lambda)$--neighborhood of $A$, so we may apply Lemma~\ref{lem:little} with $K = 4\delta+\lambda$.  
  Let $y''=P_A(y)$, $z''=P_A(z)$, and let $y' \in [y,y''] \cap S_K(A)$, $z' \in [z,z''] \cap S_K(A)$.

  Another easy argument using quasi-convexity of $A$ shows $(w\mid y)_{y''}$ and $(w\mid z)_{z''}$ are bounded above by $\lambda+\delta$.  Since $K>\lambda+\delta$, there are points $\widehat{y},\widehat{z}$ on $\alpha$ within $\delta$ of $y',z'$, respectively, and satisfying $d(y,\widehat{y})=d(y,y')$, $d(z,\widehat{z})=d(z,z')$.  Lemma~\ref{lem:little} implies that $d(\widehat{y},\widehat{z})\le 10\delta$.  So in particular we have $d(z,\widehat{z})\sim_{10\delta}d(z,\widehat{y})$, and so (since  $d(y,A) = d(y,y')+d(y',y'') = d(y,\widehat{y}) + K$, and similarly for $d(z,A)$)
  \[ d(y,A) - d(z,A) \sim_{10\delta} d(y,\widehat{y})-d(z,\widehat{y}) = d(y,z) .\]
  The last equality holds because $\widehat{y}$ lies between $w$ and $z$ on $\alpha$.  This is because $\widehat{y}$ is at most $5\delta+\lambda$ from $A$, but, as noted in the first paragraph, the subsegment of $\alpha$ joining $z$ to $y$ lies entirely outside the $(5\delta+\lambda)$--neighborhood of $A$.
\end{proof}

\begin{definition}\label{def:cusping} 
Let $K>\lambda_0+\delta_0$ and suppose $\Xz^{\mathrm{cusp}}(\Yx,K)$ is obtained by gluing a combinatorial horoball $\horbaD = \mathcal{H}(\CS_K(\Yx))$ based on $\CS_K(\Yx)$ to $\CTC_K(\Yx)$.  

We say that $\Xz^{\mathrm{cusp}}(\Yx,K)$ is obtained from $\Xz$ by {\em cusping the $K$--shell around $\Yx$}. Note that this depends on the set $\Yx$, which is fixed in Assumptions \ref{ass:X}. Therefore, we drop the $\Yx$ from our notation and refer simply to $\Xz^{\mathrm{cusp}}(K)$.\end{definition} 

The hyperbolicity constant of $\horbaD= \mathcal{H}(\CS_K(\Yx))$ does not depend on $K$.

For each pair $(x,y)$ of points in $\Xz^{\mathrm{cusp}}(K)$, we define a path $\mc{P}(x,y)$ connecting points $x,y$.   For each $x$ outside the tube, choose once and for all some $x'\in S_K(\Yx)$ closest to $x$ and a geodesic $[x,x']$ in $\Xz$ from $x$ to $x'$.  Two such choices of $x'$ differ only by $\delta_0$.  If $x,y$ lie outside the horoball $\horbaD$, then let $[x,y]$ be a  geodesic in $\Xz$ from $x$ to $y$.  If $[x,y]$ does not intersect $T_K(\Yx)$, then set $\mc{P}(x,y)=[x,y]$.  Otherwise,  let $\mc{P}(x,y)$ be the union of the two geodesics $[x,x']$ and $[y,y']$ and a geodesic in $\horbaD$ connecting $x',y'$.  If $x,y$ both lie in $\horbaD$, then let $\mc{P}(x,y)$ just be a geodesic in $\horbaD$ connecting them.  If one of them, say $x$, lies outside $\horbaD$ and $y$ lies in $\horbaD$, then again consider a shortest geodesic $[x,x']$ from $x$ to (the closure of) $T_K(\Yx)$, and define $\mc{P}(x,y)$ to be the concatenation of $[x,x']$ and a geodesic in $\horbaD$ from $x'$ to $y$.

\begin{theorem}\label{thm:cusp_tube_hyp}
Under Assumptions~\ref{ass:X} there exists $\delta_1=\delta_1(\delta_0)$ so that for any $K \ge \lambda_0 + 4\delta_0$ the space $\Xz^{\mathrm{cusp}}(K)$ is $\delta_1$--hyperbolic and $\delta_1$--visible.
\end{theorem}

\begin{proof}
 We prove that $\Xz^{\mathrm{cusp}}(K)$ is hyperbolic by applying Proposition \ref{prop:Bowditch guess} to the paths $\mc{P}(x,y)$.

Suppose $d_{\Xz^{\mathrm{cusp}}(K)}(x,y)=1$. If $x,y$ lie outside $\horbaD$, then so does the edge connecting them, and $\mc{P}(x,y)$ is an edge.  This is also the case if $x,y$ both lie in $\horbaD$.  In the mixed case, the point lying in $\horbaD$ actually lies in the boundary horosphere, and again $\mc{P}(x,y)$ is a single edge.

To apply Proposition~\ref{prop:Bowditch guess} we must show that the $\mc{P}$--triangles are uniformly slim.  We consider the case of vertices $x,y,z$ all lying outside $\horbaD$, the other cases being similar.  Suppose that none of $\mc{P}(x,y),\mc{P}(x,z)$ and $\mc{P}(y,z)$ intersects $\horbaD$.  Slimness of the corresponding $\mc{P}$--triangle follows from hyperbolicity of $\Xz$ and the following claim:

\begin{claim} \label{claim:closemetrics} For any $a,b\in \Xz \ssm T_K(\Yx)$ if $d_{\Xz}(a,b)\leq 8\delta_0$, then $d_{\Xz^{\mathrm{cusp}}(K)}(a,b)\leq 8\delta_0$.
\end{claim} 
\begin{proof}
 There is a geodesic $\alpha$ in $\Xz$ of length at most $8\delta_0$ joining $a,b$. Every maximal subpath of $\alpha$ that is contained in the closure of $T_K(\Yx)$ has endpoints at distance at most $8\delta_0$ from each other in $\Xz$.  Since $s = 8\delta_0$ (see Convention \ref{conv:CS}) the distance between these endpoints in the completed shell (and in $\Xz^{\mathrm{cusp}}(K)$) is no more than the distance in $\Xz$. \end{proof}

Suppose that all sides $\mc{P}(x,y),\mc{P}(x,z)$ and $\mc{P}(y,z)$ intersect $\horbaD$, then we can decompose the triangle into
\begin{itemize}
 \item 3 pairs of geodesics in $\Xz$ that coincide, and
 \item a geodesic triangle in $\horbaD$.
\end{itemize}
From hyperbolicity of $\horbaD$, we get the desired conclusion (letting $\delta_1$ be bigger than this hyperbolicity constant).  

Let us now consider the case that $\mc{P}(x,y),\mc{P}(x,z)$ intersect $\horbaD$, but $\mc{P}(y,z)$ does not.  Let $x',y', z'$ be closest points to $x,y,z$ in the closure of $T_K(\Yx)$ (chosen as in the definition of $\mc{P}$).  In this case we decompose the triangle into
\begin{itemize}
 \item a pair of coinciding geodesics starting at $x$ ($[x,x']$),
 \item geodesics in $\horbaD$ from $x'$ to $y'$ and from $x'$ to $z'$,
 \item a geodesic $[y,z]$ and geodesics $[y,y'],[z,z']$.
\end{itemize}

Notice that, since $[y,z]$ does not intersect $T_K(\Yx)$,  we have $d_{\Xz}(y',z')\leq 8\delta_0$ by Lemma \ref{lem:little}.  Hence, hyperbolicity of $\horbaD$ says that the geodesics in the second item fellow-travel with uniform constant. By considering the quadrilateral with vertices $y,y',z,z'$, and using the fact that $y' = P_{S_K(Y_0)}(y)$ and $z' = P_{S_K(Y_0)}(z)$, we see any point on a geodesic from the third item is within $\Xz$--distance $10\delta_0$ of the union of the other two.  So we are done in this case as well, since the $\Xz^{\mathrm{cusp}}(K)$--distance is bounded in terms of this in Claim \ref{claim:closemetrics} 
The case that only one of the sides, say $\mc{P}(x,y)$ intersects $\horbaD$ is similar, except that in this case we have the bound $d_{\Xz^{\mathrm{cusp}}}(x',y')\leq d_{\Xz^{\mathrm{cusp}}}(x',z')+d_{\Xz^{\mathrm{cusp}}}(z',y')\leq 16\delta_0$.

So far we not only proved hyperbolicity of $\Xz^{\mathrm{cusp}}(K)$, but we also have a description of paths $\mc{P}(\cdot,\cdot)$ that stay close to geodesics. 

It remains to show visibility. In order to show it, we prove that there exists a constant $v=v(\delta_0)$ such that every $p\in \Xz^{\mathrm{cusp}}(K)$ lies $v$--close to a bi-infinite $(v,v)$--quasi-geodesic. This suffices since this implies that we can ``prolong'' any geodesic using one of the two rays from a geodesic line passing near its endpoint.

We first claim that there exists $v_0 = v_0(\delta_0)$ such that given any $x\in \Xz \ssm T_K(\Yx)$ there is an $\Xz$ geodesic ray $\alpha(x)$ in $\Xz$ starting on $S_K(\Yx)$, passing $v_0$--close to $x$, and intersecting $S_K(\Yx)$ only at its starting point. To see this, note that coarse surjectivity of $\Pi_{K'}$ for $K'$ such that $x$ lies in the $K'$--shell (Lemma \ref{lem:project_from_infty}) says that there exists an $\Xz$--ray starting on $\Yx$ and passing $8\delta_0$--close to $x$. We claim that a subray intersecting $S_K(\Yx)$ only at its starting point $q$ is a uniform quasi-geodesic in $\Xz^{\mathrm{cusp}}(K)$. This is because, by Lemma \ref{lem:constant_progress}, for $q'$ on the ray, the length of the subpath of the ray from $q$ to $q'$ is coarsely equal to the distance in $\Xz$, hence coarsely equal to the distance in  $\Xz^{\mathrm{cusp}}(K)$, from $q'$ to $S_K(\Yx)$. (We apply Lemma \ref{lem:constant_progress} to a smaller tube to get this, and we apply it in $\Xz$.) This implies that there cannot be ``shortcuts'' from $q$ to $q'$ in $\Xz^{\mathrm{cusp}}(K)$.

Consider first the case that $p$ lies in $\horbaD$. The required geodesic line consists of a vertical ray in $\horbaD$ within distance 1 of $p$, and a geodesic ray $\alpha(x)$ where $x\in S_K(\Yx)$ is the starting point of the vertical ray. The fact that these two rays form a quasi-geodesic line is again an application of Lemma \ref{lem:constant_progress}. If $p$ does not lie in $\horbaD$, we concatenate $\alpha(p)$ and a vertical ray in $\horbaD$.
\end{proof}

\subsection{Fundamental groups of completed shells}\label{ss:noname}
The purpose of this subsection is to understand the (coarse) fundamental groups of completed shells.

We first define a map from the fundamental group of $\partial \Xz\ssm\Lambda \Yx$ to the coarse fundamental group of a completed shell, in the proof of Proposition \ref{prop:homomorphism} below.  We do this by projecting paths using $\Pi_K$.  Since $\Pi_K$ is not continuous, it does not literally send paths to paths, but by sampling a path in $\partial \Xz \ssm \Lambda \Yx$ sufficiently densely, we obtain a coarse path in the shell.  ``Connecting the dots'' in the completed shell gives an honest path, which is well defined up to homotopy.  We formalize this process of projection in the following definition.

\begin{definition}\label{def:project_paths}
  Let $\eta: [a,b]\to \partial\Xz\ssm\Lambda \Yx$ be a path, and let $K\ge \lambda_0 + 3\delta_0$, so we can apply Lemma~\ref{lem:project_from_infty}.
  A \emph{projection of $\eta$ to $\CS_K(\Yx)$} is any path $\alpha$ constructed by the following procedure:  Let $a=t_0,\ldots,t_n=b$ be a partition so that each $\eta([t_i,t_{i+1}])$ lies in an open neighborhood $U_p$ as in Lemma~\ref{lem:project_from_infty}.  For each $i$ let $z_i = \Pi_K(\eta(t_i))$, let $\gamma_i$ be a geodesic segment in $\CS_K(\Yx)$ joining $z_i$ to $z_{i+1}$, and let $\alpha$ be the concatenation $\gamma_0\cdots\gamma_{n-1}$.
\end{definition}
The projection of $\eta$ is not unique, but the next lemma says that any two projections are homotopic rel endpoints.

\begin{lemma}\label{lem:welldef_projected_path}
  Let $(\Xz,\Yx)$ be as in Assumption~\ref{ass:X}.  For any $K \ge \lambda_0 + 3\delta_0 $ and any $D \ge 24\delta_0$, the following holds:
  Let $\eta$ be a path in $\partial\Xz\ssm\Lambda \Yx$, and let $\alpha$ and $\alpha'$ be two projections of $\eta$ to $\CS_K(\Yx)$.  Then $\alpha$ and $\alpha'$ are homotopic rel endpoints in $(\CS_K(\Yx))^D$.

  Moreover, the paths $\alpha$ and $\alpha'$ are Hausdorff distance at most $24\delta_0$ from one another, with respect to the path metric on $\CS_K(\Yx)$.
\end{lemma}
\begin{proof}
  We use the notation of Definition~\ref{def:project_paths}, adding a prime to objects used in defining $\alpha'$.
  Suppose first that $\alpha$ and $\alpha'$ arise from the same partition of the domain of $\eta$, so $t_i = t_i'$ for each $i$.  For each $i$, the segments $\gamma_i$ and $\gamma_i'$ form a loop of length at most $16\delta_0 < D$ so they are homotopic rel endpoints in $(\CS_K(\Yx))^D$.  Thus $\alpha$ and $\alpha'$ are homotopic rel endpoints in $(\CS_K(\Yx))^D$.

  Now suppose that $\alpha'$ arises from a refinement of the partition for $\alpha$, say by adding a single point $t_{i+1}' \in (t_i,t_{i+1})$.  The geodesics $\gamma_i$, $\gamma_i'$ and $\gamma_{i+1}'$ form a geodesic triangle of length at most $24\delta_0\le D$, so again $\alpha$ and $\alpha'$ are homotopic rel endpoints in $(\CS_K(\Yx))^D$.

  The general case follows readily.

  For the statement about Hausdorff distance, note that $\alpha$ and $\alpha'$ come from partitions with a common refinement.  If $\alpha''$ is a projection using this refinement, it is straightforward to see that $\alpha''$ is Hausdorff distance at most $12\delta_0$ from either $\alpha$ or $\alpha'$.
\end{proof}

\begin{proposition}\label{prop:homomorphism}
    Let $(\Xz,\Yx)$ be as in Assumption~\ref{ass:X}.  For any $K\ge \lambda_0+3\delta_0$ and any $D\ge 24\delta_0$, and any point $\ast\in \partial\Xz\ssm \Lambda \Yx$, the projection $\Pi_K$ induces a well-defined homomorphism \[\Pi_\ast: \pi_1(\partial \Xz\ssm\Lambda \Yx,\ast)\to \pi_1^{D}(\CS_K(\Yx),\Pi_K(\ast)).\]
\end{proposition}

\begin{proof} 
  Let $\eta: [0,1] \rightarrow \partial \Xz \ssm \Lambda \Yx$ be a loop based at $\ast$.  We define $\Pi_*([\eta]) = [\alpha]$ where $\alpha$ is any projection of $\eta$ to $\CS_K(\Yx)$.  Lemma~\ref{lem:welldef_projected_path} tells us this is a well-defined element of $\pi_1^D(\CS_K(\Yx),\Pi_K(\ast))$.  
	
  To show that $\Pi_*([\eta])$ is independent of the representative $\eta$, we argue as follows:  Suppose $\eta,\zeta$ represent the same element of $\pi_1(\partial \Xz\ssm \Lambda \Yx,\ast)$.	The image of a homotopy from $\eta$ to $\zeta$ is a compact subset of $\partial \Xz\ssm \Lambda \Yx$, so we can cover it with finitely many neighborhoods $U_p$ as in Lemma~\ref{lem:project_from_infty}.  Now choose a sufficiently fine triangulation $T$ of $I \times I$ so that the image of every triangle is contained in one of the $U_p$.  Take projections to $\CS_K(\Yx)$ of all the edges of this triangulation.  Each edge projects to a path of length at most $8\delta_0$, so the boundary of each triangle projects to a loop of length at most $24\delta_0\le D$.  Putting these together gives a homotopy (rel $\Pi_K(\ast)$) in $(\CS_K(\Yx))^D$ from a projection of $\eta$ to a projection of $\zeta$.  In particular these projections represent the same element of $\pi_1^D(\CS_K(\Yx),\Pi_K(\ast))$.

Finally, we show that $\Pi_\ast$ is a homomorphism.  Let $\eta_1$ and $\eta_2$ be loops in $\partial \Xz \ssm \Lambda \Yx$ based at $\ast$.
If we ensure $\frac{1}{2}$ is a vertex in the subdivision of $[0,1]$ then we may obtain a projection of $\eta_1\cdot\eta_2$ which is a concatenation $\alpha_1\cdot\alpha_2$, where $\alpha_i$ is a projection of $\eta_i$ for $i\in \{1,2\}$.  
From this it follows that $\Pi_{\ast}$ is a homomorphism.
\end{proof}

\begin{definition} \label{def:center}
    Suppose that $w \in \Xz$ and that $\alpha$ and $\beta$ are rays in $ \Xz$ starting at $w$, ending at points $z,z' \in \partial\Xz$, respectively.  Let $T(\alpha,\beta)$ be the tripod obtained by gluing two rays together along an initial segment of length $(z \mid z')_w$.  Let $c(w,\alpha,\beta)$ be the point of $\alpha$ in the preimage of the trivalent point of $T(\alpha,\beta)$ (or the point $w$ if $(z\mid z')_w = 0$).  
\end{definition}
According to \cite[Lemma 3.7]{GMS}, there is a $(1,5\delta_0)$--quasi-isometry from $\alpha \cup \beta$ to $T(\alpha,\beta)$ which is isometric on each of $\alpha$ and $\beta$.   The utility of the center $c(w,\alpha,\beta)$ is that it lies at distance $(z\mid z')_w$ from $w$, and therefore can be used to estimate the visual distance between $z$ and $z'$ as viewed from $w$.

We now fix some more constants:
Let
    \begin{equation*}
        C_0 = 2\left(\frac{2\log\kappa + \log L_0}{\epsilon} + 29\delta_0\right)+20\delta_0,
    \end{equation*}
and let
 \[ D_0 =  3\Phi(C_0) + 108\delta_0, \]
        where $\Phi$ is the function from Lemma~\ref{lem:proper_dist}.

    \begin{lemma}\label{lem:claim_localpath}
      There exists $R_{\ref{lem:claim_localpath}}$ depending only on $\delta_0$, $\lambda_0$, and $L_0$,
      so that if $K\ge R_{\ref{lem:claim_localpath}}$, the following holds.
          Suppose that $x,y\in \partial \Xz\ssm \Lambda \Yx$ satisfy $d(\Pi_K(x),\Pi_K(y))\leq 24\delta_0$. There exists a path $\eta$ in $\partial \Xz\ssm \Lambda \Yx$ connecting $x$ and $y$, and such that for every $z\in \eta$ we have $d_{\Xz}(\Pi_K(x),\Pi_K(z))\leq C_0$.
    \end{lemma}
        \begin{proof}
Recall that $\rho_0$ is a $\delta_0$--adapted visual metric on $\partial \Xz$ based at $w_0\in \Yx$ with constants $\kappa, \epsilon = \epsilon(\delta_0)= \frac{1}{6\delta_0}$ (see Definition~\ref{def:delta_adapted}). That is, for all $x,y\in \partial \Xz$ we have
	\begin{align}
		\label{e:d1}
		\kappa^{-1}e^{-\epsilon(x|y)_{w_0}}\le \rho_0(x,y) \le  \kappa e^{-\epsilon(x|y)_{w_0}}.
	\end{align}
  We prove the lemma with
 \[   R_{\ref{lem:claim_localpath}} = \left(\frac{2\log\kappa + \log L_0}{\epsilon} \right) + 40\delta_0 + \lambda_0 .\]
        Accordingly, assume $K \ge  R_{\ref{lem:claim_localpath}}$.
          By linear connectivity, there exists a path $\eta$ in $\partial \Xz$ connecting $x$ and $y$ and such that,
	for every $z\in \eta$ we have 
	\begin{align}
		\label{e:D5}
		\rho_0(x,z)\leq L_0\rho_0(x,y).
	\end{align}
        Fix some such $z$.
        We will see using the $\lambda_0$--quasi-convexity of $\Yx$  that $z\notin\Lambda \Yx$.
        Let $\alpha$ be the geodesic ray defining $\Pi_K(x)$, and let $\beta_y,\beta_z$ be the geodesic rays defining $\Pi_K(y),\Pi_K(z)$ respectively.

        Let $c_y,c_z\in \alpha$ be the centers $c_y = c(w_0,\alpha,\beta_y)$ and $c_z = c(w_0,\alpha,\beta_z)$ from Definition~\ref{def:center}.  If $d_{\Xz}(c_z,\Yx) \ge K$ then it is easy to bound $d_{\Xz}(\Pi_K(x),\Pi_K(z))$, so we assume $d_{\Xz}(c_z,\Yx) < K$.

        By \eqref{e:d1} we have (writing $|\cdot|$ for $d_{\Xz}(w_0,\cdot)$)
	\begin{align*}
          \rho_0(x,z) & \ge \kappa^{-1}e^{-\epsilon |c_z| },\mbox{ and }\\
          \rho_0(x,y) & \le \kappa e^{-\epsilon |c_y|}.
	\end{align*}
        Taking logs and applying \eqref{e:D5} we deduce
        \[ -\epsilon |c_y| + \log\kappa + \log L_0 \ge -\epsilon|c_z| - \log\kappa.\]
        Rearranging gives
        \begin{equation}
          \label{eq:czest1}
          |c_z| \ge |c_y| - \frac{2\log\kappa + \log L_0}{\epsilon}.
        \end{equation}

        Next we note that since $d_{\Xz}(\Pi_K(x),\Pi_K(y))\le 24\delta_0$, the distance between the images of $\Pi_K(x)$ and $\Pi_K(y)$ in the comparison tripod $T(\alpha,\beta_y)$ is at most $29\delta_0$, and so
        \begin{equation}
          \label{eq:cyest}
          |c_y| \ge |\Pi_K(x)| - 29\delta_0.
        \end{equation}
        Combining \eqref{eq:czest1} with \eqref{eq:cyest} gives
        \begin{equation*}
          \label{eq:czest2}
          |c_z| \ge |\Pi_K(x)| - \left(\frac{2\log\kappa + \log L_0}{\epsilon} + 29\delta_0\right).
        \end{equation*}
        Since $c_z$ and $\Pi_K(x)$ both lie on the geodesic ray $\alpha$, and $d_{\Xz}(\Pi_K(x),\Yx) = K$, we have
        \begin{equation*}
          \label{eq:awayfromgamma0}
          d_{\Xz}(c_z,\Yx) \ge K - \left(\frac{2\log\kappa + \log L_0}{\epsilon} + 29\delta_0\right) \ge 11\delta_0 + \lambda_0.
        \end{equation*}
        Let $c_z' = c(w_0,\beta_z,\alpha)\in \beta_z$, so that $d_{\Xz}(c_z,c_z')\le 5\delta_0$.  We have $d_{\Xz}(c_z',\Yx) \ge 6\delta_0 + \lambda_0$.  In particular, since  $\beta_z$ contains a point further than $\lambda_0$ (the quasi-convexity constant) from $\Yx$, we deduce that $z$ is not in the limit set of $\Yx$.        
        In particular we may apply Lemma~\ref{lem:constant_progress}  to estimate the distance between points on $\beta_z$ in terms of the distance to $\Yx$.
        We conclude $d_{\Xz}(c_z',\Pi_K(z)) \le \left(\frac{2\log\kappa + \log L_0}{\epsilon} + 29\delta_0+ 5\delta_0 \right) + 10\delta_0$, and finally
        \begin{align*}
          d_{\Xz}(\Pi_K(x),\Pi_K(z)) & \le d_{\Xz}(\Pi_K(x),c_z) + d_{\Xz}(c_z,c_z')+d_{\Xz}(c_z',\Pi_K(z))\\
                   & \le \left(\frac{2\log\kappa + \log L_0}{\epsilon} + 29\delta_0\right) + 5\delta_0 \\ & + \left(\frac{2\log\kappa + \log L_0}{\epsilon} + 29\delta_0+ 5\delta_0 \right) + 10\delta_0\\
          & = 2\left(\frac{2\log\kappa + \log L_0}{\epsilon} + 29\delta_0\right) + 20\delta_0 =C_0
        \end{align*}
        as desired. 
        \end{proof}
\begin{proposition}\label{prop:out_and_back}
  Let $(\Xz,\Yx)$ be as in Assumption~\ref{ass:X}.  For any $D \ge D_0$, any $K \ge \max\{R_{\ref{lem:claim_localpath}}, C_0 + \Delta_0 + \lambda_0 + 3\delta_0 \}$, and any points $x,y\in \partial\Xz\ssm\Lambda \Yx$ the following holds:  Let $\alpha$ be a path in $\CS_K(\Yx)$ joining $\Pi_K(x)$ to $\Pi_K(y)$.  Then there is a path $\eta$ from $x$ to $y$ in $\partial\Xz\ssm\Lambda \Yx$ so that any projection of $\eta$ to $\CS_K(\Yx)$
  is homotopic rel endpoints to $\alpha$ in $(\CS_K(\Yx))^D$.
  Moreover such a projection lies in a $D_0/3$--neighborhood of $\alpha$.
\end{proposition}

\begin{proof}
  Let $\alpha$ be such a path, and divide it into subpaths $\alpha_i$ from $x_i$ to $x_{i+1}$ of length at most $8\delta_0$, where the points $x_i$ lie in the shell $S_K(\Yx)$.
  For each $i$ there is a point $p_i\in \partial \Xz$ so that $d_{\Xz}(\Pi_K(p_i), x_i) \le 8\delta_0$.  (Note that $p_i\notin \Lambda \Yx$.)  Inserting geodesic segments back and forth between the points $x_i$ and $\Pi_K(p_i)$ produces a path $\alpha'$ composed of subpaths $\alpha_i':= [\Pi_K(p_i),x_i]\cdot\alpha_i\cdot [x_{i+1},\Pi_K(p_{i+1})]$.  (Being sure to use the same geodesic segment $[x_{i+1},\Pi_K(p_{i+1})]$ at the end of $\alpha_i'$ and, in reverse, at the beginning of $\alpha_{i+1}'$.)

  By Lemma~\ref{lem:claim_localpath} there is, for each $i$, a path $\eta_i$ in $\partial X_0\ssm \Lambda \Yx$ joining $p_i$ to $p_{i+1}$ so that $\Pi_K(\eta_i)$ lies in a $C_0$--ball (with respect to $d_{\Xz}$) around $\Pi_K(p_i)$.  Since $K$ is large enough to apply Lemma~\ref{lem:proper_dist}, we conclude that $d_{\CS_K(\Yx)}(\Pi_K(p_i),\Pi_K(z))\le \Phi(C_0) < D/6$ for any $z\in \eta_i$.  Let $\eta$ be the path from $p$ to $q$ which is the concatenation of these paths $\eta_i$, and let $\alpha''$ be a projection of $\eta$ to $\CS_K(\Yx)$ which passes through the points $\Pi_K(p_i)$.  (Choose the partition to be a refinement of the one which includes the endpoints of the domains of the $\eta_i$.)  Let $\alpha_i''$ be the subpath of $\alpha''$ from $\Pi_K(p_i)$ to $\Pi_K(p_{i+1})$.  This path is composed of segments $[\Pi_K(z),\Pi_K(z')]$ of length at most $8\delta_0$, so that both endpoints are $\Xz$--distance at most $8\delta_0$ from either $\Pi_K(p_i)$.

  We claim the loop composed of $\alpha_i''$ and $\alpha_i'$ is contractible in $(\CS_K(\Yx))^D$.  We first note that each $\alpha_i''$ is homotopic to a $\CS_K(\Yx)$--geodesic $\gamma_i$ joining its endpoints.  This is because each projected $\eta_i$ is composed of geodesic segments of length at most $8\delta_0$, each of whose endpoints can be connected to $p_i$ (the initial point of $\alpha_i''$) by a path in $\CS_K(\Yx)$ of length at most $\Phi(C_0)$.
  This decomposes the loop consisting of $\alpha_i''$ and $\gamma_i$ into triangular loops of length at most $2\Phi(C_0) + 8\delta_0 < D$.
  See Figure~\ref{fig:homotopy}.

\begin{figure}[h!]
	\begin{center}
		\includegraphics[scale=0.6]{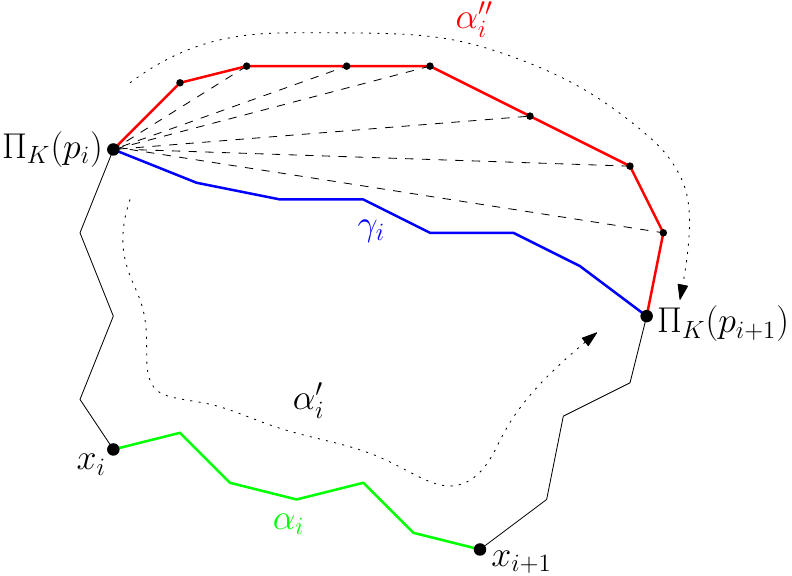}
	\end{center}
	\caption{Homotoping a projection of $\eta$ to $\alpha$.}\label{fig:homotopy}
\end{figure}
  
  The paths $\gamma_i$ and $\alpha_i'$ have length at most $24\delta_0$, so together they bound a loop of length at most $48\delta_0<D$.

  Putting the pieces together we obtain a homotopy rel endpoints from the projection $\alpha''$ to $\alpha$.  Lemma~\ref{lem:welldef_projected_path} says any two projections are homotopic rel endpoints.  

  We also see from this picture that $\alpha''$ is contained in the $\Phi(C_0) + 12\delta_0$--neighborhood of $\alpha$.  Hence (by Lemma~\ref{lem:welldef_projected_path}) any projection of $\eta$ lies within $\Phi(C_0) + 36\delta_0=D_0/3$ of $\alpha$.
\end{proof}

\begin{corollary}\label{c:surject_pi1}
    Let $(\Xz,\Yx)$ be as in Assumption~\ref{ass:X}.  Then for $R_{\ref{c:surject_pi1}}=\max\{R_{\ref{lem:claim_localpath}}, C + \Delta_0 + \lambda_0 + 3\delta_0 \}$, for all $K \geq R_{\ref{c:surject_pi1}}$ and $D\ge D_0$,  the following holds:

  For any point $\ast\in \partial \Xz\ssm \Lambda \Yx$, the homomorphism  $\Pi_\ast : \pi_1(\partial \Xz\ssm\Lambda \Yx,\ast)\to \pi_1^{D}(\CS_K(\Yx),\Pi_K(\ast))$ from Proposition~\ref{prop:homomorphism} is surjective.
\end{corollary}

Recall the choice of rays $\eta_p$ from Definition~\ref{def:proj}, used to define the projection function $\Pi_K$.

\begin{lemma}\label{lem:small_eta}
    In the context of Proposition~\ref{prop:out_and_back}, there is a function $R_{\ref{lem:small_eta}}(E,\iota)$ defined for positive $E$ and $\iota$ so that for any $K>R_{\ref{lem:small_eta}}(E,\iota)$ and any path $\alpha$ in $\CS_K(\Yx)$ of diameter at most $E$, the path $\eta$ produced by Proposition~\ref{prop:out_and_back} has $\rho_0$--diameter at most $\iota$.
\end{lemma}
\begin{proof}
  Let $z,z'$ be two points in $\eta$.  There is a projection of $\eta$ which passes through the points $a=\Pi_K(z)$ and $a'=\Pi_K(z')$, so the points $a,a'$ lie in a $D_0/3$--neighborhood of $\alpha$.  In particular $d_{\Xz}(a,a')< E + D_0$.  One can see from this that the center $c(w_0,\ray{z},\ray{z'})$ must be at least $K - (E + D_0 + 5\delta_0)$ from the basepoint $w_0$.  The distance from the basepoint to the center is the Gromov product $(z|z')_{w_0}$ so we deduce
  \[ \rho_0(z,z') \le \kappa e^{-\epsilon(K - (E + D_0 + 5\delta_0))}. \]
  Fixing $E$, we may increase $K$ to make the right-hand side less than $\iota$, as desired.
\end{proof}

\begin{assumption} \label{ass:for_injectivity}
  In the setting of Assumption \ref{ass:X}, assume further that $\partial \Xz\ssm \Lambda \Yx$ is locally simply connected, and that some $H_0< \operatorname{Isom}(\Xz)$ preserves and acts cocompactly on $\Yx$.  Choose a constant $A_0$ so that for all $x \in \Yx$ there exists $h \in H_0$ so that $d_{\Xz}(h\cdot x, w_0) \le A_0$.
\end{assumption}

\begin{lemma} \label{lem:translate_to_w_0}
    In the setting of Assumption~\ref{ass:for_injectivity}, for all $\zeta \in \partial X_0 \smallsetminus \Lambda \Yx$ there exists $h \in H_0$ so that for any $\xi \in \Lambda \Yx$ and any choice of geodesic ray $\beta_\xi$ from $w_0$ to $\xi$ we have 
    \[ d\left(w_0,c(w_0,\ray{h\cdot \zeta},\beta_\xi)\right) \le A_0 + \lambda_0 + 12\delta_0 . \]
    In particular, $h\cdot \zeta \in \partial \Xz \smallsetminus \mathring{N}_{r_0}(\Lambda \Yx )$, for $r_0 = \kappa^{-1} e^{-\epsilon(A_0 + \lambda_0 + 12\delta_0)}$.
\end{lemma}
\begin{proof}
    Let $\ray{\zeta}$ be the geodesic ray from $w_0$ to $\zeta$ chosen in Definition~\ref{def:proj}, and let $z = \Pi_{\lambda_0 + 6\delta_0}(\zeta)$ be the projection from $\zeta$ to $S_{\lambda_0 + 6\delta_0}$ as in that definition.  Then $z$ lies at distance exactly $\lambda_0 + 6\delta_0$ from $\Yx$, so there exists $h \in H_0$ so that $d(w_0,h \cdot z) \le A_0 + \lambda_0 + 6\delta_0$.

    Let $\ray{h \cdot \zeta}$ and $v = \Pi_{\lambda_0 + 6\delta_0}(h \cdot \zeta)$ be the analogous geodesic ray from $w_0$ to $h \cdot \zeta$ and the projection for $h \cdot \zeta$ to $S_{\lambda_0 + 6\delta_0}$.  Lemma~\ref{lem:piK well-defined} implies that $d(h \cdot z,v) \le 6\delta_0$, which implies that $d(w_0, v) \le A_0 + \lambda_0 + 12\delta_0$.  Now, $v$ lies at distance $\lambda + 6\delta_0$ from $\Yx$, and every point after $v$ on $\ray{h \cdot \zeta}$ lies at least that distance from $\Yx$.

    Let $\xi \in \Lambda \Yx$ be arbitrary, and choose a geodesic ray $\beta_{\xi}$ from $w_0$ to $\xi$.  By the assumption on $\Yx$ in Assumption~\ref{ass:X}, $\beta_{\xi}$ lies entirely in the $\lambda_0$--neighborhood of $\Yx$.  It follows that the center $c\left(w_0,\ray{h \cdot \zeta},\beta_{\xi} \right)$ lies closer to $w_0$ on $\ray{h \cdot \zeta}$ than $v$ does, which implies that  \[ d\left(w_0,c(w_0,\ray{h\cdot \zeta},\beta_\xi)\right) \le A_0 + \lambda_0 + 12\delta_0 ,\]
    as required.
\end{proof}

\begin{theorem}\label{t:Pi_isomorphism}
  Let $\Xz, \Yx, H_0, A_0$ be as in Assumption~\ref{ass:for_injectivity}, and let $D_0 = 3\Phi(C_0) + 108\delta_0$ be as above. For all $D\geq D_0$, and for all $\sys > 0$ there exists
  \[ R_{\ref{t:Pi_isomorphism}}=R_{\ref{t:Pi_isomorphism}}({\Xz},\Yx,D,\sys,A_0) \] so that for all $K \geq R_{\ref{t:Pi_isomorphism}}$ and all $\ast\in\partial\Xz\ssm\Lambda\Yx$
  the following properties hold: 
	\begin{enumerate}
		\item\label{item:boundary_to_shell} the homomorphism $$\Pi_{\ast} \co \pi_1(\partial {\Xz}\ssm\Lambda \Yx,\ast)\to \pi_1^{D}(\CS_K(\Yx),\Pi_K(\ast)) $$ 
  from Proposition~\ref{prop:homomorphism} is an isomorphism;
		\item\label{item:long_loop} any homotopically nontrivial loop in the one-skeleton of $(\CS_K(\Yx))^D$ has length at least $\sys$;
		\item\label{item:diagram} for any $h\in H_0$, there is a canonical isomorphism \[\pi_1^D(\CS_K(\Yx),\Pi_K(h\ast))\cong \pi_1^D(\CS_K(\Yx),h\Pi_K(\ast))\] and moreover
  there is a commutative diagram as follows:  
  		\begin{equation} \label{eq:pi1_cd} \tag{$\dagger$}
			\begin{tikzcd}	
                \pi_1(\partial {\Xz}\ssm\Lambda \Yx,\ast) \arrow{r}{\text{$\Pi_{\ast}$}} \arrow{d}& 
                \pi_1^{D}(\CS_K(\Yx),\Pi_K(\ast))\arrow{d}\\
                \pi_1(\partial {\Xz}\ssm\Lambda \Yx,h\ast)\arrow{r}{\text{$\Pi_{h\ast}$}} &
                \pi_1^{D}(\CS_K(\Yx),\Pi_K(h\ast))
			\end{tikzcd}
		\end{equation}
  The left vertical arrow is the obvious map induced by the action of $h$ on $\partial{\Xz} \ssm \Lambda \Yx$ and the right vertical arrow is the composition of the obvious map induced by the action of $h$ on $\CS_K(\Yx)$ with the above canonical isomorphism.
	\end{enumerate}
\end{theorem}
\begin{proof}
  We claim first there is a function $\theta: (0,\infty)\to (0,1]$, depending only on the pair $(\Xz,\Yx)$ and the visual metric $\rho_0$ on $\partial \Xz$ chosen in Assumption~\ref{ass:X}, so that any loop in $\partial {\Xz}\ssm \mathring{N}_r(\Lambda \Yx )$ of diameter at most $\theta(r)$ is contractible in $\partial {\Xz}\ssm \Lambda \Yx$.  (The proof of the existence of such a function $\theta$ uses the local simple connectivity of $\partial {\Xz}\ssm \Lambda \Yx$ and a compactness argument.)  We may assume that $\theta$ is non-decreasing.

  Let $A_0$ be the diameter of a fundamental domain for the action of $H_0$ on $\Yx$ as fixed in Assumption~\ref{ass:for_injectivity}, and let $A = A_0 + \lambda_0 + 12\delta_0$.
Fix $D \ge D_0$ and also $\sys$, and let $\iota_0 = \min\{ \frac{1}{2} \kappa^{-1}e^{-\epsilon A}, \theta\left(\frac{1}{2} \kappa^{-1}e^{-\epsilon A}\right) \}$.

We claim that
\[  R_{\ref{t:Pi_isomorphism}} = \max \left\{ R_{\ref{c:surject_pi1}}, R_{\ref{lem:small_eta}}\left( \max\left\{ D,\sys\right\} + 24\delta_0,\iota_0\right) \right\} \]
satisfies the conclusions of the theorem.

We first establish Item~\eqref{item:long_loop}.  Let $\alpha$ be a loop in $\CS_K(\Yx)$ of length less than $\sys$.  If $\alpha$ is not already contractible in $\CS_K(\Yx)$, then it contains a point in $S_K(\Yx)$.
Some point of $\alpha$ is at most $8\delta_0$ from some point $\Pi_K(x)$.  In particular $\alpha$ is freely homotopic in $\CS_K(\Yx)$ to a loop $\alpha'$ based at $\Pi_K(x)$, of length at most $\sys+16\delta_0$.  Apply Lemma~\ref{lem:translate_to_w_0} to the point $x$ to find an element $h$ so that $h\cdot x \in \partial \Xz \smallsetminus \mathring{N}_{r_0}\left(\Lambda \Yx \right)$, for $r_0 = \kappa^{-1}e^{-\epsilon A}$.  Let $\eta$ be the path produced from $h \cdot \alpha'$ by Proposition~\ref{prop:out_and_back}.

By the choice of $\iota_0 = \min\left\{ \frac{1}{2} \kappa^{-1}e^{-\epsilon A}, \theta\left(\frac{1}{2} \kappa^{-1}e^{-\epsilon A}\right) \right\}$, and the fact that $K \ge R_{\ref{lem:small_eta}}\left( \max\left\{ D,\sys\right\} + 24\delta_0,\iota_0\right)$, Lemma~\ref{lem:small_eta} implies that the $\rho_0$--diameter of $\eta$ is at most $\frac{1}{2} \kappa^{-1}e^{-\epsilon A}$, so $\eta \subseteq \partial \Xz \smallsetminus \mathring{N}_{r}(\Lambda \Yx)$, for $r = \frac{1}{2} \kappa^{-1}e^{-\epsilon A}$.  Since the $\rho_0$--diameter of $\eta$ is at most $\theta\left(\frac{1}{2} \kappa^{-1}e^{-\epsilon A}\right)$, the choice of the function $\theta$ implies that $\eta$ is contractible in $\partial\Xz \smallsetminus \Lambda \Yx$.

Applying Proposition~\ref{prop:homomorphism} with $\ast = h \cdot x$ (the basepoint of $\eta$) implies that any projection of $\eta$ to $\CS_K(\Yx)$ is contractible in $(\CS_K(\Yx))^D$.  It follows from Proposition~\ref{prop:out_and_back} that $h \cdot \alpha'$, and hence $h \cdot \alpha$ are contractible in $(\CS_K(\Yx))^D$.  Since $h \in \Isom(\Xz)$ preserves $\CS_K(\Yx)$, we see that $\alpha'$ is contractible in $(\CS_K(\Yx))^D$, proving Item~\ref{item:long_loop}.

We now prove item~\eqref{item:boundary_to_shell}.  In view of Corollary~\ref{c:surject_pi1}, it is enough to prove that the homomorphism $\Pi_{\ast}$ is injective.  To that end, let $\alpha$ be a loop in $\CS_{K}(\Yx)$ based at $\Pi_K(\ast)$ which represents the trivial element of $\pi_1^{D}(\CS_K(\Yx),\Pi_K(\ast))$.  By applying Proposition~\ref{prop:out_and_back}, we may assume that $\alpha$ is the projection to $\CS_{K}(\Yx)$ of a path $\eta$ in $\partial \Xz \smallsetminus \Lambda \Yx$ based at $\ast$.  Since $\alpha$ represents the identity, it can be filled with a van Kampen diagram whose $2$--cells are labelled by loops in $\CS_K(\Yx)$ of length at most $D$.  Modifying this diagram, we may assume that all $2$--cells are triangles of length at most $D+24\delta_0$, all of whose corners are in the image of $\Pi_K$.

Applying Proposition~\ref{prop:out_and_back} to the $1$--cells of this van Kampen diagram $\Delta$, we obtain a map from the $1$-skeleton of $\Delta$ into $\partial \Xz \smallsetminus \Lambda \Yx$ so that the map restricted to the boundary of $\Delta$ is $\eta$, and the boundary of each $2$--cell in $\Delta$ is mapped to a loop which projects to a loop in $\CS_{K}(\Yx)$ at distance at most $D_0/3$ from the corresponding loop in the original van Kampen diagram. The argument above proving item~\eqref{item:long_loop} shows that the boundary of each $2$--cell in $\Delta$ maps to a contractible loop in $\partial \Xz \smallsetminus \Lambda \Yx$.  This proves that $\eta$ represents the trivial element of $\pi_1\left( \partial \Xz \smallsetminus \Lambda \Yx, \ast \right)$, as required.  This proves item~\eqref{item:boundary_to_shell}.

Finally, we prove item~\ref{item:diagram}.  Observe that by Lemma~\ref{lem:piK well-defined}, for any $\zeta \in \partial \Xz \smallsetminus \Lambda \Yx$ and any $h \in H_0$, we have $d_{\Xz} \left( h \cdot \Pi_K(\zeta), \Pi_K(h \cdot \zeta) \right)\le 8\delta_0$. Since there is a unique $D$--homotopy class of path of length at most $8\delta_0$ between $\Pi_K(h\cdot \ast)$ and $h \cdot \Pi_K(\ast)$, there is a canonical change of basepoint isomorphism, which is the isomorphism from the statement.

We now establish the commutativity of \eqref{eq:pi1_cd}.  Let $\eta$ be a loop based at $\ast$ in $\partial \Xz\ssm\Lambda \Yx$, and let $\alpha$ be a projection of $\eta$ to $\CS_K(\Yx)$.  Let $\ast = z_0,\ldots,z_n = \ast$ be the points used in defining the projection $\alpha$.  Adding more points if necessary, we may use the points $h z_i$ to define a projection $\alpha'$ of $h\eta$.
The diagram  
		\begin{equation*}
			\begin{tikzcd}
				\partial {\Xz}\ssm\Lambda \Yx\arrow{d}{\text{$h$}}\arrow{r}{\text{$\Pi_K$}} & \CS_K(\Yx)\arrow{d}{\text{$h$}}  \\
				\partial {\Xz}\ssm\Lambda \Yx \arrow{r}{\text{$\Pi_K$}} & 
				\CS_K(\Yx) 
			\end{tikzcd}
		\end{equation*}
commutes up to $8\delta_0$.    
Hence for each $i$, there is a path of length at most $8\delta_0$ joining $\Pi_K(h z_i)$ to $h\Pi_K(z_i)$ in $\CS_K(\Yx)$.  
These combine with subsegments of $h \alpha$ and $\alpha'$ to form quadrilaterals of length at most $32\delta_0 \le D$ determining a $D$--homotopy from $h \alpha$ to $\alpha'$.  In the diagram \eqref{eq:pi1_cd} the loop $\alpha'$ represents the result of going down and then right; the loop $h\alpha$ represents the result of going right and then down.  
\end{proof}

The following is an immediate consequence of Item~\ref{item:boundary_to_shell} of Theorem~\ref{t:Pi_isomorphism}, and Lemma~\ref{lem:shcon} which implies that the basepoint does not matter.

\begin{corollary} \label{cor:pi_1_Z}
    Let $\Xz$, $\Yx$, $H_0$, and $A_0$ be as in Assumption~\ref{ass:for_injectivity}, and suppose furthermore that $\Yx$ is a bi-infinite quasi-geodesic, and that $\partial \Xz \cong S^2$.  Let $D_0 = 3\Phi(C_0) + 108\delta_0$ be as above, and fix $D \ge D_0$ and $\sys$.  Let $R_{\ref{t:Pi_isomorphism}} = R_{\ref{t:Pi_isomorphism}}(X_0,Y_0,D_0,\sys,A_0)$ be the constant from Theorem~\ref{t:Pi_isomorphism}.  Then for all $K \ge R_{\ref{t:Pi_isomorphism}}$ and all $p \in \CS_K(\Yx)$ we have
    \[\pi_1^D\left( \CS_K(\Yx), p \right) \cong \Z.\]
\end{corollary}

\begin{remark}
    In case $\partial \Xz \cong S^2$ and $\Isom(\Xz)$ acts cocompactly on $\Xz$, it follows from \cite[Lemma 2.5]{BonkKleiner:QS} that $\partial \Xz$, equipped with any visual metric, is \emph{linearly locally contractible}.  This means that one can take the function $\theta$ from the proof of Theorem~\ref{t:Pi_isomorphism} to be a linear function.  It is then straightforward to prove that the conclusion of Corollary~\ref{cor:pi_1_Z} remains true (for a (possibly larger) $R_{\ref{t:Pi_isomorphism}}$) even without the assumption that $\Yx$ admits a cocompact action by $H_0$, and without the reliance on the diameter of the fundamental domain.  In fact, in this setting, the constant $R_{\ref{t:Pi_isomorphism}}$ does not depend on the choice of $\Yx$, only on its quasi-convexity constant (as well as $\Xz$, etc.).

    Since this argument works only in a setting of considerably smaller generality than Theorem~\ref{t:Pi_isomorphism}, we decided not to pursue it, since we believe that Theorem~\ref{t:Pi_isomorphism} should be applicable in many situations beyond the one considered in this paper.

    However, we also believe that the uniformity of drilling radius sketched in this remark will be important in other contexts.
\end{remark}

\subsection{Coarse topology of tube complements}\label{ss:ct of tubes}

In this section we assume that we are in the setting of Assumption~\ref{ass:for_injectivity}.

Recall the construction of the completed shell from Definition~\ref{def:completed shell}, the completed tube complement from Definition~\ref{def:CTC}, and also Conventions \ref{conv:CS} and \ref{conv:CTC}.

In this subsection, we investigate the coarse topology of completed shells and completed tube complements.

Throughout this subsection, fix the following notation. 

\begin{notation} \label{not:Y} 
Let $\delta_1$ be the hyperbolicity constant of ${\Xz}^{\mathrm{cusp}}(K)$ (for large enough $K$) from Theorem~\ref{thm:cusp_tube_hyp}, and let $\delta_2 = 1500\delta_1$ (5 times the hyperbolicity constant from the conclusion of the Coarse Cartan-Hadamard~\ref{t:CCH}).
Without loss of generality, we assume that $\delta_1 \ge \delta_0$, where $\delta_0$ is the hyperbolicity constant for $\Xz$ from Assumption~\ref{ass:X}.  Recall that $Y_0 \subset X_0$ is a $\lambda_0$--quasi-convex subset of $X_0$, as in Assumption~\ref{ass:X}.
\end{notation}

\begin{definition} \label{def:comparable}
Suppose that $Z$ is a graph, that $Y \subset Z$, and that $\alpha > 0$.  Then $(Z,Y)$ is \emph{$\alpha$--tube comparable to $(\Xz,\Yx)$} if there is an isomorphism of graphs from $N^Z_{\alpha}(Y)$ to $N^{\Xz}_{\alpha}(\Yx)$ which takes $Y$ (isomorphically) onto $\Yx$.
\end{definition}

\begin{remark} \label{rem:def_ambi}
In our context, all spaces are graphs.  There are two natural metrics on $N^{\Xz}_\alpha(\Yx)$ that one might consider.  The first is the induced metric from $d_{\Xz}$.  The second is the induced graph metric on $N^{\Xz}_\alpha(\Yx)$, and this is what is meant in Definition~\ref{def:comparable} above.  We choose this second meaning since it is easier to verify.

For some of our applications, it is not enough that $N^Z_\alpha(Y)$ and $N^{\Xz}_\alpha(\Yx)$ be isomorphic as graphs, and we need their induced metrics to be isometric.  However, in our applications, each of these graphs are $2\delta_2$--quasi-convex (by the obvious variant of Lemma~\ref{lem:tubequasiconvex}).  In this case, if $N_{\alpha+2\delta_2}^Z(Y)$ and $N_{\alpha+2\delta_2}^{\Xz}(\Yx)$ are isometric with the induced graph metrics then $N_{\alpha}^Z(Y)$ and $N_\alpha^{\Xz}(\Yx)$ are isometric with their induced metrics from $d_Z$, $d_{\Xz}$, respectively.  
\end{remark}
The issue raised in Remark~\ref{rem:def_ambi} is the reason why many of the hypotheses about sets being tube comparable have an extra $2\delta_2$ beyond what might first be expected.

The following is immediate from Remark~\ref{rem:def_ambi}, from the definition of the completed shell, and from the isomorphism between shells induced by being $(K+4\delta_0)$--tube comparable.  We need $K+4\delta_0$ so that any path of length $8\delta_0$ between points on the $K$--shell is included in the tube comparison.
\begin{lemma} \label{lem:CS iso}
Suppose that $(Z,Y)$ is $(K+4\delta_0)$--tube comparable to $(\Xz,\Yx)$.  Then $\CS^Z_K(Y)$ is isometric to $\CS^{\Xz}_K(\Yx)$.
\end{lemma}

The proof of the next result is identical to that of
Lemma~\ref{lem:CTC connected}, except that Lemma~\ref{lem:CS iso} above is used as well as Lemma~\ref{lem:shcon}.

\begin{lemma}
Suppose that $K \ge \lambda_0 + 3\delta_0$, that $\Gamma$ is a connected graph, and that $(\Gamma,Y)$ is $(K+4\delta_0)$--tube comparable to $(\Xz,\Yx)$.  Then $\CTC_K^\Gamma(Y)$ is a connected graph. 
\end{lemma}

\begin{lemma} \label{lem:gamma_0 qg}
Suppose that $\Upsilon$ is $\delta_2$--hyperbolic.
If $K \ge \lambda_0 + \delta_2$ and $(\Upsilon,Y)$ is $(K+2\delta_2)$--tube comparable to $(\Xz,\Yx)$ then $Y$ is a $\lambda_0$--quasi-convex subset in $\Upsilon$.  
\end{lemma}
\begin{proof}
Suppose that $\sigma$ is a geodesic with endpoints $p,q$ on $Y$.
By Remark~\ref{rem:def_ambi}, there is an isometry $\phi: N_K(\Yx)\to N_K(Y)$ taking $\Yx$ onto $Y$.  Let $\sigma'$ be a geodesic in $X$ between $\phi^{-1}(p)$ and $\phi^{-1}(q)$.  Since $\Yx$ is $\lambda_0$--quasi-convex in $\Xz$, and $K>\lambda_0$, the geodesic $\sigma'$ is in the domain of $\phi$, and its image $\phi(\sigma')$ is an $\Upsilon$--geodesic joining $p$ to $q$, and lying in the $\lambda_0$--neighborhood of $Y$.  Since $\Upsilon$ is $\delta_2$--hyperbolic, $\sigma$ lies in the $\lambda_0 + \delta_2$--neighborhood of $Y$.  Since $K\ge \lambda_0 + \delta_2$, the preimage $\phi^{-1}(\sigma)$ is a geodesic joining $\phi^{-1}(p)$ to $\phi^{-1}(q)$.  Since $\Yx$ is $\lambda_0$--quasi-convex, $\phi^{-1}(\sigma)$ actually lies in the $\lambda_0$--neighborhood of $\Yx$, and so $\sigma$ lies in the $\lambda_0$--neighborhood of $Y$, as required.
 \end{proof}

Recall the constant $\Delta_0$ from Lemma~\ref{lem:GMS} and the function $\Phi$ from Lemma~\ref{lem:proper_dist}.  
The following is an analogue of Lemma~\ref{lem:proper_dist}, and follows immediately from that result and Lemma \ref{lem:CS iso} (and the definition of $\CS^\Upsilon_K(Y)$, Definition~\ref{def:completed shell}).

\begin{lemma} \label{lem:proper_dist2}
Suppose that $\Upsilon$ is a $\delta_2$--hyperbolic space.  Suppose that $C > 0$, that $K \ge C + \Delta_0 + \lambda_0 + 3\delta_0$, and that $(Z,Y)$ is $(K+2\delta_2)$--tube comparable to $(\Xz,\Yx)$.  For any $z_1,z_2 \in S_K^Z(Y)$ so that $d_Z(z_1,z_2) \le C$ we have $d_{\CS^Z_K(Y)}(z_1,z_2) \le \Phi(C)$.
\end{lemma}

Recall the definition of $Q$--deformation retraction from Definition~\ref{def:def retract}.
\begin{lemma}\label{lem:easy retract}
Suppose that $\Upsilon$ is a $\delta$--hyperbolic graph, that $\Xi \subset \Upsilon$ is a $\lambda$--quasi-convex sub-graph, and that $K \ge 4\delta + \lambda$.  There exists a $(2\delta+1)$--deformation retraction $\{ r_i  \co \Upsilon \to \Upsilon \}$ of $\Upsilon$ onto $N_K(\Xi)$ so that for all $i \ge 1$ and all $z \in \Upsilon \ssm N_K(\Xi)$ we have $r_i(z) \in \Upsilon^{(0)}$.

Moreover if $z\notin N_K(\Xi)$ then $d_\Upsilon(r_i(z),\Xi)\ge K$ for all $i$.
\end{lemma}
\begin{proof}
For $z \in \Upsilon$, let $P_\Xi(z)$ be a closest point of $\Xi$ to $z$, and choose a unit speed geodesic $\alpha_z: [-d_\Upsilon(z,\Xi),0]\to X$ joining $z$ to $P_\Xi(z)$.  (The domain is chosen so that whether or not $z$ is a vertex, $\alpha_z$ sends integers to vertices.)
Define $f_z(i) = \min\{-K,\lfloor i-d_\Upsilon(z,\Xi)\rfloor\}$, and $r_i(z) = \alpha_z(f_z(i))$.  
Thus $r_1(z)$ is the first vertex on $\alpha_z$ after $z$, $r_2(z)$ is the next, and so on, until for large $i$ we have $r_i(z)$ a closest point to $z$ in $N_K(\Xi)$. 

 Conditions \eqref{eq:0Id}, \eqref{eq:evId} and \eqref{eq:speed} from Definition~\ref{def:def retract} obviously hold.  Condition \eqref{stability} holds because every point in $\Upsilon$ is finite distance from $\Xi$.
The condition that $r_i(z) \in \Upsilon^{(0)}$ for $i \ge 1$ and $z \in \Upsilon \ssm N_K(\Xi)$ is clear from the construction, as is the ``Moreover'' statement.

It remains to verify Condition~\eqref{eq:CL} from Definition~\ref{def:def retract}.  To that end, suppose that $z_1,z_2 \in \Upsilon$ satisfy $d_\Upsilon(z_1,z_2) \le 1$, and consider the geodesic quadrilateral with two sides $\alpha_{z_1}$ and $\alpha_{z_2}$, a geodesic between $z_1$ and $z_2$, and a geodesic between the projections $P_\Xi(z_1)$ and $P_\Xi(z_2)$.  Since $K \ge \lambda + 4\delta$, it is straightforward to see from the fact that geodesic quadrilaterals are $2\delta$--thin, that for all $i$ we have $d_\Upsilon(f_i(z_1),f_i(z_2)) \le 2\delta +1$, as required.
\end{proof}

Recall the construction of the completed tube complement 
\[	\CTC^\Upsilon_K(Y)  = \left( \Upsilon \ssm T^\Upsilon_K(Y) \right) \sqcup_{S^\Upsilon_K(Y)} \CS^\Upsilon_K(Y)	\]
from Definition~\ref{def:CTC}, and let
\[	\iota \co \Upsilon \ssm T^\Upsilon_K(Y) \to \CTC^\Upsilon_K(Y)	\]
denote the inclusion map.  

\begin{lemma}\label{lem:CTC dist}
Let $C > 0$ and suppose $K \ge C + \Delta_0 + \lambda_0 + 3\delta_0$.  Suppose also that $\Upsilon$ is a $\delta_2$--hyperbolic metric space and that $(\Upsilon,Y)$ is $(K+\delta_2)$--tube comparable to $(\Xz,\Yx)$.  For any $z_1,z_2 \in \Upsilon \ssm T_K^\Upsilon(Y)$ so that $d_\Upsilon(z_1,z_2) \le C$ we have
\[	d_{\CTC_K^\Upsilon(Y)}(\iota(z_1),\iota(z_2)) \le 2\Phi(C)	.	\]
\end{lemma}
\begin{proof}
Let $\gamma$ be a geodesic in $\Upsilon$ between $z_1$ and $z_2$.  If $\gamma$ does not meet $S_K^\Upsilon(Y)$ there is nothing to show, since $\Phi(C)\ge C$.

Suppose that $z_1'$ is the first point (from $z_1$) on $\gamma$ that intersects $S_K^\Upsilon(Y)$ and that $z_2'$ is the last point (closest to $z_2$ along $\gamma$) that intersects $S_K^\Upsilon(Y)$.  Those parts of $\gamma$ that lie between $z_1$ and $z_1'$ and between $z_2'$ and $z_2$ can be considered as paths in $\CTC_K^\Upsilon(Y)$.  On the other hand, $z_1'$ and $z_2'$ are points in $\CTC_K^\Upsilon(Y)$ which satisfy $d_\Upsilon(z_1',z_2') \le d_\Upsilon(z_1,z_2) \le C$, so by Lemma~\ref{lem:proper_dist2} we have $d_{\CTC_K^\Upsilon(Y)}(z_1',z_2') \le \Phi(C)$.  It follows that
$d_{\CTC_K^\Upsilon(Y)}(z_1,z_2) \le C + \Phi(C) \le 2\Phi(C)$, as required.
\end{proof}

We now apply Lemma~\ref{lem:easy retract} to a slightly different situation. 

\begin{lemma} \label{lem:cdr}
Suppose that $K \ge 5\delta_2 + \Delta_0 + \lambda_0 + 1$
Suppose further that $\Upsilon$ is a $\delta_2$--hyperbolic space and that $(\Upsilon,Y)$ is $(K+2\delta_2)$--tube comparable to $(\Xz,\Yx)$.
There exists a $2\Phi(2\delta_2+1)$--deformation retraction
\[	\left\{ f_i \co \CTC^\Upsilon_K(Y) \to \CTC^\Upsilon_K(Y)	\right\}		\]
of $\CTC^\Upsilon_K(Y)$ onto $\CS^\Upsilon_K(Y)$.

Furthermore, if $f$ is the stable map of the sequence $\{ f_i \}$ then for any $z_1,z_2 \in \Upsilon$ with $d_{\Upsilon}(z_1,z_2) \le 1$ we have 
\[  d_{\CS^\Upsilon_K(Y)}(f(z_1),f(z_2)) \le \Phi(2\delta_2+1).  \]
\end{lemma}
\begin{proof}
Recall that
\[	\CTC^\Upsilon_K(Y)  = \left( \Upsilon \ssm T^\Upsilon_K(Y) \right) \sqcup_{S^\Upsilon_K(Y)} \CS^\Upsilon_K(Y)	.	\]
For $i \ge 0$ we define maps $f_i \co \CTC^\Upsilon_K(Y) \to \CTC_K^\Upsilon(Y)$ satisfying the conditions for Definition~\ref{def:def retract}.  If $y \in \CS^\Upsilon_K(Y)$ then define $f_i(y) = y$ for all $i$, as required by Definition \ref{def:def retract}.\eqref{eq:evId}.

On the other hand, suppose that $y \not\in \CS_K^\Upsilon(Y)$.  Then $y \in \Upsilon \ssm T^\Upsilon_K(Y)$.  Let $\{ r_i \co \Upsilon \to \Upsilon \}$ be the $(2\delta_2+1)$--deformation retraction from $\Upsilon$ to $N_K(Y)$ afforded by Lemma~\ref{lem:easy retract}.  By the ``Moreover'' statement in that lemma, since $y \not\in T_K^\Upsilon(Y)$ in $\Upsilon$, for all $i$ we have $r_i(y) \not\in T_K^\Upsilon(Y)$.  Therefore, since $r_i(y) \in \Upsilon \ssm T_K^\Upsilon(Y)$, we can define $f_i(y) = r_i(y) \in \CTC_K^\Upsilon(Y)$.

This defines the maps $\{ f_i \co \CTC_K^\Upsilon(Y) \to \CTC_K^\Upsilon(Y) \}$.  We claim that this sequence of maps is a $2\Phi(2\delta_2+1)$--deformation retraction of $\CTC_K^\Upsilon(Y)$ onto 
$\CS_K^\Upsilon(Y)$, as required.  Indeed, Conditions \eqref{eq:0Id}, \eqref{eq:evId}, \eqref{eq:speed}, and \eqref{stability} all clearly hold by definition and Lemma~\ref{lem:easy retract}.  The only property left from Definition~\ref{def:def retract} is Condition~\ref{eq:CL}, which follows quickly from Lemmas~\ref{lem:CTC dist} and \ref{lem:easy retract}.

Finally, suppose that $f$ is the stable map of the sequence $\{ f_i \}$ and suppose that $z_1, z_2$ in $\Upsilon$ satisfy $d_{\Upsilon}(z_1,z_2) \le 1$.  The construction of $f(z_1)$ and $f(z_2)$ follow by applying the construction from Lemma~\ref{lem:easy retract}, so we have $d_\Upsilon(f(z_1),f(z_2)) \le 2\delta_2+1$ by that lemma.  The definition of the function $\Phi$ from Lemma~\ref{lem:proper_dist} now implies that
\[ d_{\CS^\Upsilon_K(Y)}(f(z_1),f(z_2)) \le \Phi(2\delta_2+1) , \]
as required.
\end{proof}

Recall the constant $A_0$ was fixed in Assumption~\ref{ass:for_injectivity}.
\begin{proposition} \label{prop:cdr on pi_1}
Let $Q_0 = \max\{ 2\Phi(2\delta_2+1), \delta_2\}$, let $D_0 = 3\Phi(C_0)+108\delta_0$ be the constant from Subsection~\ref{ss:noname}, let $D \geq \max\{ D_0 , 2Q_0+2 \}$, let $\sys = Q_0D_0$,  and
 let $R_{\ref{t:Pi_isomorphism}} = R_{\ref{t:Pi_isomorphism}}(\Xz,\Yx,D,\sys,A_0)$ be the constant from that theorem applied with these constants.  Let $K \ge \max\left\{ R_{\ref{t:Pi_isomorphism}}, 5\delta_2+\Delta_0+\lambda_0+1 \right\}$ and let $\Upsilon$ be a connected $\delta_2$--hyperbolic graph with path $\gamma$ so that $(\Upsilon,Y)$ is $(K+2\delta_2)$--tube comparable to $(\Xz,\Yx)$. The inclusion
\[	\CS^\Upsilon_K(Y) \into \CTC^\Upsilon_K(Y)	\]
induces an isomorphism
\[	\pi_1^D\left( \CS^\Upsilon_{K}(Y) \right) \cong	\pi_1^D\left( \CTC^\Upsilon_{K}(Y) \right) .	\]
\end{proposition}

\begin{proof}
According to Lemma \ref{lem:cdr} there exists a $Q_0$--deformation retraction $ f_i \co \CTC^\Upsilon_K(Y) \to \CS^\Upsilon_K(Y)$.  Let $f \co \CTC^\Upsilon_K(Y) \to \CS^\Upsilon_K(Y)$ be the associated stable map.  For any $b_1, b_2 \in \CTC^\Upsilon_K(Y)$ so that $d_{\CTC^\Upsilon_K(Y)}(b_1,b_2) \le 1$ we have $d_{\CS^\Upsilon_K(Y)}(f(b_1),f(b_2)) \le Q_0$. By the choice of $R_{\ref{t:Pi_isomorphism}}$ and Theorem \ref{t:Pi_isomorphism}.\eqref{item:long_loop} all $\CS_K^\Upsilon(Y)$--loops of length at most $Q_0D$ represent the (conjugacy class of the) identity element of $\pi_1(\CS_K^\Upsilon(Y))$.  The result now follows immediately from Proposition \ref{prop:def retract on pi_1}.
\end{proof}

The space $\Xz$ and the quasi-convex subset $\Yx$ satisfy the assumptions of Proposition \ref{prop:cdr on pi_1}.  Therefore, we immediately have the following.
\begin{corollary} \label{cor:cdr on pi_1}
Let $Q_0 = \max\{ 2\Phi(2\delta_2+1), \delta_2\}$.
Let $D_0 = 3\Phi(C_0) + 108\delta_0$ be the constant from Subsection~\ref{ss:noname}, let $D \geq \max\{ D_0 , 2Q_0+2 \}$, let $\sys = Q_0 D_0$ and
 choose $R_{\ref{t:Pi_isomorphism}} = R_{\ref{t:Pi_isomorphism}}(\Xz,\Yx,D,\sys,A_0)$ to be the constant from Theorem~\ref{t:Pi_isomorphism}.  For any $K \ge \max\left\{ R_{\ref{t:Pi_isomorphism}}, 5\delta_2+\Delta_0+\lambda_0+1 \right\}$, the inclusion induces an isomorphism 
\[	\pi_1^D(\CS_K(\Yx)) \cong \pi_1^D\left(\CTC_K(\Yx) \right)	.	\]
\end{corollary}

\section{Unwrapping the complement of a tube} 
\label{sec:unwrap} 

In this section we start with the axis $\gammax$ of the hyperbolic isometry $g$ of the hyperbolic space $\Xz$ with boundary $S^2$, and we study the space constructed (approximately) as follows: remove a large tube around $\gammax$, take the universal cover of the resulting space, and glue a horoball on the lift of the boundary of the tube. We show that this space is hyperbolic (Lemma~\ref{lem:app_of_CCH}) and has boundary $S^2$ (Proposition~\ref{prop:unwrap_S^2}), allowing us to repeat the procedure. However, we need careful control on all the constants involved to make sure that the procedure can be repeated, so we now fix once and for all a set of constants that allows us to do so.

\begin{assumption} \label{ass:X2}
For this section, and the remainder of the paper, assume we are in the setting of Assumption \ref{ass:X}.  Assume further that $\partial \Xz \cong S^2$, and that $\Yx=\gammax$ is a  (bi-infinite) continuous quasi-geodesic in $\Xz$ which is $\lambda_0$--quasi-convex, and let $g$ be a  loxodromic element $g$ for which $\gammax$ is an axis, and let $A_0$ be a constant satisfying Assumption~\ref{ass:for_injectivity} (for the $\langle g \rangle$--action on $\gammax$).  
\end{assumption}

In particular, $\Xz$ is $\delta_0$-hyperbolic and has boundary homeomorphic to $S^2$.

We now fix a collection of constants needed later in the construction and the proof of Theorem \ref{t:Drill}.  Further constants are chosen at the beginning of Section~\ref{s:unwrap family}.

\begin{notation}\label{not:first_constants}\ 

\begin{enumerate}
\item $s_0 = 8\delta_0$, the scale of coarse path connectedness of shells;
\item $\delta_1$ is the hyperbolicity and visibility constant of $\Xz^{\mathrm{cusp}}(K)$ (see Definition~\ref{def:cusping}), for sufficiently large values of $K$.  We assume $\delta_1 \ge \max\{\delta_0,100\}$;
\item $\delta_2 = 1500\delta_1$ (5 times the hyperbolicity constant from the conclusion of the Coarse Cartan-Hadamard Theorem~\ref{t:CCH}); 
\item $L_0$ is the constant of linear connectedness of $\partial \Xz$, from Assumption \ref{ass:X};
\item $Q_0 = \max\{ 2\Phi(2\delta_2+1), \delta_2\}$ is the constant from Proposition \ref{prop:cdr on pi_1}, where $\Phi$ is the function from  Lemma~\ref{lem:proper_dist};
\item $D_0 = 3\Phi(C_0) + 108\delta_0$ is the constant from Subsection~\ref{ss:noname} (just before the statement of Lemma~\ref{lem:claim_localpath});  
\item $D_1 = \max\{D_0, 2Q_0+2, 16\delta_0+16\delta_0\Phi(16\delta_0)\}$, the scale of loops for defining the coarse fundamental group;
\item $\sigma_0 = \max\{ 10^7 \delta_1, 10^5D_1 \}$, the first requirement is the size of balls needed for the Coarse Cartan-Hadamard to apply, the second is used in the proof of Theorem~\ref{thm:longfilling};
\item $\sys_0 = 2^{25\sigma_0} Q_0$ is, because of Theorem \ref{t:Pi_isomorphism}.\eqref{item:long_loop} and the choice of $R_0$ in the item below, a lower bound for the length of a path in the completed shell $\CS_{R_0}(\gammax)$ representing a nontrivial conjugacy class in $\pi_1^{D_1}(\CS_{R_0}(\gammax))$.  This value of $\sys_0$ is chosen to make the proof of Lemma \ref{lem:isom_of_balls} below work;
\item $R_0=\max\{R_{\ref{t:Pi_isomorphism}}, \lambda_0+2\delta_2,6\sigma_0\}$ where $R_{\ref{t:Pi_isomorphism}} = R_{\ref{t:Pi_isomorphism}}(X_0,\gammax,D_1,\sys_0,A_0)$ is the constant from Theorem \ref{t:Pi_isomorphism}.  This is the radius of the tubes and shells that we henceforth consider.  
\end{enumerate}
\end{notation}

\begin{notation}
Having fixed the constants above, for the remainder of the paper, we always take completed $R_0$--shells, and completed $R_0$--tube complements.  Therefore, we henceforth drop the subscript $R_0$ from our notation.

Thus, for a space $\Upsilon$, and subset $\gamma$, we write
\[  \SCS_\Upsilon(\gamma) = \CS_{R_0}^{\Upsilon}(\gamma), \text{ the completed shell}   \]
and
\[  \SCTC_\Upsilon(\gamma) = \CTC_{R_0}^{\Upsilon}(\gamma), \text{ the completed tube complement}.\]

We are particularly interested in the completed shells and tube complements associated to $\gammax$ in $\Xz$, and for these, we further drop all of the decoration, and write
\[  \SCS = \CS_{R_0}^{\Xz}(\gammax)   \]
and
\[  \SCTC = \CTC_{R_0}^{\Xz}(\gammax).   \]
\end{notation}

(The definitions of completed shell and completed tube complement are \ref{def:completed shell} and \ref{def:CTC}, and the definition of \emph{tube comparable} is Definition~\ref{def:comparable}.)
The next result summarizes various results from the previous section.  

\begin{proposition}\label{p:summary}
Let $\Xz$ be the space from Assumption \ref{ass:X2}.  Suppose that $\Upsilon$ is $\delta_2$--hyperbolic and that $(\Upsilon,\gamma)$ is $(R_0 + 2\delta_2)$--tube comparable with $(\Xz,\gammax)$.
With the choice of constants as in Notation \ref{not:first_constants}, the following hold:
\begin{enumerate}
\item\label{eq:pi_1X} $\pi_1^{D_1} \left( \SCTC \right) \cong \pi_1^{D_1}(\SCS) \cong \Z$;
\item\label{eq:systole} For any loop $\alpha$ in $\SCS$, if $[\alpha] \ne 1$ in $\pi_1^{D_1}(\SCS)$ then the length of $\alpha$ is at least $\sys_0$;
\item\label{eq:systole2} For any loop $\alpha'$ in $\SCTC_\Upsilon(\gamma)$, if $[ \alpha'] \ne 1$ in $\pi_1^{D_1}\left(\SCTC_\Upsilon(\gamma)\right))$ then the length of $\alpha'$ is at least $\sys_0/Q_0$;
\item\label{eq:pi_1Y} There are isomorphisms
\[	 \pi_1^{D_1}\left(\SCTC_\Upsilon(\gamma)\right) \cong \pi_1^{D_1}\left(\SCS_\Upsilon(\gamma)\right)	= \pi_1^{D_1}(\SCS) \cong \Z . \]	
  The first isomorphism is induced by inclusion.
\end{enumerate}
\end{proposition}
\begin{proof}
The first isomorphism in \eqref{eq:pi_1X} follows from Corollary \ref{cor:cdr on pi_1} and the second from Corollary~\ref{cor:pi_1_Z}.

Item \eqref{eq:systole} is Theorem \ref{t:Pi_isomorphism}.\eqref{item:long_loop}.
For Item \eqref{eq:systole2},  suppose that $\alpha$ is an edge-loop in $\SCTC_\Upsilon$ of length less than $\sys_0/Q_0$. Lemma~\ref{lem:cdr} (include the ``Furthermore" statement) imply that the hypotheses of Proposition~\ref{prop:def retract on pi_1} hold (Item~\eqref{item:defret1} from the statement holds by the assumption on $\sys_0$).  Therefore, by the ``Moreover" part of the conclusion of Proposition~\ref{prop:def retract on pi_1}, there is a loop $f(\alpha)$ on $\SCS_\Upsilon$ of length less than $\sys_0$ which represents the same conjugacy class in $\pi_1^{D_1}(\SCTC_\Upsilon)$.  Since $\SCS_\Upsilon = \SCS$, we can interpret $f(\alpha)$ as a loop on $\SCS$, and by Item~\eqref{eq:systole}, $f(\alpha)$ represents the identity element of $\pi_1^{D_1}(\SCS) = \pi_1^{D_1}(\SCS_\Upsilon) \cong \pi_1^{D_1}(\SCTC_\Upsilon)$,  Therefore, $\alpha$ also represents the trivial element, proving Item~\eqref{eq:systole2}.
Item \eqref{eq:pi_1Y} follows immediately from Proposition \ref{prop:cdr on pi_1} Item~\ref{eq:pi_1X} and the fact that $\SCS_\Upsilon = \SCS$.
\end{proof}

\subsection{Unwrapping}

In the rest of this section we  study the spaces described below:

\begin{notation}
\label{not:H}
Let $\horba$ be the combinatorial horoball $\horba(\widetilde{\SCS})$ based on the $D_1$--universal cover $\widetilde{\SCS}$ of $\SCS$.  Let $\mathring{\horba} = \horba \ssm \widetilde{\SCS}$ be the set of points in $\horba$ at positive depth. By $\widetilde{\SCTC}$ we denote the $D_1$--universal cover of $\SCTC$.
\end{notation}

By Proposition \ref{p:summary}.\eqref{eq:pi_1X}, the $D_1$--universal covers $\widetilde{\SCS} \to \SCS$, and $\widetilde{\SCTC} \to \SCTC$ are $\Z$--covers.

The following is an analogue of Definition~\ref{def:cusping}.
The construction is well-defined since by Proposition \ref{p:summary}.\eqref{eq:pi_1Y} the preimage of the completed shell is its $D_1$--universal cover.

\begin{definition} \label{def:zcusp}
Suppose that $\Upsilon$ is $\delta_2$--hyperbolic and that $(\Upsilon,\gammay)$ is $(R_0+2\delta_2)$--tube comparable to $(\Xz,\gammax)$.  The space $\Upsilon^{\mathrm{cusp}}$ is obtained by gluing a combinatorial horoball $\horbaD$ based on $\CS^\Upsilon(\gammay) \cong \SCS$ to $\CTC^\Upsilon(\gammay)$.
\end{definition}
In particular, $\Xz^{\mathrm{cusp}}$ denotes $\Xz^{\mathrm{cusp}}(R_0)$, as defined in Definition~\ref{def:cusping}.

\begin{construction}\label{unwrap_one_tube} (Unwrap and glue for one tube) 
Suppose that  $\Upsilon$ is $\delta_2$--hyperbolic, and that $(\Upsilon,\gamma)$ is $(R_0+2\delta_2)$--tube comparable to $(\Xz,\gammax)$.

Let $\UG(\Upsilon, \gammay)$ be obtained by taking the $D_1$--universal cover of $\CTC^\Upsilon(\gammay)$ and gluing a copy of $\horba$ along the preimage of $\CS^\Upsilon (\gammay)$.

The covering map 
\[	q_{R_0} \co \widetilde{\CTC^\Upsilon (\gammay)} \to \CTC^\Upsilon (\gammay)	\]
extends to a continuous depth-preserving map
\[	q \co \UG(\Upsilon, \gammay) \to 	\Upsilon^{\mathrm{cusp}} \]
\end{construction}
Some edges in $\horba$ factor to loops under the natural quotient map, and these loops are not in the horoball based on $\CS ^\Upsilon(\gamma)$.  These loops are mapped to their endpoint.

The map $q$ turns out to be an isometry on balls of radius $\sigma_0$ with center not too deep in $\horba$. We show this in Lemma \ref{lem:isom_of_balls} below, after stating and proving the following general lemma on, roughly, promoting local homeomorphisms to local isometries.

\begin{lemma}
\label{lem:R/3}
Let $\Gamma,\Xi$ be metric graphs, and let $B_R(p)$ be a ball of radius $R$ in $\Gamma$ for some integer $R> 3$. Suppose that we have an injective simplicial map $f\co B_R(p)\to \Xi$ which is a local homeomorphism at every point of $B_{R-1}(p)$. Then $f$ restricts to an isometry $f'\co B_{R/3}(p)\to B_{R/3}(f(p))$.
\end{lemma}

\begin{proof}
Note that the restriction $f'$ of $f$ to $B_{R/3}(p)$ is $1$--Lipschitz, so that its image is indeed contained in $B_{R/3}(f(p))$. Therefore it suffices to construct a $1$--Lipschitz map $g:B_{R/3}(f(p))\to B_{R/3}(p)$ such that $g\circ f'=id$.

To construct the map, for $x\in B_{R/3}(f(p))$ we choose a geodesic from $f(p)$ to $x$, we lift it starting at $p$, and we define $g(x)$ to be the endpoint of the lift. Note that indeed $g(x)\in  B_{R/3}(p)$. By construction we have $g(f'(x))\in f^{-1}(f'(x))$ for each $x\in B_{R/3}(p)$, and since $f$ is injective we have $f^{-1}(f'(x))=\{x\}$. That is, $g(f'(x))=f'(x)$, and $g\circ f'=id$. We are left to show that $g$ is $1$--Lipschitz. This is because any geodesic connecting points in $B_{R/3}(f(p))$ is contained in $B_{2R/3}(f(p))$, and hence can be lifted to a path in $B_{2R/3}(p)$. (Note that $R-1 > 2R/3$ since $R >
3$.)
\end{proof}

\begin{lemma} \label{lem:isom_of_balls}
Suppose that  $\Upsilon$ is $\delta_2$--hyperbolic, and that $(\Upsilon,\gamma)$ is $(R_0+2\delta_2)$--tube comparable to $(\Xz,\gammax)$.
For any point $z \in \UG(\Upsilon, \gammay)$ that either does not lie in $\horba$ or lies at depth at most $3\sigma_0$ in $\horba$ the map $q$ as above is an isometric embedding when restricted to the ball $B \left(z ,3\sigma_0 \right)$, and moreover
\[	q \left( B \left( z, 3\sigma_0 \right) \right) = B \left( q(z), 3\sigma_0 \right)	.	\]
\end{lemma}

\begin{proof}
Let $z \in \UG(\Upsilon, \gammay)$ be as in the statement of the lemma.  We claim that the map $q$ is injective on $B(z,{9\sigma_0})$.  This and Lemma~\ref{lem:R/3} immediately implies the desired conclusion, since $q$ is then a covering on such balls.

In order to obtain a contradiction, suppose $z_1, z_2 \in B(z,9\sigma_0)$ are distinct but satisfy $q(z_1) = q(z_2)$.  Note that $d_{\UG(\Upsilon,\gammay)}(z_1,z_2) \le 18 \sigma_0$.  Since $z$ has depth at most $3\sigma_0$, each of $z_1$ and $z_2$ have depth at most $12\sigma_0$.  If the depth of $z_1$ or $z_2$ is positive, we replace it with a point on $\widetilde{\SCS} $ directly above it (this corresponds to replacing a point $(x,n)$ with a point $(x,0)$ in Definition~\ref{def:combhoroball}).  We may therefore suppose, at the cost of assuming that $d_{\UG(\Upsilon,\gammay)}(z_1,z_2) \le 42\sigma_0$, that $z_1$ and $z_2$ do not lie in the open horoball.

The points $z_1$ and $z_2$ may be joined by a path in 
$\widetilde{\SCTC_\Upsilon}$ whose image represents a nontrivial element of $\pi_1^{D_1}(\SCTC_\Upsilon)$.  Therefore, by Proposition \ref{p:summary}.\eqref{eq:systole2}, we have $d_{\widetilde{\SCTC_\Upsilon}}(z_1,z_2) \ge \sys_0/Q_0=2^{25\sigma_0}$, by the choice of $\sys_0$ in Notation~\ref{not:first_constants}.  On the other hand, $\UG(\Upsilon,\gammay)$ is built from $\widetilde{\SCTC_\Upsilon}$ by gluing on a copy of a horoball.  Any path $\sigma$ in $\UG(\Upsilon,\gammay)$ may be pushed out the horoball at the cost of increasing the length by an exponential function.  Precisely, we have (using Lemma~\ref{lem:horbadistort} and the fact that any subsegment that is pushed out of the horoball has length at least $4$):
\[	d_{\widetilde{\SCTC_\Upsilon}} (z_1,z_2) \le 2 \cdot 2^{\frac{1}{2}d_{\UG(\Upsilon, \gammay)}(z_1,z_2)} \le 2^{21\sigma_0+1} < 2^{25\sigma_0}	.	\]
This is a contradiction. 
\end{proof}

From Lemma \ref{lem:isom_of_balls}, we control the shape of balls in $\UG(\Upsilon,\gamma)$ of reasonable size:

\begin{corollary}\label{cor:local_balls_unwrapped} Let $\sigma_0$ and ${R_0}$ be as fixed in Notation \ref{not:first_constants}. Suppose that  $\Upsilon$ is $\delta_2$--hyperbolic, and that $(\Upsilon,\gamma)$ is $({R_0}+6\sigma_0+2\delta_2)$--tube comparable to $(\Xz,\gammax)$.

 Any ball of radius $\sigma_0$ in $\UG(\Upsilon, \gammay)$ is isometric to one of
 \begin{itemize}
 \item[--] a ball in $\mathring{\horba}$,
 \item[--] a ball in $\Upsilon$, or
 \item[--] a ball in $\Xz^{\mathrm{cusp}}$.
 \end{itemize}
 \end{corollary} 

\begin{proof}
Consider a ball $B$ as in the statement, and let $B^+$ be the ball with the same center and radius $3\sigma_0$. By Lemma \ref{lem:isom_of_balls} either $B^+$ is entirely contained in $\mathring{\horba}$ (in which case Lemma \ref{lem:R/3} yields the required isometry), or $q$ maps the ball $B^+$ isometrically to a ball in $\Upsilon^{\mathrm{cusp}}$, whose center lies within $3\sigma_0$ of $\CTC ^\Upsilon(\gamma)$. The second case has two sub-cases. If the ball $q(B^+)$ is disjoint from the horoball of $\Upsilon^{\mathrm{cusp}}$, then $q(B^+)$ is isomorphic to a ball in $\Upsilon$, and as $B$ is isometric to a ball in $\Upsilon$ by Lemma \ref{lem:R/3}.  Otherwise, notice that the $6\sigma_0$--neighborhood of the horoball of $\Upsilon^{\mathrm{cusp}}$ is isomorphic to the corresponding neighborhood in $\Xz^{\mathrm{cusp}}$, since $(\Upsilon,\gamma)$ is $({R_0} +6\sigma_0+2\delta_2)$--tube comparable to $(\Xz, \gammax)$. We can then again use Lemma \ref{lem:R/3} to conclude.
\end{proof}

We now control the local geometry of an unwrapped and glued space.  

\begin{definition}
Suppose that $\mc{A} = (Z_1, \ldots , Z_n)$ is a finite collection of metric spaces and that $k > 0$.  We say that a metric space $Z$ is {\em $k$--modeled on $\mc{A}$} if for every $z \in Z$ there is an $i$ so that the closed ball of radius $k$ about $z$ is isometric to a closed ball of radius $k$ in $Z_i$.
\end{definition}

\begin{lemma} \label{lem:app_of_CCH}
If $Y$ is connected, $D_1$--simply-connected, and $\sigma_0$--modeled on $\left\{ \Xz^{\mathrm{cusp}},\mathring{\horba} \right\}$ then $Y$ is $\delta_2$--hyperbolic and $\delta_2$--visible.
\end{lemma}

\begin{proof}
Hyperbolicity, with constant $\delta'=300\delta_1$ is an immediate consequence of the Coarse Cartan--Hadamard Theorem \ref{t:CCH}, along with the choices of constants and the fact that both $\Xz^{\mathrm{cusp}}$ and $\mathring{\horba}$ are $\delta_1$--hyperbolic.

For visibility with constant $5\delta' = \delta_2$, by Proposition~\ref{lem:local_visibility} it suffices to prove that if $p,q\in Y$ satisfy $d_Y(p,q)\leq 100\delta'$ there exists $q'\in Y$ with $d(p,q')\geq 200\delta'$ and a geodesic $[p,q']$ that passes within $\delta'$ of $q$.

Note that the closed $200\delta'$--ball around $p$ is compact, and therefore (since $\sigma_0\geq 200\delta'=60000\delta_1$) it is isometric to either a ball in $\Xz^{\mathrm{cusp}}$, or to a ball of $\horba$ whose center is deeper than its radius.\footnote{this is the reason why we model on $\mathring{\horba}$ rather than $\horba$.} In either case, we can find the required $q'$. Indeed, for $\horba$ we can use the explicit description of geodesics (cf. \cite[Lemma 2.23]{GMS}), and for $\Xz^{\mathrm{cusp}}$ we just apply $\delta_1$--visibility.
\end{proof}

\begin{proposition} \label{prop:hatY_modeled}
Let $\Upsilon$ be a connected $D_1$--simply connected graph, and let $\gamma$ be a bi-infinite path in $\Upsilon$. Assume that $\Upsilon$ is $\sigma_0$--modeled on $\left\{ \Xz^{\mathrm{cusp}}, \mathring{\horba} \right\}$.
Suppose further that $(\Upsilon,\gammay)$ is $({R_0}+6\sigma_0+2\delta_2)$--tube comparable to $(\Xz,\gammax)$.

Then $\UG(\Upsilon,\gammay)$ is connected, $D_1$--simply connected, and $\sigma_0$--modeled on $\left\{ \Xz^{\mathrm{cusp}}, \mathring{\horba} \right\}$.  In particular, $\UG(\Upsilon,\gammay)$ is $\delta_2$--hyperbolic and $\delta_2$--visible.
\end{proposition}

\begin{proof}
Proposition \ref{prop:CvK} implies that $\UG(\Upsilon,\gamma)$ is  connected and $D_1$--simply connected. 

The $\sigma_0$--modeled statement follows directly from Corollary \ref{cor:local_balls_unwrapped} and the assumption that $\Upsilon$ is $\sigma_0$--modeled on $\left\{ \Xz^{\mathrm{cusp}}, \mathring{\horba} \right\}$.  The last assertion of the proposition now follows immediately from Lemma~\ref{lem:app_of_CCH}.
\end{proof}

\begin{definition}
\label{def:separated_tubes}
Let $\sep > 10R_0+ 60\sigma_0 +20\delta_2$ be a constant.
Suppose that $\Upsilon$ is a graph which is $\sigma_0$--modeled on $\left\{ \Xz^{\mathrm{cusp}}, \mathring{\horba} \right\}$.  A collection of subsets $\mc{A}$ in $\Upsilon$ is a {\em $\sep$--separated collection of tubes and horoballs} if the following hold:
\begin{enumerate}
    \item Each component of $\mc{A}$ is either isometric to $T_{R_0}(\gammax)$ ($=T^{\Xz}_{R_0}(\gammax)$) and then called a \emph{tube}, or to $\mathring{\horba}$, and called a \emph{horoball}.
    \item Each tube $A\in \mc{A}$ contains a specified path $\gamma(A)$ that we call the \emph{core}, and the pair $(\Upsilon,\gamma(A))$ is $\frac{\sep}{10}$--tube comparable to $(\Xz,\gammax)$.
    \item For each horoball $A \in \mc{A}$ the $\frac{\sep}{10}$--neighborhood of $A$ in $\Upsilon$ is isomorphic to the $\frac{\sep}{10}$--neighborhood of the horoball in $\UG(\Xz,\gammax)$.
    \item Any two components of $\mc{A}$ are at distance at least $\sep$ apart.
\end{enumerate}
\end{definition}

\begin{remark}\label{rem:separatedisomorphmetric}
  Assuming that $\Upsilon$ is $\delta_2$--hyperbolic, each tube or horoball in a collection as described above is $2\delta_2$--quasiconvex.  In particular the metric on such a tube or horoball is determined by the isomorphism type of a $2\delta_2$--neighborhood.
\end{remark}

In Section~\ref{s:unwrap family} we need the following construction.

\begin{construction}\label{con:hollow}
Suppose that $\Upsilon$ is a graph which is $\sigma_0$--modeled on $\left\{ \Xz^{\mathrm{cusp}}, \mathring{\horba} \right\}$ and that $\mc{A}$ is a $\sep$--separated collection of tubes and horoballs in $\Upsilon$ (for some $\sep > 10R_0+20\delta_2$).

Let $A$ be the union of all the cores of tubes of $\mc{A}$. Let $\Upsilon'$ be obtained from $\Upsilon$ by removing all open horoballs of $\mc{A}$.  The \emph{hollowed-out space associated to $\Upsilon$} is the space
\[  \mathrm{HO}(\Upsilon,\mc{A})= \CTC^{\Upsilon'} (A) . \]
\end{construction}

\begin{lemma}\label{lem:sheath}
Let $\tau > R_0$ and suppose that $\Upsilon$ is a graph which is $\sigma_0$--modeled on $\left\{ \Xz^{\mathrm{cusp}}, \mathring{\horba} \right\}$, and that $\gamma$ is a path in $\Upsilon$ so that $(\Upsilon,\gamma)$ is $\tau$--tube comparable to $(\Xz,\gammax)$.  Let $\mathrm{Ann}_\tau = N_\tau^\Upsilon(\gamma) \smallsetminus T_{R_0}^\Upsilon(\gamma)$.

The preimage of $\mathrm{Ann}_\tau$ in $\UG(\Upsilon,\gamma)$ is isomorphic to the preimage of $N_\tau(\gammax) \smallsetminus T_{R_0}(\gammax)$ in $\UG(\Xz,\gammax)$.
\end{lemma}
\begin{proof}
The deformation retraction $\CTC ^\Upsilon(\gamma) \to \CS ^\Upsilon(\gamma)$ from Lemma~\ref{lem:cdr} restricts to a deformation retraction $\mathrm{Ann}_\tau \to \CS ^\Upsilon(\gamma)$, so the inclusion $\CS ^\Upsilon(\gamma) \into \mathrm{Ann}_\tau$ induces an isomorphism on $\pi_1^D$ by Proposition~\ref{prop:def retract on pi_1}.
Similarly, there is a deformation retraction from $\CTC ^\Upsilon(\gamma)$ onto $\mathrm{Ann}_\tau$ (built just as in the proof of Lemma~\ref{lem:cdr}), so the inclusion $\mathrm{Ann}_\tau \into \CTC ^\Upsilon(\gamma)$ induces an isomorphism on $\pi_1^D$.

It follows that the coarse universal cover of $\mathrm{Ann}_\tau$ embeds in that of $\CTC ^\Upsilon(\gamma)$, containing the distinguished copy of $\widetilde{\CS ^\Upsilon(\gamma)}$ onto which the horoball $\horba$ is glued.

The same considerations for the construction of $\UG(\Xz,\gammax)$ imply the result.
\end{proof}

\begin{lemma}\label{lem:horbaproject}
Let $(\Upsilon,\gamma)$ be as in Construction~\ref{unwrap_one_tube}. 
Let $p:\widetilde{\CTC ^\Upsilon(\gamma)} \to  \widetilde{\CS ^\Upsilon(\gamma)}$ be the ``projection'' map defined by considering an (equivariantly chosen) ray $r_x$ in $\UG(\Upsilon,\gamma)$ starting at $x$ and limiting to the point at infinity of $\horba$ and letting $p(x)$ be the intersection of $r_x$ and the horosphere $\widetilde{\CS ^\Upsilon(\gamma)}$.

Let $[x,y]$ be a geodesic in $\UG(\Upsilon,\gamma)$ that does not intersect $\horba$. 
Then $diam_{\horba}(p([x,y]))\leq 8 \delta_2 +4$.
\end{lemma}

\begin{proof}
By Proposition \ref{prop:hatY_modeled} and Lemma \ref{lem:app_of_CCH} we have that $\UG(\Upsilon,\gamma)$ is $\delta_2$--hyperbolic.

To prove the lemma it suffices to show that $d(p(x),p(y))\leq 8\delta_2+4$ (since we can then apply this to sub-geodesics of a given geodesic).

We can form a geodesic quadrangle containing $[x,y]$, two segments of the rays $r_x,r_y$ and a geodesic of length at most $1$ and at depth larger than $4\delta_2+2$.

Consider $x'$ in $\horba$ along $r_x$ at distance exactly $2\delta_2+2$ from $p(x)$. Since $[x,y]$ does not intersect $\horba$, $x'$ is not $2\delta_2$--close to $[x,y]$, and is also not $2\delta_2$--close to the geodesic of length $1$. Therefore, $x'$ is $2\delta_2$ close to a point $w$ on $r_y$, and in fact $6\delta_2+2$--close to $p(y)$ (since $w$ is at depth at most $4\delta_2+2$).  Recalling that $x$ is exactly $2\delta_2+2$ from $x'$, we obtain the desired bound on $d(p(x),p(y))$.
Notice that all the distances in this argument occur in $\horba$, completing the proof of the lemma.
\end{proof}

\begin{lemma}\label{lem:injectiveballs}
    With the same assumptions as in the previous lemma, if $B$ is any ball in $\UG(\Upsilon,\gamma)$ disjoint from $\horba$, then the covering map $\widetilde{\CTC ^\Upsilon(\gamma)}\to \CTC ^\Upsilon(\gamma)$ is injective on $B$.
\end{lemma}
\begin{proof}
    Lemma~\ref{lem:horbaproject} implies that the diameter (measured in $\horba$) of the projection of $B$ to the horosphere is at most $16\delta_2 + 8$.

    Suppose that there are two points $x$, $y$ of $B$ which project to the same point of $\CTC ^\Upsilon(\gamma)$.  Then there is an element of the deck group of $\widetilde{\CTC ^\Upsilon(\gamma)}\to \CTC ^\Upsilon(\gamma)$ which takes $x$ to $y$, and hence $p(x)$ to $p(y)$. 
    
    In particular, this element has translation length in $\widetilde{\CS ^\Upsilon(\gamma)} \cong \SCS$ at most $2\cdot 2^{\frac{1}{2}(16\delta_2 + 8)}<2^{16\delta_2}$ (the exponential difference coming from the difference between $\CS ^\Upsilon(\gamma)$ and $\horba$ -- see Lemma~\ref{lem:horbadistort}).  But this contradicts the choice of $\sys_0$ in Notation~\ref{not:first_constants}.
\end{proof}

\begin{proposition} \label{prop:hatY_separated}
Suppose that $\Upsilon$ is $D_1$--simply-connected, $\delta_2$--hyperbolic graph.  Suppose further that, for some $\sep>10R_0+60\sigma_0+20\delta_2$, $\mc{A}$ is a $\sep$--separated collection of tubes and horoballs in $\Upsilon$.

Let $C$ be a tube of $\mc{A}$ 
and let $\widehat{\Upsilon}= \UG(\Upsilon,\gamma(C))$.  Let $\widehat{\mc{A}}$ be the collection consisting of the horoball added to $\widetilde{\CTC ^\Upsilon(\gamma(C))}$ and all of the lifts of the elements of $\mc{A}\smallsetminus\{C\}$ to $\widetilde{\CTC ^\Upsilon(\gamma(C))}\subset \widehat{\Upsilon}$.  Then $\widehat{\mc{A}}$ forms a $\sep$--separated collection of tubes and horoballs in $\widehat{\Upsilon}$.  Moreover if $A\in \widehat{\mc{A}}$ is a lift of a tube (resp. horoball) of $\mc{A}$ then it is a tube (resp. horoball).
 \end{proposition}

\begin{proof}
First of all, note that all $\frac{\sep}{10}$--neighborhoods of horoballs and tubes other than $C$ do lift, to horoballs and tubes, respectively. This is because they are $D_1$--simply-connected. Horoballs are $D_1$--simply connected by Lemma \ref{lem:horoballs csc}.  One can see that tubes are $D_1$--simply connected by taking a loop in some tube, coning to the quasi-geodesic, and filling a triangulation of the resulting loop.  The resulting triangulated disk can be cut into loops of length at most $3\delta_0 + 3 < D_1$.

The point (2) of Definition~\ref{def:separated_tubes} follows from the fact that the cores $\gamma(A)$ of tubes all have $\frac{\sep}{10}$--neighborhoods isometric to $T_{\frac{\sep}{10}}(\gammax)$ and so each component  of the preimage of a tube which lifts also has such a neighborhood.

Similarly, a lifted horoball clearly still satisfies the point (3).  Let $\horba'$ be the new horoball glued onto the unwrapped shell which is the coarse universal cover of the shell around $C$.  
By Lemma~\ref{lem:sheath} , the $\frac{\chi}{10}$--neighborhood of $\horba'$ in $\widetilde{\CTC ^\Upsilon(\gamma(C))}$ is isomorphic to the $\frac{\chi}{10}$--neighborhood of the distinguished horoball in $\UG(\Xz,\gammax)$.  Thus the point (3) from Definition~\ref{def:separated_tubes} holds also.

Regarding the point (4), it is clear that lifts of distinct elements of $\mc{A}$ are at least as far apart as the elements. 

Denote by $\horba$ the horoball added to $\widetilde{\CTC ^\Upsilon(\gamma(C))}$.
Let $p:\widetilde{\CTC ^\Upsilon(\gamma(C))} \to  \widetilde{\CS ^\Upsilon(\gamma(C))}$ be the ``projection" map considered in Lemma~\ref{lem:horbaproject}, with $\gamma = \gamma(C)$.

By Lemma~\ref{lem:horbaproject}, for any $A\in\widehat{\mc{A}}\ssm\{\horba\}$ we have that ${\mathrm{diam}}(p(A))\leq 24\delta_2+12.$ Indeed, any two points of $A$ can be joined by a concatenation of 3 geodesics that do not intersect $\horba$.

What we are left to show is that given two lifts $\widehat A,\widehat A'$ of the same element $A$ of $\mc{A}\ssm\{C\}$, the two lifts are $\sep$--separated. Consider a shortest geodesic from $A$ to $C$. We can lift this to paths starting at $\widehat A,\widehat A'$ and ending at $\horba$, with endpoints $\widehat c$, $\widehat c'$. Since $\widehat c$ and $\widehat c'$ are translates of each other by a non-trivial deck transformation, by the choice of $\sys_0$ from Notation \ref{not:first_constants} we have that they are further than $100 \delta_2$ apart (this is not optimal). Moreover, $\widehat c$, $\widehat c'$ are (or more precisely, can be taken to be) the projections of their starting points to $\horba$, any two points in $\widehat A,\widehat A'$ have projections more than $6\delta_2+2$ apart, and therefore by Lemma~\ref{lem:horbaproject} any geodesic connecting them intersects $\horba$. Since $\widehat A,\widehat A'$ are at least $\sep$ away from $\horba$, they are at least $\sep$ (in fact, $2\sep$) apart.
\end{proof}

Recall the element $g$ fixed in Assumption~\ref{ass:X2}.  The isometry $g$ naturally induces a fixed point free isometry of $\SCS^{D_1}$, which we continue to denote by $g$.

\begin{lemma} \label{lem:pi_CS/g}
There is an isomorphism
 \[ E=\pi_1\left(\leftQ{\SCS^{D_1}}{\langle g\rangle}\right) \cong \pi_1\left(\leftQ{\partial \Xz \smallsetminus \{\gammax^\pm\}}{\langle g \rangle}\right) .   \] 
Moreover, $E$ is a semidirect product of infinite cyclic groups
\[ E = \langle a \rangle \rtimes \langle b \rangle.\]
Realizing $E$ as the deck group of the universal cover 
$\widetilde{\SCS^{D_1}}\to \leftQ{\SCS^{D_1}}{\langle g\rangle}$, the element $a$ is a generator of the deck group of the intermediate cover $\widetilde{\SCS^{D_1}}\to \SCS^{D_1}$ and $b$ is a lift of $g$ to $\widetilde{\SCS^{D^1}}$.

 Finally, $E$ acts geometrically on $\widetilde{\SCS}$, so that both $E$ and $\widetilde{\SCS}$ are quasi-isometric to $\mathbb{R}^2$.
\end{lemma}

\begin{proof}
The first statement follows immediately from Theorem \ref{t:Pi_isomorphism} and Corollary~\ref{cor:pi_1_Z}. Indeed, $\pi_1\left(\leftQ{\partial \Xz \smallsetminus \{\gammax^\pm\}}{\langle g \rangle}\right)$ is isomorphic to semi-direct product of $\pi_1(\partial \Xz \smallsetminus \{\gammax^\pm\})$ by $\langle g\rangle$, where $g$ acts on $\pi_1(\partial \Xz \smallsetminus \{\gammax^\pm\})$ via the induced action on $\pi_1$. The other $\pi_1$ admits a similar description, with the action on $\pi_1$ being compatible with the isomorphism given by item Theorem~\ref{t:Pi_isomorphism}.\eqref{item:boundary_to_shell} in view of item Theorem~\ref{t:Pi_isomorphism}.\eqref{item:diagram}.

The description of the deck group follows from the fact that we can describe the universal covering map as a composition of two covers $\widetilde{\SCS^{D_1}}\to \SCS^{D_1}$ and $\SCS^{D_1}\to \leftQ{\SCS^{D_1}}{\langle g\rangle}$, the first cover having deck group $\pi_1^{D_1}(\SCS)$ and the second cover having deck group $\langle g\rangle$.
\end{proof}

\begin{proposition}\label{prop:peripheral_is_planar}
 The action of the group $E$ from Lemma \ref{lem:pi_CS/g} on $\widetilde{\SCS}$ extends to an action on $\widetilde{\SCTC}$.
\end{proposition}

\begin{proof}
This follows from the fact that the inclusion of $\leftQ{\SCS^{D_1}}{\langle g\rangle}$ into $\leftQ{\SCTC^{D_1}}{\langle g \rangle}$ induces an isomorphism on fundamental groups, by Proposition \ref{prop:cdr on pi_1} and the fact that the action of $g$ commutes with the inclusion of $\SCS^{D_1}$ into $\SCTC^{D_1}$.
\end{proof}

\subsection{Unwrapping: boundaries}

\begin{proposition}
\label{prop:deck_Z}
Let $\Upsilon$ be a connected $D_1$--simply connected graph, and let $\gamma$ be a bi-infinite path in $\Upsilon$. Assume that $\Upsilon$ is $\sigma_0$--modeled on $\left\{ \Xz^{\mathrm{cusp}}, \mathring{\horba} \right\}$.
Suppose further that $(\Upsilon,\gammay)$ is $({R_0}+6\sigma_0+2\delta_2)$--tube comparable to $(\Xz,\gammax)$.

Let $p$ be the point at infinity of the horoball $\horba$ of $\partial \UG(\Upsilon, \gammay)$.  There exists a regular covering map $\Theta:\partial \UG(\Upsilon, \gammay) \ssm \{p\}\to \partial \Upsilon \ssm\{\gammay^{\pm}\}$, and the deck transformation group is $\mathbb Z$.

Moreover, if $(\Upsilon,\gamma)=(\Xz,\gammax)$ 
 then $\Theta$ is the universal covering map.
\end{proposition}

\begin{proof}
Let us construct the covering map $\Theta$. We have a ($D_1$--universal) covering map $q_{R_0} \co \widetilde{\CTC^\Upsilon (\gammay)} \to \CTC^\Upsilon (\gammay)$ as in Construction \ref{unwrap_one_tube}.

Fix $w \in \partial \UG(\Upsilon,\gammay) \ssm \left\{ p \right\}$. We claim that we can choose a geodesic ray $r_w$ beginning at $\widetilde{\CS^\Upsilon(\gamma)}$ and ending at $w$ so that $q(r_w)$ is a geodesic ray in $\Upsilon\ssm T_{R_0}(\gammay)$, and therefore determines a point in $\partial \Upsilon$ that we denote $\Theta(w)$. 

\begin{proof}[Proof of Claim]
Consider a sequence of points $w_i \in \UG(\Upsilon, \gammay)$ tending to $w$.  There are geodesic segments $r_i$ determined by projecting from $w_i$ to $\widetilde{\CS^\Upsilon(\gamma)}$.  Each of these have the property that $r_i(t)$ is distance $t$ from $\widetilde{\CS^\Upsilon(\gammay)}$. These sub-converge to an infinite ray $r_w$ with this property representing $w$. Then $q(r_w)$ has the property  that $q(r_w(t))$ is distance $t$ from ${\CS^\Upsilon(\gammay)}$ and is a geodesic. 
\end{proof}

Note that $\Theta(w)\neq \gamma^{\pm}$. Indeed, if $q(r_w)$ had point at infinity $\gamma^\pm$ then $q(r_w)$ would intersect $T_{R_0}(\gammay)$ arbitrarily far from its starting point, since $\gammay$ is $\lambda_0$--quasi-convex by Lemma \ref{lem:gamma_0 qg} and $R_0$ is large compared to $\lambda_0$ and $\delta_2$ (see Notation \ref{not:first_constants}), the latter being a hyperbolicity constant for $\Upsilon$ by Lemma \ref{lem:app_of_CCH}. This is a contradiction.

The map $\Theta$ is easily seen to be continuous because, by taking $z\in\partial \UG(\Upsilon, \gammay)\ssm \{p\}$ sufficiently close to $w$, it can be arranged for the corresponding rays $q(r_w),q(r_z)$ to fellow-travel for an arbitrarily long time. Note that by the same argument any choice of $r_w$ gives the same $\Theta(w)$.

Let us prove that $\Theta$ is surjective. Consider $u\in\partial \Upsilon\ssm\{\gammay^{\pm}\}$, and choose a geodesic ray $\alpha_u$ that intersects $\CS^\Upsilon(\gammay)$ only at its starting point and has point at infinity $u$. Then $\alpha_u$ can be lifted to $\UG(\Upsilon,\gammay)$. We can in fact choose $\alpha_u$ that lifts to a geodesic ray in $\UG(\Upsilon,\gammay)$; this uses the argument from the proof of the claim applied to a sequence of points along the lift, which yields a ray in $\UG(\Upsilon,\gammay)$ projecting to a ray in $\Upsilon$ at finite distance, in our case, from the original choice of $\alpha_u$. For $v$ the point at infinity of the lift, we can take said lift as $r_v$, and therefore we get $\Theta(v)=u$.

 A similar argument shows that $\Theta$ is open (since in order to prove it one has to prove a surjectivity statement, meaning that the image of an open set contains a small neighborhood around each of its points). One can also see that, in view of Theorem \ref{t:Pi_isomorphism}.\eqref{item:long_loop}, which provides a lower bound on the translation distance of any non-trivial element of the deck group of $q$, different lifts of $\alpha_u$ yield different points in $\partial \UG(\Upsilon,\gammay)$. More precisely, those lifts are rays that lie at distance at least $\sys_0$ from each other in $\Upsilon\ssm T_{R_0}(\gammay)$, and it can be deduced from Lemma \ref{lem:horbadistort} that they lie at least $2(\log_2(\sys_0)-1)$ away from each other in $\UG(\Upsilon,\gammay)$. This is much larger than the hyperbolicity constant $\delta_2$ of $\UG(\Upsilon,\gammay)$ (see Proposition \ref{prop:hatY_modeled}), so that the lifts have distinct points at infinity.  Moreover, again because of the lower bound on the translation distance, the deck group acts properly discontinuously on $\partial \UG(\Upsilon,\gammay)\ssm \{p\}$ (roughly, as above, different orbit points project very far on the horosphere). Moreover, if $\Theta(w)=\Theta(w')$ then for any point on a ray $r_w$ there is an element of the deck group mapping it within distance bounded in terms of $\delta_0$ from a point on a ray $r_{w'}$. Picking the point on $r_w$ sufficiently far from the horosphere, said element of the deck group must map $w$ to $w'$; if not, we would have orbit points of $w$ accumulating to $w'$, contradicting proper discontinuity.

We now have that $\Theta$ factors through the quotient map 
\[\Psi\co \partial \UG(\Upsilon,\gammay)\ssm \{p\}\to \leftQ{\left(\partial \UG(\Upsilon,\gammay)\ssm \{p\}\right)}{\mathrm{Deck}(q)},\]
where $\mathrm{Deck}(q)$ is the group of the extensions to $\partial \UG(\Upsilon,\gammay)\ssm \{p\}$ of the elements of the deck transformation group of $q$, which acts properly discontinuously on $\partial \UG(\Upsilon,\gammay)\ssm \{p\}$. Moreover, writing $\Theta=\iota\circ\Psi$, we have that $\iota$ is bijective and continuous, and it is open since $\Theta$ is. Hence, $\iota$ is a homeomorphism.

Since $\Psi$ is a covering map with deck group isomorphic to $\mathbb Z$, so is $\Theta$.

To prove the ``moreover" statement we consider the following (not quite commutative) diagram:

\begin{equation}\label{eq:coarsecommute}
\begin{tikzcd}
\partial \UG(\Xz,\gammax)\smallsetminus\{p\} \arrow[r,"\phi"]\arrow[d,swap,"\Theta"] & \widetilde{\SCS}\arrow[d,"q"] \\
\partial \Xz\smallsetminus\{\gammax^{\pm\infty}\}\arrow[r,"\Pi_{R_0}"] & \SCS
\end{tikzcd}
\end{equation}
The maps $\Theta$, $q$ were defined above: the map $\Pi_{R_0}$ is from Definition \ref{def:proj}.  
The map $\phi$ is defined via the choice of rays $r_w$ above; namely $\phi(w) = r_w(0)$.
\begin{claim*}
  The diagram \eqref{eq:coarsecommute} commutes up to an error of $34\delta_0$.
\end{claim*}
\begin{proof}
The point $q(\phi(w))$ is the initial point of a geodesic ray (namely $q(r_w)$) in $\Xz$ which tends to the point $\Theta(w)$, and only meets $\SCS$ in one point.  The point $\Pi_{R_0}(\Theta(w))$ is the initial point of another such ray, $r'$.  Lemma~\ref{lem:constant_progress} implies that for all $t\ge 0$,  \[d_{\Xz}(r'(t),\gammax)\sim_{10\delta_0} {R_0} + t\]
and similarly for $r_w$.  Let $q\in \gammax$ be closest to $r_w(0)$ and let $q'\in \gammax$ be closest to $r'(0)$. Since $r_w$ and $r'$ eventually stay in $2\delta_0$--neighborhoods of each other we may choose $z\in r_w$ and $z'\in r'$ which are distance at least ${R_0} + 100\delta_0$ from $\gammax$ and so that $d_{\Xz}(z,z')\le 2\delta_0$.  We consider a geodesic hexagon with vertices (cyclically ordered) $\{q,r_w(0),z,z',r'(0),q'\}$.
Let $T = 15\delta_0$ and let $a = r_w(T)$.  Since the hexagon is $4\delta_0$--slim, there is a point $a'$ on one of the other five sides so that $d_{\Xz}(a,a')\le 4\delta_0$.

By Lemma~\ref{lem:constant_progress}, the distance from $a$ to $\gammax$ is at least ${R_0}+T-10\delta>{R_0}$, so $a'$ cannot be on any of the three sides contained in $T_{R_0}(\gammax)$.  It cannot be contained in the side $z,z'$, since this side is too far away.  Thus the point $a$ is equal to $r'(S)$ for some $S$.  Again using Lemma~\ref{lem:constant_progress}, we have
$S\le {R_0} + d_{\Xz}(a',\gammax)$.  But $d_{\Xz}(a',\gammax) \le d_{\Xz}(a,\gammax) + 4\delta_0 \le {R_0} + T + 4\delta_0$.  Putting these together we deduce $d_{\Xz}(r'(0),a) \le T + 4\delta_0$, and so $d_{\Xz}(r_w(0),r'(0))\le 2T + 4\delta_0 = 34\delta_0$. 
\end{proof}

The universal covering map of the annulus $\partial \Xz \ssm \{\gammax^\pm \}$ can be characterized as the only $\mathbb Z$--cover with the property that the preimage of any given point is contained in a single connected component.

Consider any $h\in\mathrm{Deck}(q)$ and any $x\in \partial  \UG(\Xz,\gammax)\ssm\{p\}$. We want to show that $h\cdot x$ lies in the same component as $x$. Consider $\phi(x)$ and $\phi(h\cdot x)=h\cdot \phi(x)$. Since $\mathrm{Deck}(q)$ is naturally isomorphic to $\pi_1^{D_1}(\SCS)$, and $\Pi_{R_0}$ induces an isomorphism from $\pi_1(\partial \Xz \ssm\{\gammax^\pm\})$ to $\pi_1^{D_1}(\SCS)$, there exists a loop $\alpha$ in $\partial \Xz \ssm\{\gammax^\pm\}$ at $\Theta(x)$ and a loop $\alpha'$ in $\SCS$ with the following properties. First, $\alpha'$ lifts to a path $\beta'$ with endpoints  $\phi(x)$ and $\phi(h\cdot x)=h\cdot \phi(x)$ in $\UG(\Upsilon,\gammay)$. Secondly, for each $t$ we have that $\Pi_{R_0}(\alpha(t))$ is $8\delta_0$--close to $\alpha'(t)$ (see Lemma \ref{lem:project_from_infty}). We have that $\alpha$ lifts to a path $\beta$ in $\partial  \UG(\Xz,\gammax)\ssm\{p\}$ starting at $x$, and ending at point in the fiber of $\Theta(x)$, say $h'\cdot x$. We claim that $h'=h$, and this will conclude the proof.

Indeed, by the claim, for each $t$ there exists an element $g_t$ of $\mathrm{Deck}(q)\cong \mathbb Z$ such that $\phi(\beta(t))$ is $42\delta_0$--close to $g_t\beta'(t)$.  We have $g_0=id$ since $\beta(0)=x$ and $\beta'(t)=\phi(x)$. 

By the same argument as in Lemma \ref{lem:project_from_infty}, each point in $\partial  \UG(\Xz,\gammax)\ssm\{p\}$ has a neighborhood that has $\phi$--image of diameter at most $8 \delta_2$. Since any non-trivial element acts with translation distance larger than $84\delta_0+8\delta_2$ on $\widetilde{\SCS}$ (by the choice of $\sys_0$ in Notation \ref{not:first_constants}), for any $g$ in the deck group the set of $t$ such that $g_t=g$ is open. Therefore, for all $t$ we have $g_t=id$, and in particular for $t=1$ we have that $\phi(h'(x))$ is $42\delta_0$--close to $\phi(hx)$. Similarly to the previous argument, this implies that $h'=h$, as required.
\end{proof}

\begin{corollary}
\label{cor:UG0_S2}
$\partial \UG(\Xz, \gammax)$ is homeomorphic to $S^2$.
\end{corollary}

\begin{proof}
Let $p$ be the point at infinity of the horoball. 
In view of Proposition~\ref{prop:deck_Z} we have that $\partial \UG(\Xz, \gammax)\ssm \{p\}$ is homeomorphic to $\mathbb R^2$, as it is the universal cover of an open annulus.

Since $\partial \UG(\Xz, \gammax)$ is a compact metric space, and removing a point makes it homeomorphic to $\mathbb R^2$, it must be homeomorphic to $S^2$.
\end{proof}

 We say that a combinatorial graph $\Gamma$ has \emph{uniform polynomial growth} if there exist constants $c<C$ and an integer $d$ so that, for any ball $B$ of radius $r$, we have
\[ c r^d \le \# B\cap \Gamma^{(0)} \le C r^d. \] Note that a combinatorial graph with a co-compact group action and polynomial growth has uniform polynomial growth. 

The following lemma is implicit in the work of Dahmani and Yaman \cite{DahmaniYaman}.  
\begin{lemma}\label{lem:dahmaniyaman}
Let $\Gamma$ be a combinatorial graph of uniform polynomial growth. 
There is a bounded valence graph $\horba_{bv}(\Gamma)$ containing $\Gamma$ so that the identity map on $\Gamma$ extends to a quasi-isometry from $\horba_{bv}(\Gamma)$ to the combinatorial horoball based on $\Gamma$.
\end{lemma}
\begin{proof}
Since $\Gamma$ is of uniform polynomial growth, there are constants $c<C$ and an integer $d$ so that, for any ball $B$ of radius $r$, we have
 $c r^d \le \# B\cap \Gamma^{(0)} \le C r^d.$
For each $n$, choose a subset $V_n$ of  $\Gamma^{(0)}$ satisfying
\begin{enumerate}
    \item\label{itm:cover} For any $x\in \Gamma^{(0)}$, there is some $v\in V_n$ so that $d_\Gamma(x,v) < 2^{n-1}$.
    \item\label{itm:disjoint} For any $v,w\in V_n$, we have $d_\Gamma(v,w) \ge 2^{n-1}$ (so the radius $2^{n-2}$ balls around $v,w$ are disjoint).
\end{enumerate}
Now define a graph $\horba_{bv}=\horba_{bv}(\Gamma)$ with $V(\horba_{bv}) = \bigsqcup_n V_n$ and edges defined as follows.
\begin{itemize}
    \item (horizontal edges)  Connect $v,w\in V_n$ with an edge if $d_\Gamma(v,w)<2^{n+1}$.
    \item (vertical edges)  Connect $v\in V_n$ to $w\in V_{n+1}$ if $d_\Gamma(v,w)<2^{n+1}$.
\end{itemize}
Notice that $V_0 = \Gamma^{(0)}$ and the sub-graph induced by $V_0$ is canonically isomorphic to $\Gamma$.

Let $\horba=\horba(\Gamma)$ be the combinatorial horoball based on $\Gamma$. (Recall Definition~\ref{def:combhoroball} for the definition of a combinatorial horoball.)  We will define coarsely Lipschitz maps 
\[ \phi : \horba_{bv}^{(0)}\to \horba^{(0)}\quad\mbox{and}\quad \psi : \horba^{(0)}\to \horba_{bv}^{(0)}\]
which are quasi-inverses of one another, which implies they are quasi-isometries.

For $v\in V_n$, we define $\phi(v) = (v,n)$.  Suppose that $v,w\in V_n$ are connected by a horizontal edge.  Then $d_{\Gamma}(v,w)< 2^{n+1}$, so there is a horizontal path in $\horba$ of length at most $2$ connecting $\phi(v) = (v,n)$ to $\phi(w) = (w,n)$.  If $v\in V_n,w\in V_{n+1}$ are connected by a vertical edge, then there is a path of length at most $2$ connecting $\phi(v) = (v,n)$ to $\phi(w) = (w,n+1)$.  Namely, connect $(v,n)$ to $(v,n+1)$ by a vertical edge, and then $(v,n+1)$ to $(w,n+1)$ by a horizontal edge.  Summarizing, we have shown that $\phi$ is $2$--Lipschitz.

For $(v,n)\in \horba^{(0)}$, we choose a point $v'\in V_n$ minimizing $d_\Gamma(v,v')$, and define $\psi(v,n)$ to be $v'$.  We now show $\psi$ is coarsely Lipschitz.  Suppose first that $(v,n)$ is connected to $(w,n)$ by a horizontal edge, and that $v',w'$ are the chosen points of $V_n$ closest to $v,w$ respectively.  Then $d_\Gamma(v',w') <2^{n+1} + 2^n = 6\cdot 2^{n-1}$.  There is therefore a sequence of vertices $v' = x_0, x_1,\ldots, x_6 = w'$ in $\Gamma^{(0)}$ so that $d_\Gamma(x_i,x_{i+1})\le 2^{n-1}$.  For each $x_i$ there is a vertex $v_i\in V_n$ so that $d_\Gamma(v_i,x_i)< 2^{n-1}$, by \eqref{itm:cover}.  We thus have, for each $i$,  $d_\Gamma(v_i,v_{i+1})< 3\cdot 2^{n-1} < 2^{n+1}$, so $v_i$ is connected by a horizontal edge to $v_{i+1}$.  In particular, we have $d_{\horba_{bv}}(v',w')\le 6$, so $\psi$ sends vertices connected by a horizontal edge to vertices which are at most $6$ apart.

Now consider a vertical edge, joining $(v,n)$ to $(v,n+1)$ in $\horba$.  Let $v'\in V_n$ and $v''\in V_{n+1}$ be the images under $\psi$ of $(v,n)$ and $(v,n+1)$, respectively.  We have $d_\Gamma(v,v')<2^{n-1}$ and $d_\Gamma(v,v'')<2^n$.
The vertex $v'$ is connected by a vertical edge in $\horba_{bv}$ to some $w\in V_{n+1}$ so that $d_\Gamma(v',w)< 2^n$.  
We thus have $d_\Gamma(w,v'')<2^{n+2}$, so $v',w$ are connected by a horizontal edge.  In particular $d_{\horba_{bv}}(\psi(v,n),\psi(v,n+1))\le 2$.  Together with the bound in the last paragraph this shows that $\psi$ is $6$--Lipschitz.

The composition $\psi\circ\phi$ is the identity.  The composition $\phi\circ\psi$ maps a vertex $(v,n)$ to some $(v',n)$ so that $d_\Gamma(v,v')< 2^{n-1}$.  In particular, $(v,n)$ and $(v',n)$ are connected by a horizontal edge in $\horba$ and so $d_\Gamma(\phi\circ\psi(v,n),(v,n))\le 1$. 

It remains to show that $\horba_{bv}$ has bounded valence.  Let $v\in V_n$ be some vertex of $\horba_{bv}$, and let $w_1,\ldots,w_k\in V_n$ be vertices connected by an edge to $v$.  These points $w_i$ all lie in the ball of radius $2^{n+1}$ around $v$, and the $2^{n-2}$--balls around the $w_i$ are disjoint, by~\eqref{itm:disjoint}.  Each of these balls has volume at least $c(2^{n-2})^d$, and they are all contained in the $2^{n+2}$--ball around $v$, which has volume at most $C(2^{n+2})^d$, so we have
\[ k c(2^{n-2})^d \le C(2^{n+2})^d, \]
implying that $k \le 16^d\frac{C}{c}$.  

If $u_1,\ldots, u_l\in V_{n-1}$ are connected by vertical edges to $v$, their $2^{n-3}$--balls in $\Gamma$  are disjoint and contained in the $2^{n+2}$--ball around $v$, so we similarly argue that $l\le 32^d\frac{C}{c}$.

Finally, if $z_1,\ldots,z_m\in V_{n+1}$ are connected by vertical edges to $v$, their (disjoint) $2^{n-1}$--balls are contained in the $2^{n+3}$--ball around $v$, so we have $m\le 16^d\frac{C}{c}$.  Put together we get a bound on the valence of $v$ which is independent of $v$.
\end{proof}

\begin{lemma} \label{lem:bound_geom_cusp}
Suppose $\Xz$ has bounded valence.  Then the space $\UG(\Xz,\gammax)$ is quasi-isometric to a bounded valence graph. More precisely, there exists a bounded valence graph $\horba_{bv}$ which is quasi-isometric to a horoball in $\mathbb{H}^3$ and contains a copy of $\widetilde{\SCS}$ so that if $\mc{B}$ is the space obtained by gluing $\horba_{bv}$ to $\widetilde{\SCTC}$ along $\widetilde{\SCS}$ then there is a quasi-isometry
between $\UG(\Xz,\gammax)$ and $\mc{B}$ extending the identity map on $\widetilde{\SCTC}$.
\end{lemma}
\begin{proof}
By Lemma~\ref{lem:pi_CS/g}, $\widetilde{\SCS}$ has a free and cocompact action by a group quasi-isometric to $\mathbb{R}^2$.  It follows from this that the combinatorial horoball $\horba$ based on $\widetilde{\SCS}$ is quasi-isometric to a horoball in $\mathbb{H}^3$ (see, for example \cite[Corollary 9.2]{HealyHruska}).

It also follows that $\widetilde{\SCS}$ has uniform polynomial (quadratic) growth.  Let $\horba_{bv}$ be the graph $\horba_{bv}(\widetilde{\SCS})$ from the conclusion of Lemma~\ref{lem:dahmaniyaman}.  Let $\mc{B}$ be the space obtained from gluing $\widetilde{\SCTC}$ to $\horba_{bv}$ along the natural copies of $\widetilde{\SCS}$.  Since we have assumed $\Xz$ has bounded valence, and $\horba_{bv}$ does by Lemma~\ref{lem:dahmaniyaman}, the space $\mc{B}$ has bounded valence.  The horoballs $\horba$ and $\horba_{bv}$ are quasi-isometric, so the identity on $\widetilde{\SCTC}$ extends to coarsely inverse coarsely Lipschitz maps $\mc{B}\to \UG(\Xz,\gammax)$ and $\UG(\Xz,\gammax)\to \mc{B}$. This implies that the spaces are quasi-isometric.
\end{proof}

\subsection{Uniform linear connectivity}

Up to this point, we have not required there to be a cocompact group of isometries of the space $\Xz$.  However, the following result does require this assumption.  Thus we make the following extra assumption.

\begin{assumption} \label{ass:geom}
    Suppose that $\Xz$ satisfies Assumption~\ref{ass:X2}, and furthermore assume that $\Xz$ admits a group of isometries containing $g$ which is geometric.
\end{assumption}

Notice that under Assumption~\ref{ass:geom}, $\Xz$ has bounded valence (rather than just being locally finite), so the assumptions of Lemma~\ref{lem:bound_geom_cusp} hold.

In the proof of the following lemma we use a limiting argument, so the resulting constant is not explicit and depends on $\Xz,\gammax$, and ${R_0}$.

\begin{lemma}\label{lem:unwrap_is_lin_conn}
There exists $L_1=L_1(\Xz,\gammax,{R_0})$ so that any $\delta_2$--adapted visual metric on $\partial \UG(\Xz,\gammax)$ with any basepoint
is $L_1$--linearly connected.
\end{lemma}

\begin{proof}
Let $\mc{B}$ be the bounded valence graph from Lemma~\ref{lem:bound_geom_cusp}.  This space is quasi-isometric to $\UG(\Xz,\gammax)$ and is $\delta'$--hyperbolic for some $\delta'$.  Thus to prove the current lemma it suffices to show that the elements of any sequence of $\delta'$--adapted visual metrics on $\partial \mc{B}$ are uniformly linearly connected.  

We therefore assume $(y_i)$ is a sequence of points in $\mc{B}$.  Using bounded valence we may pass to a subsequence such that the pointed spaces $(\mc{B},y_i)$ strongly converge to a pointed space $(\mc{B}_\infty,y_\infty)$.  By Proposition~\ref{prop:GMS3.5}, the boundaries $\partial \mc{B}$ with adapted visual metrics at $y_i$ weakly Gromov-Hausdorff converge to the boundary $\partial \mc{B}_\infty$ with adapted visual metric at $y_\infty$.

We claim that given any such sequence $y_i$ and limit $(\mc{B}_\infty,y_\infty)$, the space $\partial \mc{B}_\infty$ does not have a weak cut point.     

Given this claim, Theorem \ref{thm:lin_conn_cut_point} implies that the $\delta'$--adapted visual metrics on $(\mc{B},y_i)$ must have been uniformly linearly connected, establishing the lemma.

It remains to prove the claim.  The group $E$ from Proposition~\ref{prop:peripheral_is_planar} acts on $\UG(\Xz,\gammax)$ and the quasi-isometry from Lemma~\ref{lem:bound_geom_cusp} therefore induces a quasi-action of $E$ on $\mc{B}$ (by uniform quality quasi-isometries).

 The possible limits depend on the position of $y_i$ compared to the distinguished horoball in $\mc{B}$. If the distance from the horoball diverges, then $\mc{B}_\infty$ is a Gromov-Hausdorff limit of $\Xz$, and hence its boundary is $S^2$ (by Assumption~\ref{ass:geom} there is a geometric action on $\Xz$, so all Gromov-Hausdorff limits of $\Xz$ are quasi-isometric to $\Xz$, which by assumption has boundary $S^2$). If the $y_i$ are deeper and deeper in the horoball, then the limit is a limit of $\horba_{bv}$, which is  quasi-isometric to a horoball in $\mathbb H^3$, by Lemma \ref{lem:bound_geom_cusp}. Hence, once again, the boundary of $\mc{B}_\infty$ is $S^2$ since $\mc{B}_\infty$ is quasi-isometric to $\mathbb H^3$. Finally, if the $y_i$ stay within bounded distance from $\widetilde{\SCS}$, then we can move them to a bounded region of $\mc{B}$ using the  quasi-action on $\mc{B}$, which is cobounded on $\widetilde{\SCS}$.  Therefore $\mc{B}_\infty$ is quasi-isometric to $\mc{B}$, and hence to $\UG(\Xz,\gammax)$. Therefore, the boundary of $\mc{B}_\infty$ is $S^2$ in this case too, by Corollary~\ref{cor:UG0_S2}.  Thus in every case $\partial \mc{B}_\infty$ is $S^2$, which does not have a weak cut point, establishing the claim, and hence the lemma.
\end{proof}

The following is an immediate consequence of Lemmas \ref{lem:lin_conn_implies_sphere_conn} and \ref{lem:unwrap_is_lin_conn}.

\begin{corollary}
There exists $\Delta = \Delta(\Xz,\gammax,{R_0})$ so that for every $y \in \UG(\Xz,\gammax)$ and every $\Sigma \ge 0$ the space $\UG(\Xz,\gammax)$ is $(\Delta,\Sigma)$--spherically connected at $y$.
\end{corollary}

By Lemma \ref{lem:sphere_conn_implies_lin_conn}, we therefore have the following.

\begin{corollary} \label{cor:modeled_on_hatX_LC}
There exist $\Sigma_0$ and $L$ so that for any $\delta_2$--hyperbolic space $Y$ which is $\Sigma_0$--modeled on $\left\{ \UG(\Xz,\gammax) \right\}$ and any $y \in Y$, any $\delta_2$--adapted visual metric on $\partial Y$ with basepoint $y$ is $L$--linearly connected.
\end{corollary}

\section{Unwrapping a family of axes} \label{s:unwrap family}
In the previous section we unwrapped the complement of a single tube. In this section we iterate this process, starting from a suitable family of axes in a hyperbolic space on which a group $G$ acts. We construct a sequence of ``partially" unwrapped spaces, and we show that they converge to a space where another group $\widehat G$ acts. This $\widehat G$ will be the drilled group from the main Theorem~\ref{t:Drill}, and the main goal of this section is to show that it is relatively hyperbolic, and a drilling of $G$ along $g$.  

\subsection{Assumptions for this section}
 Fix a group $G$ acting freely and cocompactly on a $\delta_0$--hyperbolic graph $\Xz$ with boundary homeomorphic to $S^2$, and up to increasing $\delta_0$ we can and do assume that $X_0$ is $\delta_0$--visible, by \cite[Lemma 3.1]{BesMes}. We fix a loxodromic element $g$ and a continuous quasi-geodesic $\gammax$ stabilized by $g$. Fix also a linear connectedness constant $L_0$ and a doubling constant $\doub$ as in Assumption \ref{ass:X}, which exist by \cite[Proposition 4]{BonkKleiner05} and \cite[Theorem 9.2]{BonkSchramm} respectively. Fix $\lambda_0$ as in Assumptions \ref{ass:X} (for $\Yx=\gammax$), and notice that at this point we have that Assumptions \ref{ass:X} and \ref{ass:X2} hold for $X_0$ and $\Yx=\gammax$, and the associated constants.  Further, let $s_0, D_1, \delta_1, \delta_2, \sigma_0, \sys_0$ and $R_0$ be chosen as in Notation \ref{not:first_constants}.

We now make careful choices of some further constants: 
\begin{enumerate}
 \item Fix $\Sigma_0$ and $L$ as in Corollary \ref{cor:modeled_on_hatX_LC}.  Without loss of generality, we also assume that $\Sigma_0 \ge 10^9\delta_1$.
 \item Fix $\Sigma_1 = 20\Sigma_0 + 60\sigma_0 + 10{R_0} + 50\delta_2$.
  \item Fix $\Sigma = \Sigma_1 + 2{R_0}$.  This is the constant $\Sigma$ in our main theorem, Theorem \ref{t:Drill}.
\end{enumerate}

We furthermore assume that the collection of translates $\{ h \cdot \gammax \}_{h \in {G\smallsetminus \langle g \rangle}}$ forms a $\Sigma$--separated family (that is, each two distinct elements are at least $\Sigma$ apart), so the tubes $\left\{ h \cdot T_{R_0}(\gammax) \right\}_{h \in {G \smallsetminus \langle g \rangle}}$ form a $\Sigma_1$--separated family.

\subsection{Construction of the intermediate spaces} \label{ss:intermediate}

We now describe a sequence of pointed spaces obtained by repeatedly applying the ``unwrap and glue" Construction~\ref{unwrap_one_tube}.  The limit of these spaces will be the cusped space of the group obtained by drilling $G$ along $g$.

Let $x_0$ be a base vertex of $\Xz$ which lies further than $R_0$ from any $h\cdot\gammax$ with $h\in G$.  We also regard $x_0$ as a basepoint for $\wchXz= \CTC^{\Xz}_{R_0}\left(\bigcup_{h\in G} h \cdot \gammax\right)$ (see Definition \ref{def:CTC} and Convention \ref{conv:CTC}), that is, the completed complement of all the tubes around the translates $h \cdot \gammax$.

The collection of edge-paths in $\wchXz$ beginning at $x_0$ and ending on some shell is countable.  We choose an enumeration $\{p_i\}_{i\ge 1}$ of these paths.  For each $i$ we fix some $h_i$ so that $p_i$ ends on the shell $\CS^{\Xz}(h_i \cdot \gammax)$.  It will happen for some (in fact many) $i \ne j$ that $h_i=h_j$. 

Define the pair $(\widehat{Y}_1,y_1) = (\UG(\Xz,\gammax),x_1)$, where $x_1$ is any preimage of the basepoint $x_0 \in \Xz$.  By relabeling if necessary, we suppose that the path $p_1$ lifts to a path in $\widehat{Y}_1$ which starts at $y_1$ and finishes on the (single) horoball in $\widehat{Y}_1$. Proposition~\ref{prop:hatY_separated} furnishes $\widehat{Y}_1$ with a $\Sigma$--separated collection of tubes and horoballs, $\mathcal{A}_1$ say (cf. Definition~\ref{def:separated_tubes}).
Define $q_1 \co \mathrm{HO}(\widehat{Y}_1,\mathcal{A}_1) \to \wchXz$ be the restriction of the covering map $\widetilde{\mathcal{C}} \to \mathcal{C}$ from the construction of $\UG(\Xz,\gammax)$ to $\mathrm{HO}(\widehat{Y}_1,\mathcal{A}_1)$.
(The ``hollowed out'' space $\mathrm{HO}(\cdot,\cdot)$ was defined in Construction~\ref{con:hollow}.)

Let $j \ge 1$. Suppose, by induction, that for all $l$ such that $1 \le l\le j$ we have constructed a quadruple $\left( \widehat{Y}_l, \mc{A}_l, y_l, q_l \right)$ where: $\widehat{Y}_l$ is a space, $\mc{A}_l$ is a collection of tubes and horoballs in $\widehat{Y}_l$, $y_l \in \widehat{Y}_l$ is a basepoint and \[q_l \co \mathrm{HO}(\widehat{Y}_l,\mc{A}_l) \to \mathrm{HO}(\widehat{Y}_{l-1},\mc{A}_{l-1})\] is a $D_1$--covering map (see Definition~\ref{def:coarse_fund}) such that $q_l(y_l) = y_{l-1}$. Let $\wchY_l = \mathrm{HO}(\widehat{Y}_l,\mc{A}_l)$. (Note that composing the $D_1$--covers gives a $D_1$--cover $q_1\circ\cdots \circ q_l: \wchY_l\to\wchXz$.)

Furthermore, we suppose that the following properties hold, for each $1 \le l \le j$:
\begin{enumerate}
\item\label{item:simplyconnected} $\widehat{Y_l}$ is connected and $D_1$--simply-connected;
\item\label{item:Cap_sigmamod} $\widehat{Y_l}$ is $\Sigma_0$--modeled on $\UG(\Xz,\gammax)$;
\item \label{item:sigmamod} 
$\widehat{Y_l}$ is $\sigma_0$--modeled 
on $\left\{ \Xz^{\mathrm{cusp}}, \mathring{\horba} \right\}$  (see Notation~\ref{not:H} and Definition~\ref{def:zcusp} for these spaces);
\item\label{item:Y_lhyp} $\widehat{Y_l}$ is $\delta_2$--hyperbolic;
\item\label{item:tubes} $\mc{A}_l$ is a $\Sigma_1$--separated collection of tubes and horoballs, as in Definition \ref{def:separated_tubes};
\item\label{item:paths}   Let $(p_i^l)_{i\ge 1}$ be the sequence of paths starting at $y_l$ which are lifts of the $p_i$ (with $p_i^l$ being the lift of $p_i$). 
For each $k \ge 1$, the path $p_k^l$ finishes on a shell or a horosphere associated to some element of $\mc{A}_l$.
For each $k \le l$ the path $p_k^l$ finishes on a horosphere (not a shell).
\end{enumerate}
Note that the third item follows immediately from the second, the fact that $\Sigma_0 > \sigma_0$ and Proposition~\ref{prop:hatY_modeled} (observe that $\Xz$ is $\sigma_0$--modeled on $\left\{ \Xz^{\mathrm{cusp}}, \mathring{\horba} \right\}$).

For the base case of the induction, let us prove that the properties hold for $j=1$.

The property (\ref{item:simplyconnected}) is Proposition \ref{prop:hatY_modeled}.  A space is modeled on itself, so (\ref{item:Cap_sigmamod}) holds.  We have already noted that the property (\ref{item:sigmamod}) follows from (\ref{item:Cap_sigmamod}).
The property~(\ref{item:Y_lhyp}) holds by Proposition~\ref{prop:hatY_modeled} and Lemma~\ref{lem:app_of_CCH}.  We have already noted that Proposition~\ref{prop:hatY_separated} implies (\ref{item:tubes}).  Finally, the choice of $p_1$ above implies the property~(\ref{item:paths}).

Given these assumptions, we define a quadruple $\left(\widehat{Y}_{j+1}, \mc{A}_{j+1}, y_{j+1}, q_{j+1}\right)$ indutively so that all of these assumptions are satisfied.

If $p_{j+1}^j$ ends on a horosphere, then $\widehat{Y}_{j+1} = \widehat{Y}_j$, and all of the other data remains the same.

We now suppose that $p_{j+1}^j$ ends on the (completed) shell $S_j$ of a tube $C_j$ from $\mc{A}_j$, with core $c_j$.  We let $\widehat{Y}_{j+1}=\UG(\widehat{Y}_j,c_j)$ (see Construction \ref{unwrap_one_tube}).  That is, we do the following. We remove the tube $C_j$, we complete the space we obtain, we take the $D_1$--universal cover, and finally we glue in a copy of $\horba$.

Let $y_{j+1}$ be any preimage of $y_j$.  
Let $\mc{A}_{j+1}$ be the collection of all of the lifts of the elements of $\mc{A}_j$, along with the new horoball. 
It follows from Proposition \ref{prop:hatY_modeled} (and the assumptions on $\widehat{Y}_j$) that $\widehat{Y}_{j+1}$ is connected and $D_1$--simply-connected, so the property~(\ref{item:simplyconnected}) holds.  Let $B$ be a ball of radius $\Sigma_0$ in $\widehat{Y}_{j+1}$.  If $B$ intersects the new horoball, then it is contained in the $2\Sigma_0$--neighborhood of this horoball, and so by Lemma~\ref{lem:sheath} it is isometric to a ball of radius $\Sigma_0$ in $\UG(\Xz,\gammax)$ (this uses the choice of $\Sigma$,  Definition~\ref{def:separated_tubes}, and Remark~\ref{rem:separatedisomorphmetric}).  On the other hand, if $B$ does not intersect the new horoball then, by Lemma~\ref{lem:injectiveballs}, $B$ injects into $\widehat{Y}_j$, and so is isometric to a ball in $\UG(\Xz,\gammax)$ by induction.  Hence the property~(\ref{item:Cap_sigmamod}) holds, and so (\ref{item:sigmamod}) does also.  It follows from Lemma \ref{lem:app_of_CCH} that $\widehat{Y}_{j+1}$ is $\delta_2$--hyperbolic so (\ref{item:Y_lhyp}) holds.  It follows from Proposition \ref{prop:hatY_separated} that $\mc{A}_{j+1}$ forms a $\Sigma_1$--separated collection of tubes and horoballs, so (\ref{item:tubes}) holds.  The property~(\ref{item:paths}) holds by the construction. 

We define $\wchY_{j+1}$ to be $\mathrm{HO}(\widehat{Y}_{j+1},\mc{A}_{j+1})$ and construct the map $q_{j+1} \co \wchY_{j+1} \to \wchY_j$, which will be an infinite cyclic cover. 

We claim that we have the following diagram where the vertical arrows are covers. The middle column is the same as the right column except that the shells around tubes have not been completed. 

\begin{center}
\begin{tikzcd}
\widetilde{\CTC ^{\widehat{Y}_j}(c_j)}\arrow[d] & \widehat{Y}_{j+1}\smallsetminus \bigcup\mc{A}_{j+1}\arrow[hook',l]\arrow[hook,r]\arrow[d] & \wchY_{j+1}\arrow[d] \\
\CTC ^{\widehat{Y}_j}(c_j) & \widehat{Y}_j\smallsetminus\bigcup \mc{A}_j\arrow[hook',l]\arrow[hook,r] & 
\wchY_j
\end{tikzcd}
\end{center}

The left ($D_1$--universal) cover indeed restricts to the central cover because of the definition of $\mc{A}_{j+1}$. Since tubes are $\Sigma_1$--separated and $\Sigma_1$ is much larger than $s_0$, and lifts of tubes are tubes by Proposition \ref{prop:hatY_separated}, the central cover extends to a cover on the right. Let $q_{j+1}$ be the cover on the right.  By construction, this is a $D_1$--covering.

We can now consider all lifts $p^{j+1}_i$ of the $p_i$ via $q_{j+1}$ starting at the basepoint $y_{j+1}$.  We think of these lifts as obtained by first lifting to $\wchY_j$ and then to $\wchY_{j+1}$. It is clear that all $p^{j+1}_i$ finish on a shell or horosphere associated to an element of $\mc{A}_{j+1}$.  Moreover, by construction we see that $p_{j+1}^{j+1}$ finishes on a horosphere (since the full preimage of the shell we unwrap is a horosphere).  For $1 \le k \le j$ the fact that $p_k^j$ finishes on a horosphere implies that $p_k^{j+1}$ also finishes on a horosphere.

This finishes the inductive construction of $(\widehat{Y}_{j+1},\mc{A}_{j+1}, y_{j+1},q_{j+1})$.

Since the spaces $\wchY_j$ form a tower of $D_1$--covers, the following result is immediate.
\begin{lemma} \label{lem:wchY}
    There is an inverse limit $\wchY$ of the tower of covers $\wchY_j$, along with a $D_1$--covering map $q:\wchY\to \wchXz$, and a basepoint $\widecheck{y}$, as well as $D_1$--covering maps $r_j: \wchY\to \wchY_j$, such that $r_j(\widecheck{y})=y_j$.
\end{lemma}

\begin{lemma}
\label{lem:pre-images}
Any component of the preimage of $\CS^{\Xz}(h_i\gammax)$ under $q: \wchY \to \wchXz$  is isometric to the $D_1$--universal cover $\widetilde{\CS^{\Xz}(h_i\gammax)}$.
\end{lemma}

\begin{proof}
Consider any preimage $C$ as in the statement. An edge-path from the basepoint $\widecheck{y}$ to $C$ is a lift of one of the paths $p_j$. By construction, the lift of $p_j$ to $\wchY_j$ starting at $y_j$ ends on a horosphere. 
Therefore, the same holds in the further cover $\wchY$.
\end{proof}

\begin{definition} \label{def:hatY}
Denote by $\widehat{Y}$ the space obtained from $\wchY$ by gluing in copies of $\horba$ to each component in the pre-image of a completed shell as in Lemma \ref{lem:pre-images}.    
\end{definition}

We now want to show that the $\widehat{Y}_j$ strongly converge to $\widehat{Y}$; in order to do so we will use Lemma \ref{lem:R/3} (the lemma which, roughly, allows us to promote local homeomorphisms to local isometries).

\begin{lemma}
\label{lem:Y_j-converge}
The spaces $(\widehat{Y}_j,y_j)$ strongly converge to $(\widehat{Y},\widecheck{y})$. 
\end{lemma}

\begin{proof}
Consider the ball $B_k$ of radius $k$ in $\widehat{Y}$ around $\widecheck{y}$. By Lemma \ref{lem:R/3} it suffices to construct an injective simplicial map from a sub-graph of $\widehat{Y}$ containing $B_k$ to $\widehat{Y}_j$ for any sufficiently large $j$, which is a local homeomorphism on a ball of slightly smaller radius.  Note that $B_k$ intersects finitely many horoballs, and we denote by $C_k$ the union of $B_k$ and said horoballs, and by $C'_k$ the intersection of $C_k$ with $\wchY$. We claim that the cover $r_j: \wchY\to \wchY_j$ restricted to $C'_k$ is injective for all sufficiently large $j$. Indeed, because of convergence of the covers $\wchY_i \to \wchY$, the restriction to $B_k\cap \wchY$ of $r_j$ is injective for all sufficiently large $j$. Moreover, as in the proof of Lemma \ref{lem:pre-images} each of the horospheres of $C'_k$ maps isometrically to $\wchY_j$ for any sufficiently large $j$.

For any $j$ such that $r_j$ restricted to $C'_k$ is injective, we claim that $C_k$ embeds in $\widehat{Y}_j$. Indeed, $C'_k$ does not intersect any completion of shells in $\wchY_j$ (since $C'_k$ contains all the horospheres it intersects), so that $C'_k$ embeds in $\widehat{Y}_j$. Since this embedding of $C'_k$ maps horospheres to horospheres, the embedding extends across the horoballs, giving the required embedding $C_k\to \widehat{Y}_j$. To see that the embedding is a local homeomorphism, notice that on $C'_k$ it is a composition of the restriction of a covering map and a homeomorphism, and the extension to the horoballs preserves being a local homeomorphism.
\end{proof}

\begin{corollary} \label{cor:Yhat_hyp}
The space $\widehat{Y}$ is $D_1$--simply connected and $\delta_2$--hyperbolic and $\delta_2$--visible.
\end{corollary}

\begin{proof}
Note that $\widehat{Y}$ is clearly connected and $\sigma_0$--modeled on the $\left\{\widehat{Y}_i \right\}$, so by the construction and properties of the $\widehat{Y}_i$, the space $\widehat{Y}$ is $\sigma_0$--modeled on $\left\{ \Xz^{\mathrm{cusp}}, \mathring{\horba} \right\}$.  Thus by Lemma~\ref{lem:app_of_CCH}, to prove that $\widehat{Y}$ is $\delta_2$--hyperbolic and $\delta_2$--visible it suffices to prove that 
$\widehat{Y}$ is $D_1$--simply-connected.
Therefore, suppose that $\sigma_0$ is a loop in $\widehat{Y}$.  Then $\sigma_0$ lies in some ball $B_k(y)$.  By strong convergence we can interpret $\sigma_0$ as a loop in $\widehat{Y}_i$ for all sufficiently large $i$.  These loops can be filled by a coarse disk which lies entirely in some $B_{k'}(y_i)$, where $k'$ depends on the length of $\sigma_0$, but not on $i$, since all of the $\widehat{Y}_i$ are $\delta_2$--hyperbolic.  For sufficiently large $i$, the ball $B_{k'}(y_i)$ is isometric to $B_{k'}(y)$, which means that this filling can be interpreted as a filling in $\widehat{Y}$.  This proves that $\widehat{Y}$ is $D_1$--simply-connected, and completes the proof.
\end{proof}

Recall from Definition~\ref{def:coarse_fund} that for $\Gamma$ any graph, and $D>0$, we write $\Gamma^D$ for the $2$--complex obtained by attaching disks to all loops of length at most $D$. 

\begin{lemma}
The space $\wchY^{D_1}$ is simply-connected. In particular, $\wchY^{D_1}$ is the universal cover of $\wchXz^{D_1}$.
\end{lemma}

\begin{proof}
This follows quickly from the fact that horospheres are $D_1$--simply-connected.  Take a loop in $\wchY$, fill it in $\widehat{Y}$, and then replace those parts of the fillings which lie in horoballs with fillings on the horospheres.
\end{proof}

Note that since the construction of $\wchXz$ from $\Xz$ is $G$--equivariant, $G$ acts freely and cocompactly on $\wchXz$, hence on $\wchXz^{D_1}$.  Let $Z_{D_1} = \leftQ{\wchXz^{D_1}}{G}$.  We now have a sequence of regular covering maps:
\[ \wchY^{D_1} \to \wchXz^{D_1} \to Z_{D_1}, \]
where the first map extends $q$ and the second map has deck group $G$.

\begin{definition}
Let $\widehat{G}$ be the deck group of the cover $\wchY^{D_1} \to Z_{D_1}$. That is, let $\widehat{G}$ be the fundamental group of $Z_{D_1}$.
\end{definition}

The group $\widehat{G}$ acts freely and cocompactly on the space $\wchY^{D_1}$, since it is a covering action with (compact) quotient $Z_{D_1}$.  Thus $\widehat{G}$ also acts freely and cocompactly on $\wchY$.  Moreover, this action preserves the family of horospheres in $\wchY$, which means that $\widehat{G}$ also acts freely by isometries on the $\delta_2$--hyperbolic space $\widehat{Y}$.  The quotient $\widehat{Z} := \leftQ{\widehat{Y}}{\widehat{G}}$ is clearly quasi-isometric to a ray.  

\begin{definition}
Choose a horoball  $\horba$ in $\widehat{Y}$ and let $P = \Stab_{\widehat{G}}(\horba)$.
\end{definition}

\begin{theorem} \label{GPrelhyp}
The pair $\left(\widehat{G},P\right)$ is relatively hyperbolic, and $\partial(\widehat{G},P) = \partial\widehat{Y}$.
\end{theorem}
\begin{proof}
We claim that the action of $\widehat{G}$ on $\widehat{Y}$ satisfies Gromov's definition of relative hyperbolicity (see Definition~\ref{def:gromov RH}). We know that $\widehat{Y}$ is hyperbolic and the quotient by the group action is quasi-isometric to a ray, which is in fact the image of a vertical ray $r$ in a copy of $\horba$ in $\widehat{Y}$ (which we identify with $\horba$ below). What we are left to show is that $\horba$ is coarsely the sub-level set of a horofunction $h$ based at the limit point of $r$. Notice that all vertical rays in $\horba$ are geodesic rays in $\widehat{Y}$ with the same point at infinity. Therefore the value of $h$ along any such ray $r'$ coincides up to an additive constant with minus the depth plus some value $v_{r'}$. The values $v_{r'}$ can be chosen independently of $r'$ since any two vertical rays come within bounded distance of each other and horofunctions are coarsely Lipschitz. This easily implies the claim that $\horba$ is coarsely the sub-level set of a horofunction, as required.
\end{proof}

Since $\widehat{G}$ is the deck group of the regular cover $\wchY^{D_1} \to Z_{D_1}$ and $G$ is the deck group of the intermediate (regular) cover $\wchXz^{D_1} \to Z_{D_1}$, there is a natural quotient map $\Omega :\widehat{G} \to G$ whose kernel is canonically identified with $\pi_1\left(\wchXz^{D_1}\right)$.

We have found a relatively hyperbolic group $\widehat G$ mapping onto $G$.  To show it is a drilling, we still need to show that the kernel of $\widehat G\to G$ is normally generated by a subgroup of $P$, and other related properties.

Note that each completed shell in $\Xz$ gives a conjugacy class of subgroups of $\pi_1\left(\wchXz^{D_1}\right)$. We refer to any such subgroup as a \emph{shell group}.   By Corollary~\ref{cor:pi_1_Z} (and the definition of coarse fundamental group) each shell group is isomorphic to the integers.

\begin{lemma}\label{lem:zbyz}
The peripheral subgroup $P$ is isomorphic to the $\Z$--by--$\Z$ group $E$ from Lemma \ref{lem:pi_CS/g}. Moreover, the intersection of $P$ with the kernel of the quotient $\widehat{G}\to G$ is a shell group $N$, and the image of $P$ in $G$ under the quotient map is a conjugate of $\langle g\rangle$.
\end{lemma}

\begin{proof}
This follows from the proof of Lemma \ref{lem:pi_CS/g} and from  restricting the covers $\wchY^{D_1} \to \wchXz^{D_1}$ and $\wchXz^{D_1} \rightarrow Z_{D_1}$ to the horosphere corresponding to $\horba$, which is naturally identified with $\widetilde{\SCS}$. We have the following commutative diagram, where the vertical arrows are covers, and the horizontal arrows are inclusions.  Furthermore the deck group actions on the left extend to the deck group actions on the right. 

\begin{equation}
\begin{tikzcd}
   \widetilde{\SCS^{D_1}}\arrow[r,hook]\arrow[d]\arrow[loop left, "\langle a \rangle"] & \wchY^{D_1} \arrow[d] \\
   \SCS^{D_1}\arrow[r,hook]\arrow[d]\arrow[loop left, "\langle g \rangle"] & \wchXz^{D_1}\arrow[d]\\
 \leftQ{\SCS^{D_1}}{\langle g \rangle}\arrow[r,hook] &  Z_{D_1}
\end{tikzcd}
\end{equation}

By Lemma \ref{lem:pi_CS/g}, the group $E$ describes the deck group of the composition of the covers on the left. 
\end{proof}

We now show that $\pi_1(\wchXz^{D_1})$ is generated by the shell groups, and that these are all conjugate in $\widehat{G}$. 

\begin{lemma}\label{lem:genshell}
 $\pi_1\left(\wchXz^{D_1}\right)$ is generated by shell groups.
\end{lemma}

\begin{proof}
Let $$\mathrm{CS} = \bigcup\limits_{h\in G\smallsetminus \langle g \rangle} \CS^{\Xz}(h\cdot \gammax),$$ and let
$$ \mathrm{CT} = \mathrm{CS} \ \bigcup  \left( \bigcup\limits_{h\in G \smallsetminus\langle g \rangle} T_{R_0}(h\cdot \gammax) \right).$$

We claim first that
the space $\Xz \cup \mathrm{CS}$ is $16\delta_0$--simply connected.  Indeed, let $\alpha$ be a loop in this space. By the definition of the completed shells, $\alpha$ may be homotoped to a loop in $\Xz$ across disks of circumference at most $2s = 16\delta_0$.  The space $\Xz$ is $\delta_0$ hyperbolic, so by Lemma~\ref{lem:hyp_simply_conn} it is $16\delta_0$--simply-connected.  The claim follows.

Next we claim that the space 
\begin{equation}\label{eq:vankampen}\wchXz^{D_1} \cup_{\mathrm{CS}^{D_1}} \mathrm{CT}^{D_1}
\end{equation}
is simply connected.
Let $\beta$ be a loop in the $1$--skeleton of this space.  
This loop is also a loop in the $1$--skeleton of $\Xz\cup \mathrm{CS}$.  Since that space is  $16\delta_0$--simply connected, there is a $16\delta_0$--filling of $\beta$.  We now subdivide the $2$--cells of this filling to obtain a ${D_1}$--filling of $\beta$ in the space from \eqref{eq:vankampen}.  Let $\alpha$ be one of the $2$--cell boundaries of the filling.  If it lies completely in $\wchXz$ or completely in $\mathrm{CT}$ we do not need to modify it.  Otherwise, we consider a maximal sub-segment lying in some component of $\mathrm{CT}$.  The endpoints lie in a completed shell, and are distance $\le 16\delta_0$ from one another, so they may be connected by a path in that completed shell of length at most $\Phi(16\delta_0)$, by Lemma \ref{lem:proper_dist}.  We can then replace our $2$--cell boundary by one inside a component of $\mathrm{CT}$ and one which has one fewer component not in $\wchXz$. This has length less that $16 \delta_0 + \Phi(16 \delta_0)$.  Repeating this process, we obtain a subdivision into loops of length at most ${D_1}$ inside a component of $\mathrm{CT}$ and a single loop of length at most
\[ 16\delta_0 + 16\delta_0 \Phi(16\delta_0) \]
in $\wchXz$.
This quantity is less than ${D_1}$ so the loop bounds a disk in $\Xz^{D_1}$.  Applying this subdivision process to all the loops in our original $16\delta_0$--filling gives a contraction in the space described in~\eqref{eq:vankampen}.

By Seifert--van Kampen the fundamental group of $\wchXz^{D_1}$ is generated by the fundamental groups of the components of $\mathrm{CS}^{D_1}$, in other words by the shell groups.
\end{proof}

\begin{lemma}
The shell groups are all conjugate in $\widehat{G}$. 
\end{lemma} 

\begin{proof}
This follows from the fact that there is only one $G$--orbit of shells in $\wchXz$. A specific conjugate of a shell group corresponds to a choice of path from a basepoint in $\wchXz$ to a shell, and given two such paths the corresponding paths give a loop in $Z_{D_1}$; this loop represents the conjugator in $\widehat{G}$ (thought of as $\pi_1(Z_{D_1})$) between the two shell groups.
\end{proof}

From the previous two lemmas,  $\pi_1(\wchXz^{D_1})$, considered as a subgroup of $\widehat{G}$, is normally generated by any single shell group.  Choose the shell group $N$ that lies in $P$.  Then $\widehat{G} / \llangle N \rrangle = G$. 

Thus, we have proved that $\widehat{G}$ is a drilling of $G$ along $g$ as required by Theorem \ref{t:Drill}.  

We must prove that $P$ is free abelian if and only if the action of $g$ on the boundary is orientation-preserving. 
This follows because by Lemma~\ref{lem:pi_CS/g} $P$ is the fundamental group of 
\[ \leftQ{\left( \partial \Xz \ssm \{\gammax^\pm\} \right)} {\langle g \rangle} \]
which is a torus, or a Klein bottle, according to whether the action of $g$ on $\partial \Xz \ssm \{\gammax^\pm\}$ is, or is not, orientation-preserving. 

To finish the proof of Theorem \ref{t:Drill} it remains to prove that the Bowditch boundary of $(\widehat{G},P)$ is a $2$--sphere, and also the final assertion about torsion (or lack thereof).

\section{Linearly connected limits of spheres}
\label{sec:limitsofspheres}
In this section we prove the following criterion for a weak Gromov-Hausdorff limit to be $S^2$:

\begin{theorem}\label{thm:limit_sphere}
 Suppose that $(M_i,d_i)$ are $L$--linearly connected metric spaces homeomorphic to $S^2$, and that the metric space $(M,d)$ is a weak Gromov-Hausdorff limit of the $M_i$. Suppose furthermore that:
 \begin{enumerate}
  \item $M$ is nondegenerate,
  \item $M$ is compact,
  \item $M$ has no cut-points,
  \item (Self-similarity) Given any $p\in M$ and any open set $U\subseteq M$, there exists a neighborhood $V$ of $p$ and an open embedding $V\to U$.
 \end{enumerate}
 Then $M$ is homeomorphic to $S^2$.
\end{theorem}

We start with a general lemma.

\begin{lemma}
\label{lem:lin_conn_lim}
Let $(M_i,d_i)$ be $L$--linearly connected spaces that weakly Gromov-Hausdorff converge to the compact metric space $M$. Then $M$ is linearly connected, connected,  and locally path-connected.
\end{lemma}

\begin{proof}
A linearly connected space is connected and locally path-connected.  Indeed, if $B$ is a ball of radius $\epsilon$ around a point $x$, then $L$--linear connectivity allows one to construct a path-connected $C \subset B$ containing the $\epsilon/L$--ball $B'$ around $x$.
This set $C$ is the union of the path-connected sets of controlled diameter containing the pairs $\{x,y\}$ for $y\in B'$.  Therefore, it suffices to prove that $M$ is linearly connected.

By assumption, there exist $\lambda$ and $\epsilon_i\to 0$ such that there exist $(\lambda,\epsilon_i)$--quasi-isometries $\psi_i:M_i\to M$ and $\phi_i: M\to M_i$ that are $\epsilon_i$--quasi-inverses of each other. Given any $p,q\in M$, consider connected sets $C_i\subseteq M_i$ of diameter at most $Ld_i(\phi_i(p),\phi_i(q))\leq L\lambda d(p,q)+L\epsilon_i$ containing $\phi_i(p),\phi_i(q)$. We can then consider the closures of $\psi_i(C_i)$, and consider a Hausdorff limit $D$ of a subsequence. It is easy to see that $D$ is connected, contains $p$ and $q$, and has diameter bounded linearly in terms of $d(p,q)$.
\end{proof}

The following two propositions are (essentially) proven in \cite{GMS}.

\begin{proposition}\label{prop:limit_planar}
 Suppose that $(M_i,d_i)$ are $L$--linearly connected metric spaces homeomorphic to $S^2$, and that the compact metric space $(M,d)$ is a weak Gromov-Hausdorff limit of the $M_i$. Either $M$ has a cut-point or it is planar.
\end{proposition}

\begin{proof}
By \cite[Lemma 3.9]{GMS}, non-planar graphs cannot topologically embed in $M$. Moreover, $M$ is connected and locally connected by Lemma \ref{lem:lin_conn_lim}.  Hence, $M$ is a Peano continuum that does not contain any non-planar graphs, and therefore by \cite{Claytor34} either it has a cut-point or it is planar.
\end{proof}

\begin{definition} \cite[Definitions 8.1, 8.2]{GMS}
Let $M$ be a metric space, and let $f \co S^1 \to M$.  For $\epsilon >0$, an \emph{$\epsilon$--filling} of $f$ is a triangulation of the unit disk $D^2$, along with a (not necessarily continuous) extension of $f$ to $\overline{f} \co D^2 \to M$ so that every simplex of the triangulation is mapped by $\overline{f}$ to a set of diameter at most $\epsilon$.

We say that $M$ is \emph{weakly simply connected} if, for every $\epsilon > 0$, every continuous $f \co S^1 \to M$ has an $\epsilon$--filling.
\end{definition}

The second result is an immediate consequence of \cite[Theorem 8.6]{GMS}.
\begin{proposition}\label{prop:simply_connected} 
Suppose that $(M_i,d_i)$ are simply connected $L$--linearly connected metric spaces, and that the compact metric space $(M,d)$ is a weak Gromov-Hausdorff limit of the $M_i$. Then $M$ is weakly simply connected.
\end{proposition}

Observe that if a compact metric space $M$ is weakly simply connected, and if $M'$ is a metric space homeomorphic to $M$, then $M'$ is also weakly simply connected. Thus we can speak about weakly simply connected compact spaces.

\begin{lemma}\label{lem:non-empty_interior}
Let $M$ be a subspace of $S^2$ (as topological spaces). Suppose that:
\begin{enumerate} 
 \item $M$ is nondegenerate,
 \item $M$ is compact,
 \item \label{itm:ncp} $M$ has no cut points, 
 \item \label{itm:clpc} $M$ is connected and locally path connected, 
 \item $M$ is weakly simply connected.
 \end{enumerate} 
 Then $M$ contains an open subset homeomorphic to $\R^2$.
\end{lemma}

 \begin{proof}
We note first that~\eqref{itm:clpc} implies that $M$ is path connected and that \eqref{itm:clpc} and \eqref{itm:ncp} together imply that $M$ has no path cut point.
 
 We consider the standard round metric $d_{S^2}$ on $S^2$ (that is, of diameter $\pi$), and the induced metric on $M$.

  Since $M$ is nondegenerate, there are distinct points $x$ and $y$ in $M$. Since it is path connected, there is a path $p$ in $M$ between $x$ and $y$. Take a third point $z$ on the path. Since $z$ is not a path cut point, there is another path $p'$ in $M$ between $x$ and $y$.  The union of $p$ and $p'$ contains an embedded circle $C$. 

Jordan's theorem ensures that $C$ cuts $S^2$ into two disks $D_1$ and $D_2$.  We claim that $M$ contains at least one of the disks, which proves that $M$ contains an open subset homeomorphic to $\R^2$, as required.

Otherwise, there are points $p_1 \in D_1\ssm M$, $p_2\in D_2\ssm M$.  Since $M$ is compact, the distance between $\{ p_1, p_2 \}$ and $M$ is realized.  Choose a positive $\epsilon$ so that
\[ \epsilon<\min\left\{\pi/2,d_{S^2}(M,\{p_1,p_2\})/2\right\}.\] 

Let $f \co S^1 \to M$ be so that $f(S^1) = C$.  Since $M$ is weakly simply connected, there is an $\epsilon$--filling $\overline{f} \co D^2 \to M$ of $f$. 
Since $\epsilon<\pi/2$, the $\overline{f}$--images of every two adjacent vertices in the triangulation of $D^2$ can be connected by a geodesic segment in $S^2$ of length at most $\epsilon$.  The resulting triangles can be filled in without increasing their diameter.  So there exists a continuous map $g\co D^2 \to S^2$ such that $g|_{S^1}=f$ and $g(D^2)$ lies in the $2\epsilon$--neighbourhood of $M$.  But then $g(D^2)$ must contain one of the two disks $D_1$ or $D_2$, and hence either $p_1$ or $p_2$, a contradiction.
\end{proof}

We are ready to prove the theorem.

\begin{proof}[Proof of Theorem \ref{thm:limit_sphere}]
  By Lemma~\ref{lem:lin_conn_lim}, $M$ is connected and locally path-connected.  

 By Proposition \ref{prop:limit_planar}, $M$ embeds in $S^2$, since it does not have cut-points. Moreover, it is weakly simply connected by Proposition \ref{prop:simply_connected}. But then Lemma \ref{lem:non-empty_interior} applies, so $M$ contains an open subset homeomorphic to $\R^2$. By self-similarity, each point of $M$ has an open neighborhood homeomorphic to $\R^2$. Hence, any embedding $M\to S^2$ is open. Such an embedding is also closed (since $M$ is compact and $S^2$ is Hausdorff), so it is a homeomorphism.
\end{proof}

\section{Proof of Theorem \ref{t:Drill}}
\label{sec:proof_of_main_theorem}

As noted at the end of Section \ref{s:unwrap family}, it remains to show that the Bowditch boundary of $(\widehat{G},P)$ is a $2$--sphere, and the statement about torsion.  We retain the notation from the previous sections.

By Theorem \ref{GPrelhyp} we have to show that $\partial \widehat{Y}$ is homeomorphic to the 2-sphere.

By Lemma \ref{lem:Y_j-converge}, the spaces $(\widehat{Y}_i,y_i)$ strongly converge to $(\widehat{Y},\widecheck{y})$.  By the construction of the $\widehat{Y}_i$, they are all $\delta_2$--hyperbolic and $\delta_2$--visible, so by Proposition~\ref{prop:GMS3.5} the spaces $\partial\widehat{Y}_i$ weakly Gromov--Hausdorff converge to $\partial\widehat{Y}$ (when all are equipped with $\delta_2$--adapted visual metrics at the given basepoints).

It is clear that $\partial\widehat{Y}$ is infinite (since the parabolic subgroup has infinite index, as can be seen in the quotient $G$ of $\widehat{G}$) and compact.  Moreover it is self-similar (as in Theorem \ref{thm:limit_sphere}), since it is the boundary of a relatively hyperbolic group with infinite index peripheral subgroups.  Therefore, the endpoints of loxodromic elements are dense in $\partial\widehat{Y}$.   Since each of the $\widehat{Y}_i$ is $\delta_2$--hyperbolic and $\Sigma_0$--modeled on $\left\{ \UG(\Xz,\gammax)\right\}$, by Corollary \ref{cor:modeled_on_hatX_LC} there is an $L$ such that the spaces $\partial\widehat{Y}_i$ are all $L$--linearly connected. We now show that each of the $\partial\widehat{Y}_i$ is homeomorphic to $S^2$. 

\begin{proposition}\label{prop:unwrap_S^2}
Each of $\partial\widehat{Y}_i$ is homeomorphic to $S^2$. 
\end{proposition}

\begin{proof}
We prove this by induction on $i$.  In the base case $i=1$ we have $\widehat{Y}_1 = \UG(\Xz,\gammax)$, whose boundary is homeomorphic to $S^2$ by Corollary~\ref{cor:UG0_S2}.

Suppose now that $i \ge 1$ and that $\partial \widehat{Y}_i$ is homeomorphic to $S^2$.  
Then $\widehat{Y}_{i+1}= \widehat{Y}_{i}$ (and then $\partial \widehat{Y}_{i+1}$ is homeomorphic to $S^2$) or $\widehat{Y}_{i+1}=\UG(\widehat{Y}_i,c_i)$, for a core $c_i$ of a tube in $\widehat{Y}_i$.

Notice that $\partial \widehat{Y}_i\ssm\{c_i^{\pm}\}$ is an (open) annulus. By Corollary~\ref{cor:modeled_on_hatX_LC} we know that $\partial \UG(\widehat{Y}_i,c_i)$ is uniformly linearly connected (and, in particular, connected). Hence, by Proposition \ref{prop:cut point implies not uniformly linearly connected}, $\partial \UG(\widehat{Y}_i,c_i)$ does not have cut-points, so $\partial \UG(\widehat{Y}_i,c_i) \ssm \{ p_i \}$ is connected, where $p_i$ is the point at infinity of the new horoball in $\UG(\widehat{Y}_i,c_i)$.

In view of Proposition~\ref{prop:deck_Z} we have that $\partial \UG(\widehat{Y}_i,c_i)\ssm \{p_i\}$ is a connected $\Z$--cover of an annulus, and therefore it is homeomorphic to $\mathbb R^2$. As in the  proof of Corollary \ref{cor:UG0_S2}, $\partial \UG(\widehat{Y}_i,c_i)$ is a compact metric space, and removing a point makes it homeomorphic to $\mathbb R^2$.  So $\partial \UG(\widehat{Y}_i,c_i)$ is homeomorphic to $S^2$, and the induction is complete.
\end{proof}

Proposition \ref{prop:no cut points} now implies that $\partial\widehat{Y}$ has no cut-points. By Lemma \ref{lem:lin_conn_lim}, $\partial \widehat{Y}$ is locally connected. It now follows from Theorem \ref{thm:limit_sphere} that $\partial\widehat{Y}$ is homeomorphic to $S^2$, as required.  

Finally, we consider the issue of torsion in $G$ and $\widehat{G}$.  The following statement immediately implies the final assertion of Theorem~\ref{t:Drill} (and in fact proves somewhat more).

\begin{proposition} \label{prop:torsion}
Let $\Omega : \widehat{G} \to G$ be the quotient map described in Section~\ref{s:unwrap family}.  Then $\Omega$ is injective on any finite subgroup of $\widehat{G}$.  Conversely, if $F$ is a finite subgroup of $G$ then there exists a finite subgroup $\widehat{F} \subset \Omega^{-1}(F)$ so that $\Omega|_{\widehat{F}} : \widehat{F} \to F$ is an isomorphism. 
\end{proposition}
\begin{proof}
    Let $\widehat{F}$ be a finite subgroup of $\widehat{G}$.  Then $\widehat{F}$ acts isometrically on the $\delta_2$--hyperbolic space $\widehat{Y}$.  By \cite[Lemma 4.2]{GM-splittings} there is an $\widehat{F}$--orbit $\widehat{F} \cdot \mathfrak{o}$ of diameter at most $4\delta_2+2$ in $\widehat{Y}$.

    Suppose that $\widehat{F}$ is nontrivial.  Then it cannot be contained in a stabilizer of a horoball in $\widehat{Y}$, by Lemma~\ref{lem:zbyz}.  However, the horoballs in $\widehat{Y}$ are $\Sigma_1$--separated, so if $\widehat{F}\cdot \mathfrak{o}$ lies within $20\delta_2$ of a horoball, $\widehat{F}$ must in fact stabilize that horoball.  It follows that $\widehat{F} \cdot \mathfrak{o}$ lie in a ball $B$ in $\widehat{Y}$ which does not lie within $15\delta_2$ of any horoball.

    By the definition of $\widehat{Y}$ as a strong limit of the $\widehat{Y}_j$, and by induction using Lemma~\ref{lem:isom_of_balls}, the ball $B$ maps injectively to a ball $\overline{B}$ in $\wchXz$ (since $B$ is not close to any horoball).

    Recall that $\widehat{G} = \pi_1\left(Z_{D_1}\right)$, and that $G$ corresponds to the deck group of the cover $\wchXz^{D_1} \to Z_{D_1}$ which is intermediate to the universal cover $\wchY^{D_1} \to Z_{D_1}$. The nontrivial elements of $\widehat{F}$ are represented by loops in $Z_{D_1}$.  These loops elevate to paths in $\wchXz$ which are not loops because $B$ maps injectively to the ball $\overline{B}$ in $\wchXz$. Thus, $\Omega$ is injective on $\widehat{F}$.

    Conversely, consider a finite subgroup $F$ of $G$.  Then $F$ has an orbit ${F} \cdot \mathfrak{o}$ of diameter at most $4\delta_0+2$ in $\Xz$.  If $F$ is nontrivial the assumption that translates of $\gammax$ are $\Sigma$--separated forces a ball around $\widehat{F} \cdot \mathfrak{o}$ to lie in $\wchXz$.  This ball lifts to a ball $B$ in $\wchY$, since this is a $D_1$--cover, and $D_1 \gg 4\delta_0+2$.  The lift of $\widehat{F} \cdot \mathfrak{o}$ in $B$ defines a collection of deck transformations in $\widehat{G}$ which forms a section of $\Omega$ on $F$ and also forms a subgroup of $\widehat{G}$, completing the proof.
\end{proof}

The proof of Theorem \ref{t:Drill} is now complete.

\section{The Cohen--Lyndon property}\label{sec:CL}
In this section we argue that the map from $(\widehat{G},P)$ to $G$ is a ``long filling'' in a reasonable sense.  Specifically, it satisfies the Cohen--Lyndon property, defined by Sun in \cite[Definition 2.2]{Sun20}
\begin{definition}
    A triple of groups $(A,B,C)$ with $C\lhd B < A$ satisfies the \emph{Cohen--Lyndon property} if there is a left transversal $T$ of $B\llangle C\rrangle$ in $A$ so that 
    \begin{equation*}
        \llangle C\rrangle = \bigast\limits_{t\in T} t C t^{-1}.
    \end{equation*}
\end{definition}
(To be more precise, the statement is that the natural map from the group on the right-hand side to $\llangle C\rrangle$ is an isomorphism.)

 Dahmani--Guirardel--Osin \cite[Theorem 2.27]{DGO} prove that the kernel of a long Dehn filling is a free product of conjugates of filling kernels, but do not give an explicit indexing set.  
 Sun clarifies that the triple $(A,B,C)$ satisfies the Cohen--Lyndon property whenever $(A,B)$ is relatively hyperbolic and $C$ is a subgroup of $B$ which is ``sufficiently long" in the sense it misses finitely many (short) elements.  (Both Sun and Dahmani--Guirardel--Osin work in a more general setting, where $B$ is (weakly) hyperbolically embedded in $A$.)
 In the relatively hyperbolic case this was proved in \cite[Theorem 4.8]{GMS}.  In fact the proof there yields the following quantitative statement:
\begin{theorem}\label{thm:cohenlyndon}
    Suppose that $(A,B)$ is a relatively hyperbolic group pair, with a geometrically finite action on a $\theta$--hyperbolic space $\Upsilon$, for some $\theta\geq1$, containing an equivariant family of convex horoballs which are $10^3\theta$--separated.  Suppose also (\emph{Very Translating Condition}) that $C\lhd B$ is chosen so that if $H$ is the horoball stabilized by $B$, then $d_\Upsilon(x,n \cdot x) \ge 10^4\theta$ for every $n\in C\ssm \{1\}$ and $x\in \Upsilon\ssm H$.

    Then $(A,B,C)$ satisfies the Cohen--Lyndon property.
\end{theorem}
Two remarks are in order.  First, the Cohen-Lyndon property and Theorem \ref{thm:cohenlyndon} can be reformulated for triples $(A,\mc{B},\mc{C})$ where $\mc{B}$ is a family of subgroups and $\mc{C}$ a family of subgroups of the form $C_B\lhd B$ for $B\in \mc{B}$ (see \cite[Definition 3.13]{Sun20}).  All the theorems mentioned above are known in this generality.
But for the current purposes it is sufficient to consider a single subgroup.

Secondly, there is a misleading comment in \cite{GMS} near the beginning of Section 4.  It is there claimed \cite[Lemma 4.1]{GMS} that the Very Translating Condition is a \emph{consequence} of being a sufficiently long filling, and the relatively hyperbolic Dehn filling theorem is referenced.  In fact,  the Very Translating Condition is only about the filling kernels, so it is satisfied whenever those filling kernels contain no short nontrivial elements of the parabolics.  It is a consequence of the proof that the normal subgroup generated by the filling kernels also contains no such short elements.

We apply Theorem \ref{thm:cohenlyndon} to the setting of the current paper as follows.
\begin{theorem}\label{thm:longfilling}
    Let $(\widehat{G},P)$ be the relatively hyperbolic pair from Theorem \ref{GPrelhyp}, and let $N\lhd P$ be the shell group described in Lemma \ref{lem:zbyz}.
    The triple $(\widehat{G}, P, N)$ satisfies the Cohen-Lyndon property.
\end{theorem}
\begin{proof}
    We have built a $\delta_2$--hyperbolic space $\widehat{Y}$ on which $(\widehat{G},P)$ acts geometrically finitely, so that $P$ preserves a horoball $\mc{H}$.  The family of translates of $\mc{H}$ is $\Sigma_1$--separated, and $\Sigma_1$ is much larger than $10^3\delta_2$, but the horoballs $\mc{H}$ may not be convex.  We replace them by nested $2\delta_2$--horoballs $g\mc{H}' \subset g\mc{H}$, which \emph{are} convex (see \cite[Lemma 3.26]{rhds} -- the proof there applies verbatim).  These nested horoballs  are obviously still $10^3\delta_2$--separated.
    (Here we are using that $\Sigma_1\ge 10 R_0 \ge 10\cdot 2\sigma_0 \ge 20\cdot 10^5 D_1 \ge 2\cdot 10^6\cdot 2 Q_0 \ge 4\cdot 10^6 \delta_2>10^3\delta_2$.)

    Next we verify the Very Translating Condition.  Let $x$ be a point outside $\mc{H'}$ and let $n\in N\ssm \{1\}$.
    Let $y$ be a closest point to $x$ on $\partial H$.  

    Now we note that our systole requirement ($\sys_0 = 2^{25\sigma_0} Q_0$) together with Lemma~\ref{lem:horbadistort} gives
    \begin{align*}
        d_{\mc{H}}(y,n\cdot y) & \ge 2\log_2\sys_0 - 2 \\
        & \ge 20\cdot 10^5 D_1 \\
        & \ge 40 \cdot 10^5 Q_0 \\
        & \ge 4\cdot 10^6 \delta_2.
    \end{align*}
    It is not hard to see that if $d_{\mc{H}}(a,b) 
    \ge 10\delta_2$, then the $\widehat{Y}$--distance is actually realized in $\mc{H}$.  Moreover a simple argument gives $d_{\widehat{Y}}(x,n\cdot x)\ge d_{\widehat{Y}}(y,n\cdot y) - 8\delta_2$. 
    
    Thus the above estimate gives $d_{\widehat{Y}}(y,n\cdot y) \ge 4\cdot 10^6\delta_2$, so $d_{\widehat{Y}}(x,n\cdot x) \ge (4\cdot 10^6 - 8)\delta_2>10^4\delta_2$, thus verifying the Very Translating Condition.  
\end{proof}

\section{Notation index}\label{sec:notation}

In this section, we collect (some of) the notation from this paper, with links to where it is introduced, and some discussion about how it is used.  Hopefully this will help the reader.

The reader may also want to consult Notation~\ref{not:first_constants}, where many of our constants are fixed.

\begin{itemize}
\item $M$ is a general compact metric space.
    \item $\Gamma$ is reserved for a general connected graph.
    \item $Z$ denotes a general metric space.
    \item $\Upsilon$ denotes a general Gromov hyperbolic metric space (with constant of hyperbolicity always locally specified, often $\delta$).

\item $\Xz$ is the $\delta_0$--hyperbolic space we start with. The constant $\delta_0$ is reserved for this specific value. See Assumption~\ref{ass:X}. From  Section~\ref{s:unwrap family} on we consider a group $G$ acting geometrically on $\Xz$.

\item Inside $\Xz$ is a $\lambda_0$--quasi-convex subset $\Yx$, with a fixed basepoint $w_0 \in \Yx$ (see also~\ref{ass:X}).  

In Assumption~\ref{ass:X2}, we further assume that $\partial \Xz \cong S^2$, and that $\Yx$ is a quasi-geodesic axis $\gammax$ for a loxodromic element $g$.  It is $\langle g \rangle$ that we drill along in the proof of Theorem~\ref{t:Drill}.  We make the assumption about the separation of the translates of $\gammax$ at the beginning of Section~\ref{s:unwrap family}.

\item For $K > 0$ an integer, the \emph{completed $K$--shell} about $\Yx$ is obtained by taking all the points at distance $K$ from $\Yx$, and adding edges of length at most $8\delta_0$ to make a connected graph, denoted $\CS_K(\Yx)$.  See Definition~\ref{def:completed shell} and Convention~\ref{conv:CS}.

\item The \emph{completed tube complement}, $\CTC_K(\Yx)$ is obtained by removing the tube about $\Yx$ and adding the edges from the completed shell.  See Definition~\ref{def:CTC} and Convention~\ref{conv:CTC}.

\item $\Xz^{\mathrm{cusp}}(K)$ is obtained by gluing a combinatorial horoball $\horbaD$ onto the copy of $\CS_K(\Yx)$ in $\CTC_K(\Yx)$.  See Definition~\ref{def:cusping}.

\item In Notation~\ref{not:first_constants}, many constants are fixed, including the radius $R_0$ of the tubes that are considered from that point forwards.  Further separation constants are fixed at the beginning of Section~\ref{s:unwrap family}.

\item $\mc{S}$ is the completed $R_0$--shell about $\gammax$ in $\Xz$, and $\SCTC$ is the associated completed tube complement.
These spaces have canonical regular $\Z$--covers $\widetilde{\mc{S}}$ and $\widetilde{\SCTC}$ which are $D_1$--simply-connected (for a constant $D_1$ fixed in \ref{not:first_constants}).  See Notation~\ref{not:H}.

\item $\horba$ denotes the combinatorial horoball based on $\widetilde{\SCS}$.  See Notation~\ref{not:H}.

\item $(\Upsilon,\gamma)$ is some space which is $R$--tube comparable (see Definition~\ref{def:comparable}) to $(\Xz,\gammax)$.

\item $\UG(\Upsilon,\gamma)$ means: Unwrap a completed tube complement $\CTC^\Upsilon_{R_0}(\gamma)$ and glue on a copy of $\horba$.  See Construction~\ref{unwrap_one_tube}.

\item $\wchXz$ is obtained by removing an equivariant family of tubes from $\Xz$ and gluing on corresponding copies of $\CS_{R_0}(\gammax)$.  See Subsection~\ref{ss:intermediate}.

\item We form the space $\wchXz^{D_1}$ by attaching disks to edge-loops of length $D_1$, see Definition \ref{def:coarse_fund}.  Then the quotient space is $Z_{D_1} = \leftQ{\wchXz^{D_1}}{G}$.  This is analogous to the manifold exterior in the 3--manifold case. 

\item The space $\wchY$ is a limit of certain covers of $\wchXz$. See Lemma~\ref{lem:wchY}.
\item The space $\widehat Y$ is the space $\wchY$ with an equivariant family of horoballs added. See Definition~\ref{def:hatY}.

In Section~\ref{s:unwrap family} we construct the group $\widehat{G}$ in the conclusion of Theorem~\ref{t:Drill} and prove in that the space $\widehat{Y}$ admits a geometrically finite $\widehat{G}$--action which exhibits the relative hyperbolicity of our drilled group.
\end{itemize}

\end{document}